\documentclass[reqno,12pt,letterpaper]{amsart}
\usepackage[proof]{zlmacros}
\usepackage{soul}

\title{2D internal waves in an ergodic setting}
\author{Zhenhao Li}
\email{zhenhao@mit.edu}
\address{Department of Mathematics, Massachusetts Institute of Technology, Cambridge, MA 02139}

\begin{document}

\maketitle

\begin{abstract}
    We study a model of internal waves under periodic forcing in an effectively 2-dimensional aquarium. When the underlying classical dynamics has sufficiently irrational rotation number, we prove that the solution to the internal waves equation remains bounded in energy space in time. This result is in contrast with the works by
    Colin de Verdi\`ere--Saint-Raymond and
    Dyatlov--Wang--Zworski, which gave a description of singular profiles when the underlying dynamics has periodic attractors. Our result is proved by showing the spectral measure near the forcing frequency is very small. 
\end{abstract}

\section{Introduction}
Below the wind-mixed surface layer of the ocean and an interfacial layer of seasonal pycnocline where the density increases rapidly, the density field of the ocean stabilizes and increases gradually with depth. To first order in this region, the density is considered to be stable-stratified, i.e. the density depends only on the vertical direction and increases with depth. Internal waves describe linear perturbations of such stable-stratified fluids, and are therefore a central topic in oceanography. These perturbations occur naturally and can arise mechanically or thermodynamically. For a more complete introduction to the physics behind internal waves, see Mass \cite{Maas_survey} and Sibgatullin--Ermanyuk \cite{SE_survey}.

We consider a model of a stratified fluid in an effectively 2D aquarium. The formation of internal waves is guided by the classical dynamics underlying the relevant wave equation. Recently, it was proven by Colin de Verdi\`ere--Saint-Raymond \cite{Verdiere_Raymond_20} and Dyatlov--Wang--Zworski \cite{DWZ} that when the underlying dynamics has hyperbolic attractors, there can be high concentration of the fluid velocity near these attractors. This phenomenon was predicted in the physics literature by Maas--Lam \cite{maas_lam_95} in 1995, and has since been experimentally observed by Maas et al. \cite{maas_observation_97}, Hazewinkel et al. \cite{Hazewinkel_2010}, and Brouzet \cite{brouzet_16}.

The behavior in the absence of attractors is expected to be dramatically different. In this paper we consider what happens in an ergodic case. We find that when the underlying classical dynamics has a sufficiently irrational rotation number, the internal waves have uniformly bounded energy for all times, see Theorem \ref{thm:evolution}. 

The domain of our 2D aquarium is denoted by $\Omega \subset \R^2 = \{x = (x_1, x_2) \mid x_j \in \R\}$. Under time-periodic forcing with profile $f \in \CIc(\Omega;\mathbb R)$, the stream function $u$ can be modeled by solutions of the following Poincar\'e problem:
\begin{equation}\label{eq:internal_waves}
    (\partial_t^2 \Delta + \partial_{x_2}^2) u = f(x) \cos \lambda_0 t, \quad u|_{t = 0} = \partial_t u|_{t = 0} = 0, \quad u|_{\partial \Omega } = 0,
\end{equation}
where $\lambda_0 \in (0, 1)$ and $\Delta := \partial_{x_1}^2 + \partial_{x_2}^2$. This equation is derived under the assumption that we are taking small linear perturbations of a uniformly stratified, invicid, hydrostatic, and nonrotating two-dimensional Boussinesq fluid. Here, we use the convention that the fluid velocity is given by $\mathbf v = (\partial_{x_2} u, -\partial_{x_1} u)$. See for instance Dauxois et al. \cite{Dauxois_18} for recent investigation of the equation in the physics literature.


It is helpful to interpret~\eqref{eq:internal_waves} as an evolution problem. For that, let $\Delta_\Omega$ denote the Dirichlet Laplacian, which has inverse $\Delta_\Omega^{-1}: H^{-1}(\Omega) \to H^1_0(\Omega)$, and define
\begin{equation}\label{eq:P_def}
    P := \partial_{x_2}^2 \Delta_\Omega^{-1}: H^{-1}(\Omega) \to H^{-1} (\Omega), \quad \langle u, w \rangle_{H^{-1}(\Omega)} := \langle \nabla \Delta_\Omega^{-1} u, \nabla \Delta_\Omega^{-1} w \rangle_{L^2(\Omega)}.
\end{equation}
The operator~$P$ is non-negative, bounded, and self-adjoint, and $\mathrm{Spec}(P) = [0, 1]$, see Ralston \cite{Ralston_73} and~\cite[\S7.1]{DWZ}. Then (\ref{eq:internal_waves}) becomes 
\begin{equation}\label{eq:evolution_problem}
    (\partial_t^2 + P) w = f \cos \lambda_0 t, \quad w|_{t = 0} = \partial_t w|_{t = 0} = 0, \quad f \in \CIc(\Omega; \R), \quad u = \Delta_\Omega^{-1} w. 
\end{equation}
Using Duhamel's formula, the solution $w(t)$ is then given by
\begin{equation}\label{eq:functional_solution}
    \begin{gathered}
        w(t) = \Re\big(e^{i \lambda_0 t} \mathbf W_{t, \lambda_0}(P) f \big) \qquad \text{where} \\
        \mathbf W_{t, \lambda_0}(z) = \int_0^t \frac{\sin(s \sqrt{z})}{\sqrt{z}} e^{-i\lambda_0 s} \, ds = \sum_{\pm} \frac{1 - e^{-it(\lambda_0 \pm \sqrt{z})}}{2 \sqrt{z} (\sqrt{z} \pm \lambda_0)}
    \end{gathered}
\end{equation}
For each $t$, $\Re\big( e^{i \lambda_0 t} \mathbf W_{t, \lambda_0}(z) \big)$ is a bounded smooth function on $[0, 1]$, so $w(t) \in H^{-1}(\Omega)$ is well-defined by the functional calculus of $P$ for each $t$. However, for $\sqrt{z} - \lambda_0 \lesssim t^{-1}$, we have $|\mathbf W_{t, \lambda_0}(z)| \sim t$, and for $\lambda_0 \in (0, 1)$, $\mathbf W_{t, \lambda_0} \to (z - \lambda_0^2 + i0)^{-1}$ as $t \to \infty$ in the sense of distributions. Consequently the behavior of $w(t)$ as $t \to \infty$
is related to spectral properties of $P$ near $\lambda_0^2$. 

\subsection{Assumptions on \texorpdfstring{$\Omega$}{} and \texorpdfstring{$\lambda_0$}{}}
Throughout the paper, we assume that $\Omega \subset \R^2$ is open, bounded, simply connected, and has $C^\infty$ boundary $\partial \Omega$. It turns out that the relevant underlying classical dynamics here is given by the reflected bicharacteristic flow for $(1 - \lambda_0^2) \xi_2^2 - \lambda_0^2 \xi_1^2$. Note that this is the Hamiltonian for the $1 + 1$ wave operator $P(\lambda_0)$ where
\[P(\lambda):= (1 - \lambda^2) \partial_{x_2}^2 - \lambda^2 \partial_{x_1}^2 = (P - \lambda^2) \Delta_\Omega, \quad \lambda \in (0, 1).\]
Since we want spectral information about $P$ near $\lambda_0^2$, we can already see that we need to study $P(\lambda)$ for $\lambda$ near $\lambda_0$. 

We can reduce the Hamiltonian flow of $P(\lambda)$ to the boundary by simply keeping track of where the trajectory hits the boundary. More formally, we have the following factorization of the quadratic form dual to $(1 - \lambda^2) \xi_2^2 - \lambda^2 \xi_1^2$:
\begin{equation}\label{eq:dual_factor}
    -\frac{x_1^2}{\lambda^2} + \frac{x_2^2}{1 - \lambda^2} = \ell^+(x, \lambda) \ell^-(x, \lambda), \quad \ell^\pm(x, \lambda) := \pm \frac{x_1}{\lambda} + \frac{x_2}{\sqrt{1 - \lambda^2}}.
\end{equation}
We will often suppress the dependence on $\lambda$ in the notation and simply write $\ell^\pm(x)$ when it is unambiguous. The level lines of $\ell^\pm$ are precisely the bicharacteristic lines of the Hamiltonian $(1 - \lambda^2) \xi_2^2 - \lambda^2 \xi_1^2$. We will make the following assumption on $\Omega$ to avoid issues with glancing and regularity in the reduction to boundary:
\begin{definition}\label{lambda_simple_def}
    Let $\lambda \in (0, 1)$. We say that $\Omega$ is \textup{$\lambda$-simple} if each of the functions $\ell^\pm(\bullet, \lambda): \partial \Omega \to \R$ has exactly two distinct non-degenerate critical points. These critical points will be denoted
    \begin{equation}\label{eq:characteristic_set}
        \begin{gathered}
            x_{\mathrm{max}}^\pm = \mathrm{arg\, max}_{x \in \partial \Omega} \, \ell^\pm(x, \lambda) \\
            x_{\mathrm{min}}^\pm = \mathrm{arg\, min}_{x \in \partial \Omega} \, \ell^\pm(x, \lambda).
        \end{gathered}
    \end{equation}
    See Figure \ref{fig:defs}.
\end{definition}
\begin{figure}
    \centering
    \includegraphics[scale = .25]{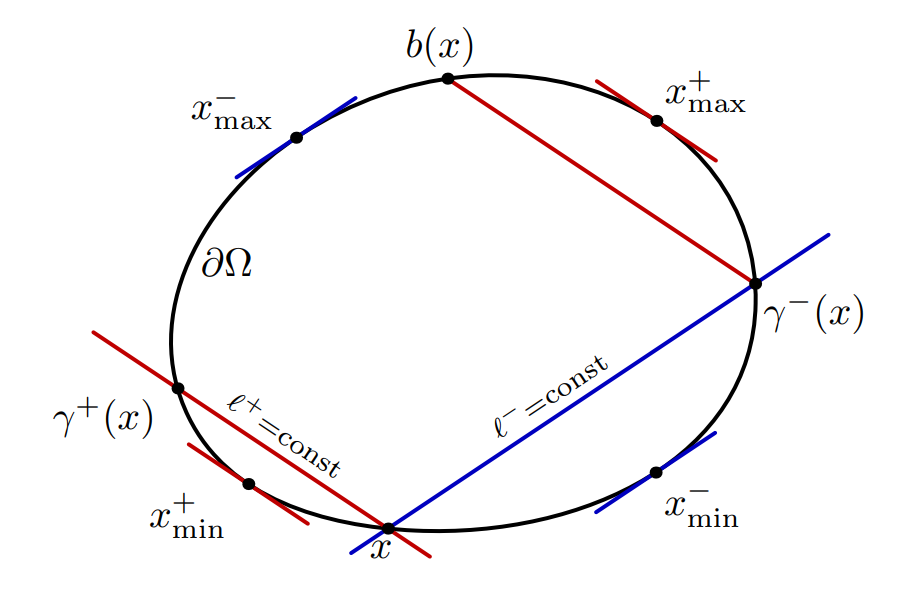}
    \caption{Definition of the involutions $\gamma^\pm$ as well as the chess billiard map $b(x)$. Also illustrates the definition of $\lambda$-simplicity. The diagram is from \cite{DWZ}, which considered the same dynamical system.}
    \label{fig:defs}
\end{figure}
Assuming that $\Omega$ is $\lambda$-simple, there exist unique smooth orientation reversing involutions $\gamma^\pm(\bullet, \lambda): \partial \Omega \to \partial \Omega$ that satisfy
\begin{equation}\label{eq:gamma_def}
    \ell^\pm(x) = \ell^\pm(\gamma^\pm(x)).
\end{equation}
In other words, under the $\lambda$-simplicity condition, each level line of $\ell^\pm$ intersects the boundary exactly twice unless the level line passes through $x^\pm_{\mathrm{min}}$ or $x^\pm_{\mathrm{max}}$. Then $\gamma^\pm$ exchanges the two points of intersection and fixes $x^\pm_{\mathrm{min}}$ and $x^\pm_{\mathrm{max}}$. Now we can define the \textit{chess billiard map} $b(\bullet, \lambda): \partial \Omega \to \partial \Omega$ by the composition
\begin{equation}\label{eq:b_def}
    b:= \gamma^+ \circ \gamma^-.
\end{equation}
$b$ is a smooth orientation preserving diffeomorphism. See Figure \ref{fig:defs}.

Finally, before stating our results, we need a dynamical assumption on $b$. Fix a positively oriented parameterization on $\partial \Omega$ given by $\mathbf x: \mathbb S^1 = \R/\Z \to \partial \Omega$. Let $\tilde \pi: \R \to S^1$ be the covering map given by $\tilde \pi(x) = x \mod 1$. Thus we have a covering map $\pi: \R \to \partial \Omega$ by $\pi = \mathbf x \circ \tilde \pi$. Fix a lift $\mathbf b(\bullet, \lambda) : \R \to \R$ of $b(\bullet, \lambda)$, i.e. $\mathbf b(\bullet, \lambda)$ satisfies
\[\pi(\mathbf b(\theta, \lambda)) = b(\pi(\theta, \lambda))\]
for all $\theta \in \R$. Fix an initial point $\theta_0 \in \R$. The \textit{rotation number} of $b( \bullet, \lambda)$ is then defined as
\begin{equation}\label{eq:rot_num_def}
    \mathbf r(\lambda) := \lim_{k \to \infty} \frac{\mathbf b^k(\theta_0, \lambda) - \theta_0}{k} \mod \Z \in \R/\Z.
\end{equation}
See \S\ref{sec:dynamical_prelim} for more details on the chess billiard map. 

We wish to study the behavior of internal waves in the scenario that the underlying chess billiard map is sufficiently chaotic. This is formulated using a ``sufficiently irrational" condition on the rotation number of $b$.
\begin{definition}\label{def:diophantine}
    A number $r \in \R$ is called \textup{Diophantine} if there exists consistants $c, \, \beta > 0$ such that 
    \begin{equation*}
        \left|r - \frac{p}{q} \right| \ge \frac{c}{q^{2 + \beta}}
    \end{equation*}
    for all $p \in \Z$ and $q \in \N$. 
\end{definition}
Roughly speaking, $r$ is Diophantine if it is sufficiently far away from any rational numbers. It is easy to see that Diophantine numbers exist and form a full measure set. Indeed, take some $\beta > 0$ and a small $c > 0$. Then the set 
\[E = \left \{r \in (0, 1) : \left|r - \frac{p}{q} \right| \ge \frac{c}{q^{2 + \beta}} \, \text{ for all $p \in \Z$ and $q \in \N$} \right\}\]
has measure bounded by 
\[|E| \ge 1 - \sum_{q = 1}^\infty 2q \cdot \frac{c}{q^{2 + \beta}} > 0\]
for all sufficiently small $c$.

\subsection{Statements of results}
Let $\mu_{f,f}$ be the spectral measure of $P$
with respect to $f \in \CIc(\Omega;\mathbb R)$,
\[\langle \chi(P)f, f \rangle_{H^{-1}(\Omega)} = \int_{\R} \chi \, d\mu_{f, f},\]
where $\chi$ is any bounded Borel measurable function on~$\mathbb R$. Our first result is

\begin{theorem}\label{thm:spectral}
    Let $\lambda_0$ be such that $\Omega$ is $\lambda_0$-simple (see Definition \ref{lambda_simple_def}) and $\mathbf r(\lambda_0)$ is Diophantine. Then for all $N > 0$ and $f \in \CIc(\Omega)$, there exists a constant $C = C(\Omega, \lambda_0, N, f)$ such that the spectral measure of the operator $P$ defined in (\ref{eq:P_def}) satisfies  
    \begin{equation}
        \mu_{f, f} \big((\lambda_0^2 - \epsilon, \lambda_0^2 + \epsilon) \big) \le C \epsilon^{N}.
    \end{equation}
\end{theorem}
In particular, this implies that $\lambda_0^2$ cannot be an eigenvalue of~$P$. Theorem~\ref{thm:spectral} has the following striking consequence for the Poincar\'e evolution problem:
\begin{theorem}\label{thm:evolution}
    Let $\lambda_0 \in (0, 1)$ be such that $\Omega$ is $\lambda_0$-simple and $\mathbf r(\lambda_0)$ is Diophantine. Assume that $f \in \CIc(\Omega; \R)$. Then the solution to \eqref{eq:internal_waves} remains bounded in energy space for all times, i.e. there exists a constant $C > 0$ such that
    \begin{equation}
        \|u(t)\|_{H^1(\Omega)} \le C
    \end{equation}
    for all $t \in \R$. 
\end{theorem}

The requirement that $\mathbf r(\lambda_0)$ is Diophantine is not too restrictive. There are many examples that satisfy this condition. As mentioned after Definition~\ref{def:diophantine}, Diophantine numbers form a full measure set. If $\Omega$ is a disk with a circular boundary, the rotation number function $\mathbf r(\lambda)$ is a smooth function of $\lambda$ with a strictly positive derivative, which means that the set of $\lambda_0$ that satisfies the hypothesis of Theorems~\ref{thm:spectral} and~\ref{thm:evolution} is a full measure subset of $(0, 1)$, see Figure~\ref{fig:rot_num}. However, for other domains, for instance the tilted square, the set of $\lambda_0$ for which $\mathbf r(\lambda_0)$ is rational appears to form a full measure set. Nevertheless, we still have uncountably many values of $\lambda_0$ that satisfy the hypothesis of our theorems since the rotation number is always nondecreasing. Although a square domain has corners and therefore does not satisfy the hypothesis of our theorems, it has the advantage of being solvable in Fourier series and supports the conclusions of our theorems. The details are worked out in Section~\ref{subsec:examples}.

\begin{figure}
    \centering
    \includegraphics[scale = 0.35]{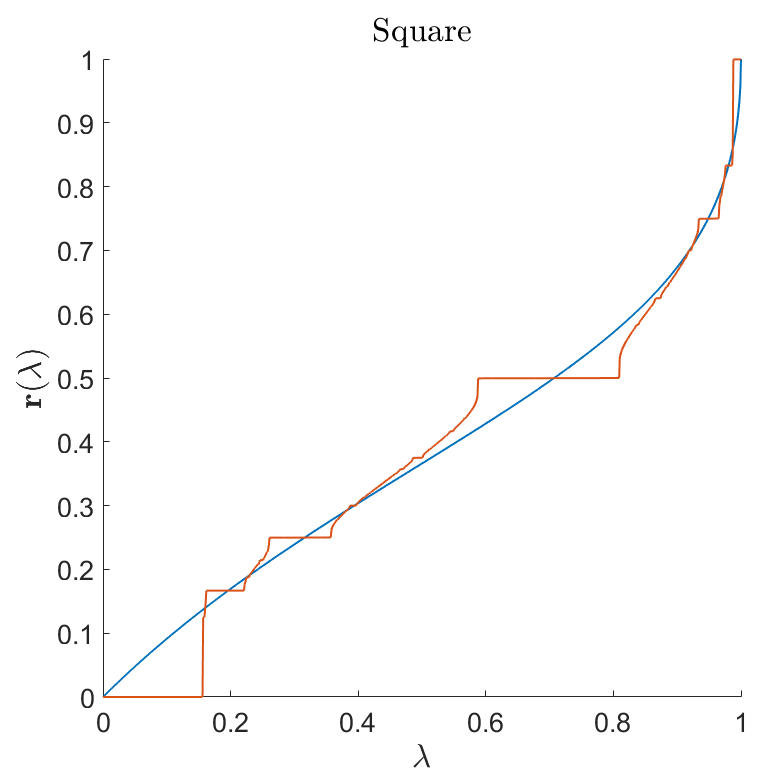}
    \includegraphics[scale = 0.35]{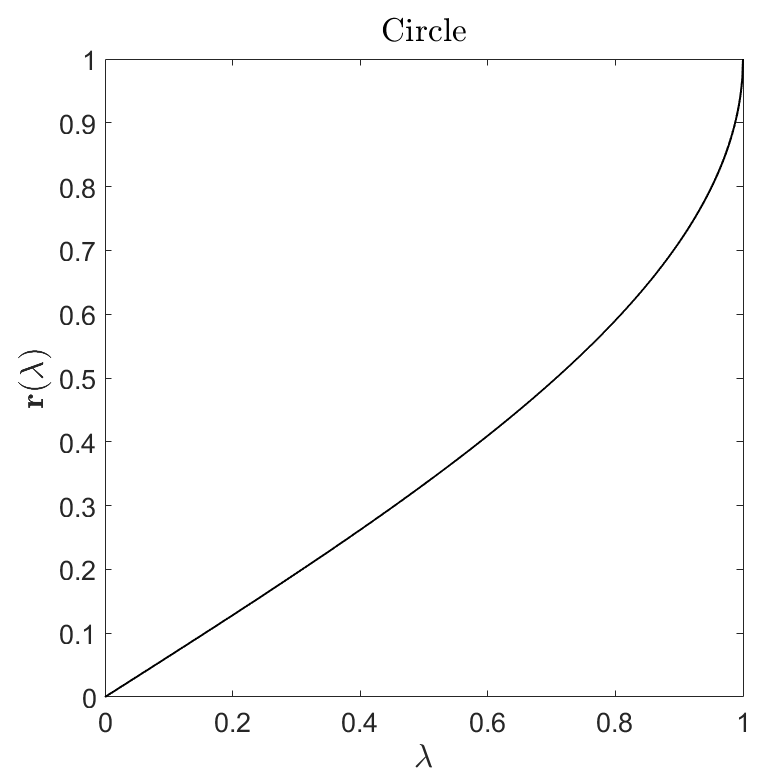}
    \caption{Graphs of the rotation number function $\mathbf r(\lambda)$ (defined in~\eqref{eq:rot_num_def}) for various domains. \textbf{Left:} The blue curve is for the untilted square $[0, 1]^2$. The orange curve is the square tilted by $\eta = \pi/20$ (see end of \S \ref{subsubsec:toy_square}) Note that it has rational plateaus. \textbf{Right:} The domain is a disk, so the boundary is a round circle.}
    \label{fig:rot_num}
\end{figure}

\subsection{Related mathematical work}
Analysis of (\ref{eq:internal_waves}) goes back to Sobolev \cite{Sobolev_54}. The spectral properties of the operator $P$ defined in (\ref{eq:P_def}) were studied by Aleksandryan \cite{Aleksandryan_60} and Ralston \cite{Ralston_73}. The study of internal waves has motivated the mathematical analysis of 0-th order self-adjoint pseudodifferential operators. 
Such operators on closed surfaces in the presence of attractors were studied by Colin de Verdi\`ere--Saint-Raymond \cite{Verdiere_Raymond_20,Colin_de_Verdiere_20} and then Dyatlov--Zworski \cite{Dyatlov_Zworski_19}.
The viscosity limits of these operators were recently studied by Galkowski--Zworski \cite{Galkowski_Zworksi_22} and Wang \cite{Wang_22}. Spectral properties of 0-th order pseudodifferential operators on the circle were studied by Zhongkai Tao~\cite{Tao_19} who produced examples of embedded eigenvalues. 
The 2D boundary case was considered in~\cite{DWZ} and our work is in the same setting but under different dynamical assumptions and with different conclusions.

The family of circle diffeomorphisms given by the chess billiard map $b(\bullet, \lambda)$ defined in (\ref{eq:b_def}) was also recently investigated by Lenci et al. in trapezoidal domains \cite{Lenci_22}. We mention that it is called the chess billiard map since it was originally used to study a generalization of the $n$-Queens problem by Hanusa--Mahankali \cite{Hanusa_Mahankali_19}, and has been further studied by Nogueira--Troubetzkoy \cite{Nogueira_Troubetzkoy_20}.

\subsection{Organization of the paper}

\noindent
\textbf{\S\ref{sec:dynamical_prelim}}: We give the smooth conjugation result from one-dimensional dynamics that is needed. A toy model that explicitly demonstrates the results of Theorems \ref{thm:spectral} and \ref{thm:evolution} is presented at the end of the section.

\noindent 
\textbf{\S\ref{sec:microlocal}}: This is dedicated to giving the necessary microlocal preliminaries. We also introduce a variation of the semiclassical calculus that allows us to much more easily keep track of the dependence on the spectral parameter.

\noindent
\textbf{\S\ref{sec:reduction_to_boundary}}: We effectively reduce the analysis of the resolvent of $P$ to the boundary using layer potentials. What we end up with are operators on the circle in the modified semiclassical calculus. 

\noindent
\textbf{\S\ref{sec:high_freq_est}}: We prove high frequency estimates for the boundary reduced problem.  

\noindent
\textbf{\S\ref{sec:LAP}}: A limiting absorption principle for the boundary reduced problem follows from the high frequency estimate together with uniqueness results proved in the beginning of this section. This is then used to give a limiting absorption principle for the resolvent of $P$ using the analysis from \S \ref{sec:reduction_to_boundary}.

\noindent
\textbf{\S\ref{sec:evolution}}: The limiting absorption principle is then used deduce the spectral result in Theorem \ref{thm:spectral}. Theorem \ref{thm:evolution} then follows by using the functional solution (\ref{eq:functional_solution}).

\section{Dynamical preliminaries}\label{sec:dynamical_prelim}
In this section, we discuss some basic properties of the underlying dynamics and a smooth conjugation result from one-dimensional dynamics that is crucial to our analysis. We also consider a toy example that can be solved explicitly to see the relation between the underlying dynamics and the behavior of the solution.  

First, we quickly explain the $\lambda$-simple assumption from Definition (\ref{lambda_simple_def}). This condition ensures that the involutions $\gamma^\pm$ defined in (\ref{eq:gamma_def}) are smooth (see \cite[\S2.1]{DWZ} for details). This immediately gives the smoothness of the chess billiard map $b(x, \lambda_0)$ defined in (\ref{eq:b_def}), which is crucial for Theorems \ref{thm:spectral} and \ref{thm:evolution}. 

It is also easy to see that $\lambda$-simplicity is an open condition (i.e. $\Omega$ is $\lambda'$-simple for all $\lambda'$ a neighborhood of $\lambda$) since the critical points of $\ell^\pm$ are assumed to be nondegenerate. Therefore the chess billiard map $b(x, \lambda)$ is well-defined in near $\lambda$, and is smooth in both $x$ and $\lambda$. Note that the smooth dependence on $\lambda$ follows immediately from smoothness of the boundary.



\subsection{One-dimensional dynamics}
From the above, we see that the chess billiard map $b(\bullet, \lambda): \partial \Omega \to \partial \Omega$ can be viewed as an orientation preserving circle diffeomorphism. We want to understand under what conditions can $b$ be smoothly conjugated to a rotation. Throughout the paper, we will denote the rotation by $\alpha \in (0, 1)$ as 
\begin{equation}\label{eq:rotation_def}
    \rho_\alpha: \mathbb S^1 \to \mathbb S^1, \qquad \rho_\alpha(x) = x + \alpha \mod 1.
\end{equation}
Here, we discuss some fundamental results from one-dimensional dynamics that will give us a better description of the chess billiard map. We cite these results without proof and refer the reader to de Melo--van Strien \cite{Melo_Strien} for a thorough introduction. 
\begin{lemma}
    The rotation number $\mathbf r(\lambda)$ as defined by the limit (\ref{eq:rot_num_def}) exists. Furthermore, it is independent of the choice of initial point $\theta_0$, the parametrization $\mathbf x$, and the lift $\mathbf b$. 
\end{lemma}
In fact, the rotation number of a circle homeomorphism $f: \mathbb S^1 \to \mathbb S^1$ (dropping the smoothness requirement) is determined entirely by the combinatorial structure, i.e. the circular ordering of the orbit $O_f(x) = \{f^n(x): n \in \Z\}$. The following theorem due to Poincar\'e characterizes circle homeomorphisms depending on if the rotation number is rational or irrational:
\begin{lemma}
    Let $f: \mathbb S^1 \to \mathbb S^1$ be a circle homeomorphism. 
    
    \noindent 
    1. $f$ has a periodic orbit if and only if the rotation number of $f$ is rational, in which case every orbit is asymptotic to a periodic orbit. 

    \noindent
    2. If $f$ has irrational rotation number $\alpha$, then orbits of $f$ have the same ordering as the orbits of the irrational rotation $\rho_{\alpha}$ map defined in (\ref{eq:rotation_def}), i.e. for any $x, y \in \mathbb S^1$, the map $\psi(f^n(x)) = \rho_\alpha^n(y)$, $n \in \Z$, is monotone. 
\end{lemma}
The orbits of an irrational rotation are dense, so completing the map $\psi$ from the above lemma by density gives a semiconjugacy between the circle homeomorphism $f$ and the irrational rotation $\rho_\alpha$, i.e. we have a continuous map $\psi: \mathbb S^1 \to \mathbb S^1$ such that 
\[\psi \circ f = \rho_\alpha \circ \psi.\]
Ultimately, we wish to find smooth coordinates on $\partial \Omega$ so that the chess billiard map $b$ is a pure rotation. In other words, we need reasonable conditions so that $\psi$ is a circle diffeomorphism. The first result in this direction was due to Denjoy \cite{Denjoy_32}, who proved that if the homeomorphism $f$ is $C^1$ and the derivative $f'$ has bounded variation and irrational rotation number $\alpha$, then $f$ is topologically conjugate to rotation by $\alpha$. However, Arnol'd found analytic circle diffeomorphisms with irrational rotation numbers that are not $C^1$ conjugate to a rotation, but proved in the same article that if the rotation number is Diophantine, analytic circle diffeomorphisms close to rotations are analytically conjugate to rotations \cite{Arnold_small_denominators_61} (and generalized to the smooth category by Moser \cite{Moser_iteration_66}). We need a stronger global result since we only know that the underlying chess billiard map is smooth and has Diophantine rotation number. The result we require is due to Herman \cite{Herman_79} and Yoccoz \cite{Yaccoz_84}. 

\begin{proposition}\label{prop:denjoy_moser}
    Let $b: \mathbb S^1 \to \mathbb S^1$ be a smooth circle diffeomorphism. Assume that the rotation number $r$ of $b$ is Diophantine. Then there exists a smooth diffeomorphism $\psi: \mathbb S^1 \to \mathbb S^1$ such that 
    $\psi \circ b \circ \psi^{-1} = \rho_r$
    where $\rho_r$ is rotation by $r$. 
\end{proposition}

An immediate corollary of smooth conjugation is ergodicity:
\begin{corollary}\label{cor:ergodic}
    Let $b: \mathbb S^1 \to \mathbb S^1$ be a smooth circle diffeomorphism with rotation number $\alpha$. Assume that $\alpha$ is Diophantine. Then $b$ is ergodic with respect to the Lebesgue measure, i.e. every measurable function $f \in L^2(\mathbb S^1)$ that satisfies $b^* f = f$ must be constant almost everywhere. 
\end{corollary}
\begin{proof}
    Let $f \in L^2( \mathbb S^1)$ be such that $b^* f = f$. By Proposition \ref{prop:denjoy_moser}, there exists a smooth diffeomorphism $\psi: \mathbb S^1 \to \mathbb S^1$ so that $\psi \circ b \circ \psi^{-1} = \rho_\alpha$. Then $g = \psi^{-1} \circ f$ satisfies $\rho^*_\alpha g = g$. It suffices to show that $g$ is constant almost everywhere. The Fourier series of $g$ must then satisfy
    \[e^{2 \pi i \alpha k} \hat g(k) = \hat g(k), \quad \hat g(k) = \int_0^1 e^{- 2 \pi i xk} g(x) \, dx.\]
    Since $\alpha$ is irrational, we see that $\hat g(k) = 0$ for $k \neq 0$, so $g$ must be constant almost everywhere. 
\end{proof}
\Remark The conditions on $b$ can be greatly relaxed to still have ergodicity. In fact, one can show that $b$ is ergodic if it is a $C^1$ diffeomorphism with bounded variation on the first derivative, and has no periodic points. See \cite{Melo_Strien}. We only need ergodicity in the Diophantine setting. 

Eventually, we will need to solve cohomological equations of the form
\begin{equation}\label{eq:basic_cohomological}
    (I - \rho^*_\alpha)v = g, \qquad \int_{\mathbb S^1} v = 0
\end{equation}
given $g \in \CIc(\mathbb S^1)$. Recall that $\rho_\alpha$ is a rotation defined in (\ref{eq:rotation_def}), so the equation can be written in a perhaps more familiar form as
\[v(x)-v(x+\alpha)=g(x).\]
If $\alpha$ satisfies the Diophantine condition, then we can have explicit high frequency control of $v$. 

\begin{lemma}\label{lem:cohom}
    Let $g \in \CIc(\mathbb S^1)$ with $\int_{\mathbb S^1} g = 0$, and assume that $\alpha$ is Diophantine with constants $c, \beta > 0$ as in Definition \ref{def:diophantine}. Then there exists unique $v \in C^\infty(\mathbb S^1)$ that solves (\ref{eq:basic_cohomological}) such that for all $s$
    \begin{equation*}
        \|v\|_{H^s} \le C\|g\|_{H^{s + \beta + 1}}
    \end{equation*} 
    where $C$ depends only on $c,s$. 
\end{lemma}
\begin{proof}
    Taking the Fourier series of both sides of (\ref{eq:basic_cohomological}),
    \[\hat f(k) = (1 - e^{2 \pi i k \alpha}) \hat v(x).\]
    For all $k \neq 0$,  
    \[\left|\frac{1}{1 - e^{2 \pi i k \alpha}} \right| \le c^{-1} |k|^{1 + \beta}.\]
    Since $\hat g(0) = 0$, $v$ can be solved modulo constant functions, and choosing $\hat v(0) = 0$, we have the estimate
    \[\|v\|_{H^s} \le C\|g\|_{H^{s + \beta + 1}}\]
    as desired. 
\end{proof}

\subsection{Examples}\label{subsec:examples}
We discuss some examples on explicit domains. In the square domain toy model, we will see how the underlying dynamics influence the regularity of the solution explicitly. 
\subsubsection{Square domain}\label{subsubsec:toy_square}
A useful example to keep in mind is the square domain case where $\Omega = [0, 1]^2$. Our arguments for general domains will depend crucially on a $C^\infty$ boundary. This toy example has the advantage that (\ref{eq:internal_waves}) can be solved directly in Fourier series. Furthermore, the chess billiard flow on the square is the same as the standard billiard flow, and the rotation number function $\mathbf r(\lambda)$ defined in (\ref{eq:rot_num_def}) is smooth and can be written down explicitly. It is given by
\begin{equation*}
    \mathbf r(\lambda) = \frac{\lambda}{\sqrt{1 - \lambda^2} + \lambda}.
\end{equation*}
See \cite{RSI} for the full derivation. We can formally write
\[u(t, x) = \sum_{\mathbf k \in \N^2} \hat u(t, \mathbf k) \sin(\pi k_1 x_1) \sin (\pi k_2 x_2) \]
where $\mathbf k = (k_1, k_2) \in \N^2$. Only in this subsection, we use the hat to denote Fourier transform with respect to the Dirichlet sine basis. If $u(t, x)$ is a solution to (\ref{eq:internal_waves}), the coefficients must satisfy the periodically driven harmonic oscillator equation
\begin{equation}\label{eq:f_coeffs}
    \begin{gathered}
        -\pi^2 (k_1^2 + k_2^2) \partial_t^2 \hat u(t, \mathbf k) + \pi^2 k_2^2 \hat u(t, \mathbf k) = \hat f(\mathbf k) \cos \lambda_0 t, \\
    \text{where} \quad \hat f(\mathbf k) = \int_{[0, 1]^2} f(x) \sin (\pi k_1 x_1) \sin(\pi k_2 x_2)\, dx_1 dx_2
    \end{gathered}
\end{equation}
This has solution 
\begin{equation}\label{eq:square_solution}
    \hat u (t, \mathbf k) = \frac{\hat f (\mathbf k)}{\lambda_0^2 k_1^2 - (1 - \lambda_0^2) k_2^2} \left[\cos(\lambda_0 t) - \cos \left( \frac{k_2}{|\mathbf k|} t \right) \right]
\end{equation}
If $\mathbf r(\lambda_0)$ is Diophantine, then the result of Theorem \ref{thm:evolution} holds. In fact, in this case, we get that $u(t) \in C^\infty(\Omega)$ uniformly in all the seminorms for all $t \in \R$. Indeed, if $\mathbf r(\lambda_0)$ is Diophantine, then there exist constants $c,\, \beta > 0$ such that
\[|q \cdot \mathbf r(\lambda_0) - p| \ge \frac{c}{q^{1 + \beta}}\]
for any $p \in \Z$ and $q \in \N$. Rewriting this condition in terms of $\lambda_0$, we find that 
\[|(q - p) \lambda_0 - p\sqrt{1 - \lambda_0^2}| \ge \frac{c}{q^{1 + \beta}}\]
where $c$ is a possibly different constant. Put $q = k_1 + k_2$ and $p = k_2$ to find that 
\begin{equation}\label{eq:sq_diophantine_est}
    |k_1 \lambda_0 - k_2 \sqrt{1 - \lambda_0^2}| \ge \frac{c}{q^{1 + \beta}} \gtrsim \frac{1}{|\mathbf k|^{1 + \beta}}.
\end{equation}
where the hidden constant is independent of $\mathbf k$. We clearly have 
\begin{equation}\label{eq:sq_est_2}
    |k_1 \lambda_0 + k_2 \sqrt{1 - \lambda_0^2}| \ge 1,
\end{equation}
so combining (\ref{eq:sq_diophantine_est}) and (\ref{eq:sq_est_2}) yields
\begin{equation}\label{eq:combined_sq_est}
    |\lambda_0^2 k_1^2 - (1 - \lambda_0^2)k_2^2| \ge \frac{c}{|\mathbf{k}|^{1 + \beta}}.
\end{equation}
for a possibly different constant $c > 0$. Since $f \in \CIc(\Omega; \R)$, $\hat f(\mathbf k)$ is rapidly decreasing in $\mathbf k$, which tempers the denominator in (\ref{eq:square_solution}) to give uniform smoothness of $u(t)$ in time.  

The above analysis also exhibits the spectral result of Theorem \ref{thm:spectral}. Note that $\kappa(\mathbf k)\sin(\pi k_1 x_1) \sin(\pi k_2 x_2)$, $\mathbf k \in \N^2$ form a complete orthonormal basis for $H^{-1}(\Omega)$, where $\kappa(\mathbf k) = \mathcal O(|k|)$ are simply normalizing constants. This basis consists of eigenfunctions with eigenvalues $\frac{k_1}{k_1 + k_2}$ for the operator $P$ defined in (\ref{eq:P_def}). In particular, $\sin(\pi k_1 x_1) \sin(\pi k_2 x_2)$ has eigenvalue $\frac{k_2^2}{k_1^2 + k_2^2}$. If $\mathbf r(\lambda_0)$ is Diophantine, then (\ref{eq:combined_sq_est}) gives a characterization of the eigenvalues near $\lambda_0^2$:
\begin{equation*}
    \left|\frac{k_2^2}{k_1^2 + k_2^2} - \lambda_0^2\right| \le \epsilon \implies \frac{1}{|\mathbf k|^{3 + \beta}} \gtrsim \epsilon^{-1}
\end{equation*}
Therefore the spectral measure $\mu_{f, f}$ satisfies the results of Theorem \ref{thm:spectral} near $\lambda_0^2$. Indeed, using the fact that the coefficients $\hat f(k)$ of $f \in \CIc(\Omega)$ defined in (\ref{eq:f_coeffs}) are rapidly decreasing, we have
\[\mu_{f, f}((\lambda_0^2 - \epsilon, \lambda_0^2 + \epsilon)) \lesssim \sum_{\mathbf k : \big|\frac{k_2}{k_1^2 + k_2^2} - \lambda_0^2 \big| \le \epsilon^{-1}} \hat f(\mathbf k)^2 \kappa(\mathbf k)^{-2} < C_d \epsilon^d \]
for any $d \in \N$. 

Finally, we mention that by tilting the square by $\eta$, the set of $\lambda$ for which $\mathbf \lambda$ is Diophantine is no longer a full measure set. More specifically, we consider the square domain specified by the vertices 
\[\big\{(0, 0), \, (\cos \eta, \sin \eta), \, (-\sin \eta, \cos \eta), \, \sqrt{2}(\cos(\eta + \tfrac{\pi}{4}), \sin(\eta + \tfrac{\pi}{4})) \big \}.\] 
Then graph of $\mathbf r(\lambda)$ is constant near values of $\lambda$ for which $\mathbf r(\lambda)$ is rational. See Figure~\ref{fig:rot_num} for an illustration and \cite[\S2.5]{DWZ} for details. 

\subsubsection{Circle}
An example of a domain that satisfies the hypothesis of our theorem generically is the disk $\mathbb D = \{|x| \le (2 \pi)^{-1}\}$. We identify the boundary with the circle $\mathbb S^1$. The boundary is smooth, and the rotation number can be computed explicitly. Indeed, the angle inscribed in the circle by the chords $(x, \gamma^-(x))$ and $(\gamma^-(x), b(x))$ is independent of $x$. The arc substended by such an inscribed angle is then independent of $x$, so we can compute that 
\begin{equation*}
    b(x) = x - \frac{2}{\pi} \arctan \left( \frac{\sqrt{1 - \lambda^2}}{\lambda} \right) \mod 1 \implies \mathbf r(\lambda) = 1 - \frac{2}{\pi} \arctan \left( \frac{\sqrt{1 - \lambda^2}}{\lambda} \right).
\end{equation*}
Clearly $\mathbf r(\lambda)$ is smooth with $\mathbf r'(\lambda) > 0$, and so the set of $\lambda \in (0, 1)$ for which $\mathbf r(\lambda)$ Diophantine has Lebesgue measure $1$. See Figure \ref{fig:rot_num}.

\section{Microlocal analysis} \label{sec:microlocal}
In this section, we review some standard results from microlocal analysis. Furthermore, we develop a modification of the semiclassical calculus, which we refer to as the positive (or negative) semiclassical calculus, that will be useful in understanding the boundary reduced problem. 

\subsection{Microlocal calculus on \texorpdfstring{$\partial \Omega$}{}}
We first discuss classical pseudodifferential operators on $\mathbb S^1$, and refer the reader to \cite[Chapter 18.1]{H3} for details. 
The boundary $\partial \Omega$ can be identified with $\mathbb S^1 = \R/\Z$. Define the space of 1-periodic Kohn--Nirenberg symbols $S^m(T^* \mathbb S^1)$ as the set of smooth functions $a \in C^\infty(\R_x \times \R_\xi)$ that satisfies
\begin{equation}\label{eq:KN_seminorm}
    a(x + 1, \xi) = a(x, \xi), \qquad |\partial_x^\alpha \partial_\xi^\beta a(x, \xi)| \le C_{\alpha, \beta} \langle \xi \rangle^{m - \beta},
\end{equation}
where $\langle \xi \rangle := \sqrt{1 + |\xi|^2}$. We will only be working with symbols on $T^* \mathbb S^1$, so simply denote $S^m = S^m(T^* \mathbb S^1)$. Pseudodifferential operators are quantizations of these symbols. Generally speaking, the quantization procedure is not canonical. However, it will be important for us to relate pseudodifferential operators to Fourier multipliers. So we fix the following standard quantization procedure. The quantization of $a \in S^m$ is the operator $\Op(a): C^\infty(\mathbb S^1) \to C^\infty(\mathbb S^1)$ given by 
\begin{equation}\label{eq:std_quant}
    \Op(a) u(x) = \frac{1}{2 \pi} \int_{\R^2} e^{i(x - y) \xi} a(x, \xi) u(y)\, dy d \xi.
\end{equation}
Here, $u \in C^\infty (\mathbb S^1)$ is understood as a 1-periodic function, and the integral is understood in the sense of oscillatory integrals (see, for instance, \cite[Chapter 1]{GS_94}). $\Op(a)$ also extends to an operator from $\mathcal D'(\mathbb S^1) \to \mathcal D'(\mathbb S^1)$. In terms of Fourier series, this quantization procedure can be equivalently defined as 
\begin{equation}\label{eq:op_fourier_series}
    \begin{gathered}
        \Op(a) u(x) = \sum_{k, n \in \Z} e^{2 \pi i n x} a_{n - k} (k) \hat u(k), \\
        a_\ell(k) := \int_0^1 a(x, 2 \pi k) e^{-2 \pi i \ell x}\, dx, \quad \hat u(k) := \int_0^1 u(x) e^{2 \pi i kx} \, dx.
    \end{gathered}
\end{equation}
Observe that when $a \in S^m$ is only a function of $\xi$, then $\Op(a)$ is a Fourier multiplier:
\[\Op(a)u(x) = \sum_{k \in \Z} e^{2 \pi i nx} a(2 \pi k) \hat u(k).\]

Now we define some spaces of pseudodifferential operators. We first define the space of smoothing operators on $\mathbb S^1$:
\begin{equation}\label{eq:smoothing}
    \Psi^{-\infty}(\mathbb S^1) := \{R: C^\infty(\mathbb S^1) \to \mathcal D'(\mathbb S^1) \mid \text{the Schwartz kernel lies in $C^\infty(\mathbb S^1 \times \mathbb S^1)$}\}.
\end{equation}
Notably, $R \in \Psi^{-\infty}(\mathbb S^1)$ extends to an operator from $\mathcal D'(\mathbb S^1) \to C^\infty( \mathbb S^1)$. Now we can define the space of $m$-th order pseudodifferential operators on $\mathbb S^1$ as 
\begin{equation}\label{eq:KN_class}
    \Psi^m(\mathbb S^1) := \{\Op(a) + R \mid a \in S^m,\, R \in \Psi^{-\infty}(\mathbb S^1)\}. 
\end{equation}
We remark that 
\[\Psi^{-\infty}(\mathbb S^1) = \bigcap_{m \in \R} \Psi^{m}(\mathbb S^1).\]

We will encounter $\omega$-dependent families of operators $A = A_\omega \in \Psi^m(\mathbb S^1)$ whose full symbols satisfy (\ref{eq:KN_seminorm}) uniformly in a parameter $\omega$. Here, a full symbol of $A$ is a family $\sigma(A) = a = a_\omega \in S^m(T^* \mathbb S^1)$ such that $\Op(a) - A \in \Psi^{-\infty}$ uniformly in $\omega$. We will often suppress the $\omega$ dependence in the notation. Note that the full symbol of $A$ is not unique, both due to the Fourier series formula for quantization (\ref{eq:op_fourier_series}) and because the operators are understood modulo $\Psi^{-\infty}$. 

The symbol composition formula descends directly from the composition formula for pseudodifferential operators on $\R$: let $a \in S^m$ and $b \in S^\ell$ then
\begin{equation}\label{eq:classical_comp}
    \begin{gathered}
        \Op(a) \Op(b) \in \Psi^{m + \ell}(T^* \mathbb S^1), \quad \Op(a) \Op(b) = \Op(a \# b), \\
        \text{where} \quad a \# b = e^{i D_\xi D_y} a(x, \xi) b(y, \eta) \big|_{\substack{y = x \\ \eta = \xi}} \in S^{m + \ell}.
    \end{gathered}
\end{equation}
Here and also throughout the paper, we adopt the H\"ormander convention of $D_x := -i\partial_x$. This also gives the asymptotic summation formula 
\begin{equation}\label{eq:comp_expansion}
    a \# b \sim \sum_{k = 0}^\infty \frac{i^k}{k!} D_\xi^k a(x, \xi) D_x^k b(x, \xi),
\end{equation}
which is understood as
\[a\# b - \sum_{k = 0}^{N - 1} \frac{i^k}{k!} D_\xi^k a(x, \xi) D_x^k b(x, \xi) \in S^{m + \ell - N}.\]

We also have a change of variables formula for pseudodifferential operators. For a diffeomorphism $\psi: \mathbb S^1 \to \mathbb S^1$, we will contract some notation and use
\[\psi^{-*}:= (\psi^{-1})^*\]
to denote the pullback of the inverse. 

\begin{proposition}\label{prop:classical_change}
    Let $\psi: \mathbb S^1 \to \mathbb S^1$ be a diffeomorphism and let $a \in S^m$. Then there exists $\tilde a \in S^m$ such that 
    \begin{equation}
        \psi^* \Op(a) \psi^{-*} = \Op (\tilde a).
    \end{equation}
    Furthermore, $\tilde a$ has asymptotic expansion 
    \begin{equation}\label{eq:change_expansion}
        \tilde a \sim \sum_{k = 0}^\infty L_k(a \circ \tilde \psi),
    \end{equation}
    where $L_k$ are differential operators of order $2k$ on $T^* \mathbb S^1$ and map $S^m \to S^{m - k}$, and $\widetilde \psi$ is the lifted symplectomorphism of $\psi$, i.e. 
    \[\widetilde \psi: T^* \mathbb S^1 \to T^* \mathbb S^1, \quad \widetilde \psi(x, \xi) := (\psi(x), (d \psi(x))^{-1} \xi).\]
    Futhermore, $L_0 = 1$, and $L_k = L'_k D_\xi$ for some differential operator $L'_k$ for every $k \ge 1$. 
\end{proposition}
We immediately have the following corollary from the asymptotic expansions (\ref{eq:comp_expansion}) and (\ref{eq:change_expansion}).
\begin{corollary}\label{cor:Heaviside_identities}
    Let $\chi \in C^\infty(T^* \mathbb S^1)$ be such that $\supp \chi \subset \{\xi \ge 0\}$ and $\chi = 1$ near $\xi = + \infty$. Then for any $\omega$-dependent families of $A \in \Psi^{m}$ and orientation preserving diffeomorphism $\psi:\mathbb S^1 \to \mathbb S^1$, we have
    \[[\Op(\chi), A] \in \mathcal O(1)_{\Psi^{-\infty}} \quad \text{and} \quad \psi^* \Op(\chi) \psi^{-*} = \Op(\chi) + \Psi^{-\infty}.\] 
\end{corollary}

For a (possibly $\omega$-dependent) operator $A \in \Psi^m(\mathbb S^1)$ with full symbol $a \in S^m$, we define the principal symbol as
\begin{equation}
    \sigma_m(A) = [a] \in S^m/S^{m - 1}.
\end{equation}
While the full symbol is not unique, the principal symbol is. In fact, we have a short exact sequence
\[0 \rightarrow \Psi^{m - 1} \rightarrow \Psi^{m} \xrightarrow{\sigma_m} S^m/S^{m - 1} \to 0. \]
The principal symbol is an element of the quotient will be implicitly understood throughout, and we will write $\sigma_m(A) = b$ for any $b$ that satisfies $a - b \in S^{m - 1}$ uniformly in $\omega$. We emphasize here that $\sigma_m$ with a subscript will denote the principal symbol, and $\sigma$ without a subscript will denote a choice of a full symbol. 

With the quantization procedure given by (\ref{eq:std_quant}), the symbol $a(x, \xi)$ is called a \textit{left symbol}. We can also have symbols of the form $a \in C^\infty(\R_x \times \R_y \times \R_\xi)$ that are $1$-periodic in both $x$ and $y$. We say that $a \in S^m(\R_x \times \R_y \times \R_\xi)$ if it satisfies the estimate
\[|\partial_{x, y}^\alpha \partial_\xi^\beta a(x, \xi)| \le C_{\alpha, \beta} \langle \xi \rangle^{m- \beta}.\]
We can quantize these symbols by 
\[\Op(a) u(x) = \frac{1}{2 \pi} \int_{\R^2} e^{i(x - y) \xi} a(x, y, \xi) u(y)\, dy d\xi.\]
This gives the same family of pseudodifferential operators on $\mathbb S^1$ as in (\ref{eq:KN_class}):
\[\Psi^m(\mathbb S^1) = \{\Op(a) + R \mid a \in S^m(\R_x \times \R_y \times \R_\xi),\, R \in \Psi^{-\infty}(\mathbb S^1) \}.\]
Moreover, any such symbol can be reduced to a left symbol by the following asymptotic formula. Let $a \in S^m(\R_x \times \R_y \times \R_\xi)$, and define the left symbol by the asymptotic expansion 
\begin{equation}\label{eq:left_reduction}
    b(x, \xi) \sim \sum_{k = 0}^\infty \frac{i^{-k}}{k!} \left(\partial_\xi^k \partial_y^k a(x, y, \xi) \right)|_{y = x},
\end{equation}
which exists by Borel's lemma (see \cite[Proposition 1.8]{GS_94}). Then $\Op(a) - \Op(b) \in \Psi^{-\infty}$. The symbol $b$ is called the \textit{left reduction} of the symbol $a$.

\subsection{Semiclassical calculus}
Now we define the semiclassical calculus over $\mathbb S^1$ as well as a slight variation that restricts to positive (or negative) frequencies. Semiclassical operators quantize symbols that depend on a semiclassical parameter $h$, such that when the symbol is differentiated, there is extra decay in $h$ as $h \to 0$.

\subsubsection{The standard semiclassical calculus}
Denote by $S^m_h(T^* \mathbb S^1)$ the set of $h$-dependent and $1$-periodic symbols that satisfy (\ref{eq:KN_seminorm}) uniformly in $h$. For $a \in S^m_h(T^* \mathbb S^1)$, we fix the semiclassical quantization procedure 
\begin{equation}
    \Op_h(a)u(x) := \frac{1}{2 \pi h} \int_{\R^2} e^{\frac{i}{h}(x - y) \xi} a(x, \xi) u(y) \, dy d\xi = \frac{1}{2 \pi} \int_{\R^2} e^{i(x - y) \xi} a(x, h\xi) u(y)\, dy d\xi.
\end{equation}
We see that this is related to the nonsemiclassical quantization procedure by 
\[\Op_h(a) = \Op(a_h) \quad \text{where} \quad a_h(x, \xi) = a(x, h \xi).\]
We define the space of $m$-th order semiclassical pseudodifferential operators on $\mathbb S^1$ as 
\begin{equation}\label{eq:semiclassical_calculus}
    \Psi^m_{h}(\mathbb S^1) := \{\Op_h(a) + R \mid a \in S_h^m, \, R \in h^\infty \Psi^{-\infty}(\mathbb S^1)\},
\end{equation}
where $h^{\infty} \Psi^{-\infty} := \mathcal O(h^\infty)_{\Psi^{-\infty}}$. We say that $a \in S_h^m(T^* \mathbb S^1)$  is a semiclassical full symbol of $A \in \Psi_h^m(\mathbb S^1)$ if $\Op_h(a) = A$ modulo $h^{\infty} \Psi^{-\infty}$. 
Note that 
\[\bigcap_{m \in \R} h^m \Psi_h^{-m}(\mathbb S^1) = h^\infty \Psi^{-\infty}.\]
Here, we intentionally put the residual class as $h^\infty \Psi^{-\infty}$ without the subscript $h$, where $\Psi^{-\infty}$ is the same nonsemiclassical residual class defined in~\eqref{eq:smoothing}, uniformly in $h$. We also define the space
\begin{equation}\label{eq:psi-inf}
    \Psi^{-\infty}_{h}(\mathbb S^1) := \{\Op_h(a) + R \mid a \in S_h^{-\infty}, \, R \in h^\infty \Psi^{-\infty}(\mathbb S^1)\}
\end{equation}
in complete analogy with~\eqref{eq:semiclassical_calculus}. It is important to note the following relationship between the semiclassical calculus with parameter $h$ and the nonsemiclassical calculus with an implicit parameter $h$:
\begin{equation}\label{eq:zero_infinity}
    h^\infty \Psi^{-\infty}_h(\mathbb S^1) \subset \Psi^{-\infty}(\mathbb S^1), \qquad \Psi^{-\infty}_h(\mathbb S^1) \subset \Psi^{0}(\mathbb S^1). 
\end{equation}
Most importantly, operators in $\Psi_h^{-\infty}(\mathbb S^1)$ are \textit{not} uniformly in $\Psi^{-\infty}(\mathbb S^1)$ with respect to $h$.

We also have the following composition and change of variables formulas. For $A \in \Psi^m_h(\mathbb S^1)$ and $B \in \Psi^\ell_h(\mathbb S^1)$ with full symbols $a$ and $b$ respectively, 
\begin{equation}
    \begin{gathered}
        AB \in \Psi_h^{m + \ell}(\mathbb S^1), \quad AB = \Op_h(a \# b) + R, \quad R \in h^\infty\Psi^{-\infty}, \\
        \text{where} \quad a \# b \sim \sum_{k = 0}^\infty \frac{(ih)^k}{k!} D_\xi^k a(x, \xi) D_x^k b(x, \xi) \in S_h^{m + \ell},
    \end{gathered}
\end{equation}
and for every diffeomorphism $\psi: \mathbb S^1 \to \mathbb S^1$, we have
\begin{equation}\label{eq:semiclassical_COV}
\begin{gathered}
    \psi^* \Op(a) \psi^{-*} = \Op (\tilde a), \\
    \text{where} \quad \tilde a \sim \sum_{k = 0}^\infty h^k L_k(a \circ \tilde \psi).
\end{gathered}
\end{equation}

There exists a canonical principal symbol map 
\begin{equation}
    \sigma_h(A) := [a] \in S_h^m / hS_h^{m - 1},
\end{equation}
where $a$ is a semiclassical full symbol of $A$. We have a short exact sequence 
\[0 \rightarrow h\Psi^{m - 1}_h \rightarrow \Psi^m_h \xrightarrow{\sigma_h} S_h^m/hS_h^{m - 1} \to 0.\]

We need a notation of where an operator is microlocally nontrivial. This is captured by the operator wavefront set. This is a closed subset $\WF_h(A) \subset T^* \mathbb S^1$. We say that $(x_0, \xi_0) \notin \WF_h(A)$ if there exists a neighborhood $U$ of $(x, \xi)$ such that a full symbol $a$ of $A$ satisfies
\[|\partial_x^\alpha \partial_\xi^\beta a(x, \xi)| \le C_{N, \alpha, \beta} h^{N} \langle \xi \rangle^{-N}.\]
One can easily check that $\WF_h(A)$ is independent of the choice of the full symbol used in the definition since $\WF_h(R) = \emptyset$ for all $R \in h^\infty \Psi^{-\infty}$. The converse of this does not exactly hold since in reality, the wavefront set should be defined on the fiber-radial compactification of $T^* \mathbb S^1$, which includes fiber infinity. However, this issue does not come up in the context of this paper. 

We will need the following result on the boundedness of pseudodifferential operators on Sobolev spaces. 
\begin{lemma}\label{lem:contraction}
    Fix $s \in \R$. Let $A \in \Psi^{0}_h(\mathbb S^1)$, and suppose the semiclassical principal symbol $\sigma_h(A)$ is supported away from the zero section of $T^* \mathbb S^1$. Furthermore, suppose $\sup_{(x, \xi) \in T^*\mathbb S^1} |\sigma_h(A)(x, \xi)| < M$ for all sufficiently small $h$. Then 
    \[\|A\|_{H^s \to H^s} < M \]
    for all sufficiently small $h$.
\end{lemma}
\begin{proof}
Denote the principal symbol by $a := \sigma_h(A)$ and the rescaling $a_h(x, \xi) := a(x, h \xi)$. Recall that $a_h \in \mathcal O(1)_{\Psi^0}$. Since $a$ is supported away from the zero section, we see that $\langle \xi \rangle^{-k} \partial_x^k a_h = \mathcal O(h^k)_{\Psi^0}$. Then we see from the nonsemiclassical composition formula (\ref{eq:classical_comp}) that
\begin{equation}\label{eq:sobolev}
    \Op(\langle \xi \rangle^s) A \Op(\langle \xi \rangle^{-s}) = \Op_h(a) + \mathcal O(h)_{\Psi^0
    }
\end{equation}
Using \cite[Proposition 13.13]{Zworski_semiclassical_analysis} for the first term on the right hand sice of (\ref{eq:sobolev}) and \cite[Proposition 4.23]{Zworski_semiclassical_analysis} for the second term, we recover the desired bound. 
\end{proof}
\Remark The support condition in the above lemma is crucial. Consider an operator $\Op_h(a)$ where $a$ is independent of $\xi$. This is simply the multiplication operator $f(x) \mapsto a(x) f(x)$ for all $h$. Even if $a(x) < M$, it is possible that $\partial_x a(x) \gg M$, so the lemma clearly fails for $s = 1$. Without the support condition, Sobolev boundedness based on the principal symbol must be formulated using semiclassical Sobolev spaces (see \cite[Appendix E]{DZ_resonances}).

\subsubsection{Positive semiclassical calculus}\label{subsubsec:positive_semiclassical} The smoothness of the symbol near the zero section is very important in the semiclassical calculus. However, the operators that will appear in the reduction to boundary of the internal waves equation (see \S\ref{sec:reduction_to_boundary}) will look like semiclassical quantization of symbols that are supported on only the positive frequencies: 
\begin{equation*}
    T^*_+ \mathbb S^1 := \mathbb S^1 \times [0, \infty)_\xi.
\end{equation*}
We will see that the quantization of such symbols will come at the cost of an extra smoothing error that does not decay with $h$. 

Let us first define the symbol class $S_h^m(T^*_+ \mathbb S^1)$. We say that $a \in C^\infty(\R_x \times [0, \infty)_\xi)$ is in $S_h^m(T^*_+ \mathbb S^1)$ if it satisfies (\ref{eq:KN_seminorm}) uniformly in $h$ for all $\xi \in [0, \infty)$. 

Fix a cutoff 
\begin{equation}\label{eq:positive_semiclassical_cutoff}
    \chi_+ \in C^\infty([0, \infty)), \quad \supp \chi_+ \subset (1/2, \infty), \quad \chi_+ \equiv 1 \quad \text{near} \quad [1, \infty).
\end{equation}
Define the quantization procedure
\begin{equation}\label{eq:positive_semiclassical_quant}
    \Op_h^+(a) = \Op(a^+_h) \quad \text{where} \quad a^+_h(x, \xi) = \chi_+(\xi) a(x, h \xi),
\end{equation}
where $\Op$ is the standard quantization procedure defined in (\ref{eq:std_quant}). We stress that $T^*_+ \mathbb{S}^1$ is a manifold with boundary, and the smoothness of $a$ up to the boundary will ensure that the residual class is $\Psi^{-\infty}$ uniformly in $h$. Define the projector onto the positive frequencies
\begin{equation}\label{eq:new_quant}
    \Pi^+ = \Op(\chi_+), \quad \Pi^+ u = \sum_{k > 0} e^{2 \pi i kx} \hat u(k).
\end{equation}
Then observe that from the Fourier series expansion of the nonsemiclassical quantization procedure (\ref{eq:op_fourier_series}), we see that (\ref{eq:positive_semiclassical_quant}) can be rewritten as
\[\Op_h^+(a) = \Op_h(\breve a) \Pi^+,\]
where $\breve a \in C^\infty(T^* \mathbb S^1)$ is any smooth extension of $a$ still lying in a symbol class. From (\ref{eq:op_fourier_series}) and (\ref{eq:positive_semiclassical_quant}), it is clear that the quantization procedure is independent of the choice of extension $\breve a$ and independent of the choice of cutoff $\chi_+$ as long as it satisfies (\ref{eq:positive_semiclassical_cutoff}). In fact, as long as we have a cutoff that is identically $1$ near infinity, it differs from the quantization with $\chi_+$ only by a smoothing operator in $\Psi^{-\infty}$ uniformly in $h$. 

The advantage of this new formula for quantization is that the composition formula and the change of variables formula for orientation preserving diffeomorphisms follow from the well-known formulas in the semiclassical and the nonsemiclassical calculi.

We first derive a composition formula. 

\begin{lemma}\label{lem:positive_comp}
    Let $a \in S_h^m(T^*_+ \mathbb S^1)$ and $b \in S_h^\ell(T^*_+ \mathbb S^1)$. Then 
    \begin{equation}\label{eq:positive_comp}
        \Op_h^+(a) \Op_h^+(b) = \Op^+_h(ab + \mathcal O(h)_{S_h^{m + \ell - 1}(T^*_+ \mathbb S^1)}) + \Psi^{-\infty}.
    \end{equation}
    Here, $\Psi^{-\infty}$ denotes a remainder that lies in $\Psi^{-\infty}$ uniformly in $h$. 
\end{lemma}
\begin{proof}
    1. We begin with some general observations. Consider a smooth extension of $a$ that specifically satisfies 
    \begin{equation}\label{eq:nice_extension}
        \breve a(x, \xi) = a(x, \xi) \quad \text{for all} \quad \xi \ge 0, \quad \supp \breve a \subset \{\xi > -1\}
.    \end{equation}
    We note that $h \langle h \xi \rangle^{-1} \le \langle \xi \rangle^{-1}$. In particular, this implies that
    \begin{equation}\label{eq:uniform_symbol_class}
        \breve a_h \in S_h^{\max\{m, 0\}} (T^* \mathbb S^1)
    \end{equation}
    where $\breve a_h(x, \xi) := \breve a(x, h \xi)$. Therefore, $\Op_h(\breve a) = \mathcal O(1)_{\Psi^{\max\{m, 0\}}}$. From Corollary \ref{cor:Heaviside_identities}, we then see that 
    \begin{equation}\label{eq:smoothing_commutator}
        [\Pi^+, \Op_h(\breve a)] = \mathcal O(1)_{\Psi^{-\infty}}.
    \end{equation}
    Finally, we note that it follows from (\ref{eq:op_fourier_series}) that 
    \[(\Pi^+)^2 = \Pi^+.\]

    \noindent
    2. Let $\breve a$ and $\breve b$ be smooth extensions of $a \in S^m_h(T^*_+ \mathbb S^1)$ and $b \in S^\ell_h(T^*_+ \mathbb S^1)$ respectively satisfying (\ref{eq:nice_extension}). Then 
    \begin{align}
        \Op_h^+(a) \Op_h^+(b) &= \Op_h(\breve a) \Pi^+ \Op_h(\breve b) \Pi^+ \nonumber \\
        &= \Op_h(\breve a)\Op_h(\breve b) \Pi^+ + \Op_h(\breve a) [\Pi^+, \Op_h(\breve b)] \Pi^+. \label{eq:comp_commutator}
    \end{align}
    The first term on the right hand side can be understood using the semiclassical calculus:
    \begin{align*}
    \Op_h(\breve a)\Op_h(\breve b) \Pi^+ &= \Op_h \big(\breve a \breve b + \mathcal O(h)_{S_h^{m + \ell - 1}(T^* \mathbb S^1)} \big) \Pi^+ \\
    &= \Op_h^+\big(ab + \mathcal O(h)_{S_h^{m + \ell - 1}(T^*_+ \mathbb S^1)}\big).
    \end{align*}
    By (\ref{eq:uniform_symbol_class}) and (\ref{eq:smoothing_commutator}), it follows from Corollary \ref{cor:Heaviside_identities} that the second term on the right hand side of (\ref{eq:comp_commutator}) belongs in $\mathcal O(1)_{\Psi^{-\infty}}$, which completes the proof. 
\end{proof}

From the above lemma, we see that a key difference between the positive semiclassical calculus and the standard semiclassical calculus is that the former will always produce smoothing errors uniform in $h$, but do not necessarily decay in $h$. 



Now we consider the change of variables formula. We stress that the orientation preserving condition is essential since orientation reversing diffeomorphisms exchange positive and negative frequencies. 
\begin{lemma}\label{lem:positive_COV}
    Let $a \in S_h^m(T^*_+ \mathbb S^1)$ and let $\psi$ be an orientation preserving diffeomorphism on $\mathbb S^1$. Then
    \[\psi^* \Op_h^+(a) \psi^{-*} - \Op_h^+(a \circ \tilde \psi) \in h\Psi^{m - 1}_{h, +} + \Psi^{-\infty},\]
    where $\tilde \psi: T^*_+ \mathbb S^1 \to T^*_+ \mathbb S^1$ is the lifted symplectomorphism, which is well-defined since $\psi$ is orientation preserving. 
\end{lemma}

\begin{proof}
    We simply use Proposition \ref{prop:classical_change} and Corollary \ref{cor:Heaviside_identities} to see that 
    \[\psi^* \Op_h(\breve a) \Pi^+ \psi^{-*} = \Op_h(\breve a \circ \tilde \psi) \Pi^+ + \Op_h(b) \Pi^+ + \Psi^{-\infty}\]
    for some $b \in S^{m - 1}_h (T^* \mathbb S^1)$. The change of variable formula for the positive semiclassical calculus then follows from (\ref{eq:new_quant}).
\end{proof}

Define the space
\begin{equation*}
    \Psi_{h, +}^m(\mathbb S^1) := \{\Op_h^+(a) + R \mid a \in S^m(T^* \mathbb S^1), \, R \in \Psi^{-\infty}(\mathbb S^1)\}.
\end{equation*}
Note that the error does not decay as $h \to 0$ unlike in the standard semiclassical calculus. We can define the principal symbol map 
\begin{equation*}
    \sigma_h^+(A) := [a] \in S_h^m(T^*_+ \mathbb S^1) / h S_h^{m - 1}(T^*_+ \mathbb S^1)
\end{equation*}
where 
\[A - \sigma_h^+(a) \in h \Psi_{h, +}^{m - 1}(\mathbb S^1) + \Psi^{-\infty}(\mathbb S^1)\]

We say that $A \in \Psi^{m}_{h, +}(T^*_+ \mathbb S^1)$ is elliptic in the positive semiclassical calculus if 
\begin{equation}\label{eq:elliptic_def}
    \langle \xi \rangle^{-m} |\sigma_h^+(A)(x, \xi)| \ge \delta > 0
\end{equation} 
for all $(x, \xi) \in T^*_+ \mathbb S^1$ uniformly for all small $h$. We can construct parametrices for elliptic operators in the positive semiclassical calculus. It is the same construction as the usual parametrix construction, the only difference being the errors lie in different classes.
\begin{lemma}\label{lem:parametrix}
    Let $A \in \Psi_{h, +}^{m}(\mathbb S^1)$ be elliptic. Then there exists $B \in \Psi_{h, +}^{- m}(\mathbb S^1)$ such that $AB = \Pi^+ + \Psi^{-\infty}$. 
\end{lemma}

\begin{proof}
    Let $b_0 = 1/\sigma_h^+(A) \in S^{-m}(T^*_+ \mathbb S^1)$. Then
    \[\Pi^+ - \Op_h^+ (a) \Op_h^+(b_0) \in h\Psi_h^{-1} + \Psi^{-\infty}.\]
    For $k \ge 1$ put $b_k = 1/\sigma_h^+\left(\Pi^+ - \Op_h^+(a) \Op_h^+ \big(\sum_{j = 0}^{k - 1} b_k \big)\right)$. Note that $b_k \in h S^{-m - k}(T^*_+ \mathbb S^1)$. Put 
    \[b \sim \sum_{k = 0}^\infty b_k.\]
    Then $B = \Op_h^+(b)$ gives the desired operator. 
\end{proof}

\Remarks 1. We can define the negative semiclassical calculus completely analogously with symbols on $T^*_- \mathbb S^1 = \mathbb S^1 \times (-\infty, 0]$. The quantization procedures is given by 
\[\Op_h^-(a_h) = \Op(a^-_h), \qquad a^-_h(x, \xi) := \chi_+(-\xi) a(x, h\xi).\]
The definition of the classes of pseudodiffernetial operators $\Psi^m_{h, +}(\mathbb S^1)$, the composition formula, and the change of variable formula are all identical to the positive case upon replacing $+$'s with $-$'s. 

\noindent
2. Observe that the quantization procedure defined using projection in (\ref{eq:new_quant}) is also valid for the negative semiclassical calculus with $\Pi^- = \Op(\chi_+(-\xi))$ instead of $\Pi^+$. From this quantization procedure, it is clear that \textit{orientation reversing} diffeomorphisms conjugates positive semiclassical operators to negative semiclassical operators, and vice versa. It is also clear then that $AB \in \Psi^{-\infty}$ if $A \in \Psi^m_{h, +}(\mathbb S^1)$ and $B \in \Psi^\ell_{h, -}(\mathbb S^1)$.

\subsection{Multiplication of distributions}
In general, distributions cannot be multiplied. However, when their singularities do not overlap, multiplication is well-defined. This is covered in detail in \cite[Chapter 8.2]{H1} and \cite[Chapter 7]{GS_94}, and we give a brief outline of the important result here. We consider distributions $u \in \mathcal D'(U)$ where $U \subset \R$ or $\mathbb S^1$ is open. Denote the wavefront set of a distribution as $\WF(u) \subset T^* U \setminus \{0\}$, where $\{0\}$ is the zero section. We say that $(x_0, \xi_0) \notin \WF(u)$ if there exists $\varphi \in \CIc(U)$ with $\varphi(x_0)\neq 0$ such that for every $N > 0$, 
\[|\widehat{\varphi u}(\xi)| \le C_N\langle \xi \rangle^{-N}, \quad \text{$\xi \in \R$ such that $\sgn \xi = \sgn \xi_0$}.\]
Clearly $\WF(u)$ is a closed conic subset of $T^* U$. Now let $\Gamma \subset T^* U \setminus \{0\}$ be a closed conic subset. Define
\begin{equation}
    \mathcal D'_{\Gamma}(U) := \{ u \in \mathcal D'(U) \mid \WF(u) \subset \Gamma\}. 
\end{equation}
Equip $\mathcal D'_\Gamma(U)$ with the seminorms
\begin{gather*}
P_\varphi(u) = |\langle \varphi, u \rangle|, \quad \varphi \in \CIc(U), \\
P_{N, \varphi, \pm}(u) = \sup_{\pm \xi > 0} |\widehat{\varphi u} (\xi)| \langle \xi \rangle^N \\
\text{for} \quad N \ge 0, \quad  \varphi \in \CIc(U),\quad (\supp \varphi \times \{\pm \xi > 0\}) \cap \Gamma = \emptyset.
\end{gather*}

Note that $P_\varphi(u)$ are simply the usual seminorms on the space of distributions $\mathcal D'(U)$. The seminorms $P_{N, \varphi, \pm}$ control the parts of the distribution away from $\Gamma$, i.e. the parts that are supposed to be smooth. Now we can state the necessary proposition on multiplication of distributions. 
\begin{proposition}\label{prop:multiplication}
    Let $\Gamma_1, \Gamma_2 \subset T^* U \setminus \{0\}$ be closed conic subsets such that
    \begin{equation}\label{eq:wf_mult}
        \Gamma_1 \cap -\Gamma_2 = \emptyset.
    \end{equation}
    Then the map $(u, v) \mapsto uv$ is well defined and sequentially continuous from $\mathcal D'_{\Gamma_1}(U) \times \mathcal D'_{\Gamma_2}(U) \to \mathcal D'_{\Gamma}(U)$, where 
    \[\Gamma:= \{(x, \xi_1 + \xi_2) \mid (x, \xi_1) \in \Gamma_1, \, (x, \xi_2) \in \Gamma_2\}.\]
\end{proposition}
Here $-\Gamma_2 := \{(x, -\xi) \mid (x, \xi) \in \Gamma_2\}$. The multiplication map is defined by pulling back the tensor $u \otimes v$ to the diagonal, which is done in \cite[Theorem 7.11]{GS_94}. The sequential continuity then follows from the sequential continuity of the tensor product and the pull-back, which are given in \cite[Proposition 7.7 and Corollary 7.9]{GS_94} respectively.

\section{Reduction to boundary}\label{sec:reduction_to_boundary}
In this section we reduce the analysis of $P - \omega^2$ in $\Omega$ to the boundary via single layer potentials. More precisely, we consider a related operator $P(\omega)$ defined in (\ref{eq:P(omega)}), which is constant-coefficient elliptic for $\Re \omega \neq 0$ and degenerates to a constant-coefficient hyperbolic operator as $\omega$ approaches the real line. The point here is that we have explicit fundamental solutions for the operator $P(\omega)$, so the boundary reduced operators can be well-understood.

Much of this section is a compilation of useful lemmas lifted directly from \cite{DWZ}, and they will be liberally referenced for the precise proofs and computations. The crucial additional pieces needed in this paper is a full symbol expansion for the restricted single layer operator so that we can split it into pieces belonging to the positive and negative semiclassical calculi (see Proposition \ref{prop:full_symbol}), and also smooth dependence of the single layer operator on the spectral parameter (see Proposition \ref{prop:S_converge}). 

\subsection{Fundamental solutions}
Define the differential operator
\begin{equation}\label{eq:P(omega)}
    P(\omega):= (1 - \omega^2) \partial_{x_2}^2 - \omega^2 \partial_{x_1}^2, \qquad \omega \in \C, \, 0 < \Re \omega < 1
\end{equation}
on $\R^2_{x_1, x_2}$. Recall that we wish to understand the spectrum of the operator $P$, and formally, $P(\omega)$ is related to $P$ by 
\begin{equation}\label{eq:resolvent_relation}
    P(\omega) = (P - \omega^2) \Delta_\Omega, \qquad P(\omega)^{-1} = \Delta_\Omega^{-1}(P - \omega^2)^{-1}.
\end{equation}
Therefore, studying the resolvent of $P$ is equivalent to studying $P(\omega)^{-1}$. Let us first make a few useful observations. We can factorize $P(\omega)$ as
\begin{equation}\label{eq:L_factor}
    P(\omega) = 4 L^+_\omega L^-_\omega, \quad L^\pm_\omega := \frac{1}{2} (\pm \omega \partial_{x_1} + \sqrt{1 - \omega^2} \partial_{x_2}).
\end{equation}
The square root here is taken on the branch $\C \setminus (-\infty, 0]$. When $\Im \omega \neq 0$, $L^\pm_\omega$ are Cauchy--Riemann type operators. On the other hand, if $\lambda \in (0, 1)$, then $L^\pm_\lambda$ are simply linearly independent vector fields. 

Recall the functions $\ell^\pm(x, \lambda)$ from (\ref{eq:dual_factor}), whose level curves formed the trajectories for the chess billiard flow. We extend the definition for $\lambda \in (0, 1)$ to include $\omega \in (0, 1) + i\R$ by
\begin{equation}\label{eq:ell}
    \ell^\pm(x, \omega):= \pm \frac{x_1}{\omega} + \frac{x_2}{\sqrt{1 - \omega^2}}.
\end{equation}
It is easy to check that 
\begin{equation}\label{eq:L_dual}
    L_\omega^\pm \ell^\pm(x, \omega) = 1, \quad L_\omega^\mp \ell^\pm(x, \omega) = 0.
\end{equation}
This property will be useful in establishing a uniqueness property in Lemma~\ref{lem:bd_uniqueness}.

The operator $P(\omega)$ has the advantage that it has explicit fundamental solutions. Recall that a fundamental solution of the operator $P(\omega)$ as defined in (\ref{eq:P(omega)}) is a distribution $E_\omega \in \mathcal D'(\R^2)$ that satisfies
\[P(\omega) E_\omega = \delta_0.\]

Note that $P(\omega)$ is elliptic when $\Im \omega \neq 0$. However, when $\omega$ is real, $P(\omega)$ degenerates to a hyperbolic operator. In the elliptic case, the fundamental solution is a rescaled version of the standard Newton potential. When $\omega$ approaches the real line, the fundamental solution converges in distribution to one of the Feynmann propagators, depending on if the limit is taken from above or below the real line. The precise formulas are given in the following lemma. 
\begin{lemma}\label{lem:fs}
For $\Im \omega \neq 0$, a fundamental solution of $P(\omega)$ is given by the locally integrable function
\begin{equation}
    E_\omega(x) = c_\omega \log A(x, \omega), \qquad x \in \R^2 \setminus \{0\}
\end{equation}
where
\begin{equation}
    A(x, \omega) := \ell^+(x, \omega) \ell^-(x, \omega) = - \frac{x_1^2}{\omega^2} + \frac{x_2^2}{1 - \omega^2}, \qquad c_\omega := \frac{i \sgn \Im w}{4 \pi \omega \sqrt{1 - \omega^2}}.
\end{equation}
    The distributional limits $E_{\lambda \pm i0} := \lim_{\epsilon \to 0} E_{\lambda \pm i \epsilon}$, given by
\begin{equation}
    E_{\lambda \pm i0}(x) = \pm c_\lambda \log(A(x, \lambda) \pm i0), \qquad c_\lambda := \frac{i}{4 \pi \lambda \sqrt{1 - \lambda^2}}.
\end{equation}
are then fundamental solutions for $P(\lambda)$ for $\lambda \in (0, 1)$. 
\end{lemma}

In fact, we can do much better than just convergence in distribution. In order to understand the spectrum near the forcing frequency $\lambda$ since we will need derivatives in $\lambda$. It is clear that $E_\omega$ is holomorphic on the interior of the regions $(0, 1) \pm i [0, \infty)$. It turns out that $E_\omega$ is also $C^\infty$ up to the boundary in the following sense:
\begin{lemma}\label{lem:fs_holomorphic}
    For every $\varphi \in \CIc(\R^2)$, the map  given by the distributional pairing 
    \[\omega \mapsto (E_\omega, \varphi)\]
    lies in $C^\infty ((0, 1) \pm i[0, \infty))$ and is holomorphic on the interior $(0, 1) \pm i(0, \infty)$. 
\end{lemma}
We refer the reader to \cite[\S4.3]{DWZ} for the proofs of Lemmas \ref{lem:fs} and \ref{lem:fs_holomorphic}.

\subsection{Boundary layer potentials}
Consider the elliptic boundary value problem 
\begin{equation}\label{eq:EBVP}
    P(\omega) u = f, \quad u|_{\partial \Omega} = 0, \quad \Re \omega \in (0, 1), \quad \Im \omega \neq 0.
\end{equation}
We wish to reduce this problem to a problem on the boundary $\partial \Omega$. For $\omega \in (0, 1) \pm i(0, \infty)$, define the convolution operator 
\begin{equation}
    R_\omega: \mathcal E'(\R^2) \to \mathcal D'(\R^2), \quad R_\omega g := E_\omega * g. 
\end{equation}
As $\omega$ approaches the real line from the upper or lower half plane, we can analogously define the limiting operators
\begin{equation}
    R_{\lambda \pm i0}: \mathcal E'(\R^2) \to \mathcal D'(\R^2), \quad R_{\lambda \pm i0} g := E_{\lambda \pm i0} * g, \quad \lambda \in (0, 1). 
\end{equation}
For $\omega \in (0, 1) + i\R$, define the `Neumann data' operator
\begin{equation}\label{eq:neumann_data}
    \mathcal N_\omega: C^\infty(\overline \Omega) \to C^\infty( \partial\Omega; T^* \partial \Omega), \quad \mathcal N_\omega u := -2 \omega \sqrt{1 - \omega^2} \mathbf j^* (L^+_\omega u d\ell^+(\bullet, \omega)).
\end{equation}
Here, $\mathbf j: \partial \Omega \to \overline \Omega$ is the canonical embedding and $\mathbf j^*$ is the pullback on 1-forms. While it may appear that the definition of $\mathcal N_\omega$ makes a preference for whether the derivative is taken in the `$+$' or `$-$' direction, we note that since $u|_{\partial \Omega} = 0$, 
\[0 = \mathbf j^* du = \mathbf j^* (L_\omega^+ u\, d\ell^+ + L_\omega^- u\, d \ell^-) \implies \mathcal N_\omega u = 2 \omega \sqrt{1 - \omega^2} \mathbf j^* (L^-_\omega u\, d\ell^-(\bullet, \omega)),\]
so either sign the same data modulo orientation.

Finally define the `tensor by delta function' operator
\begin{equation}
    \mathcal I: \mathcal D'(\partial \Omega; T^* \partial \Omega) \to \mathcal E'(\R^2), \quad \int_{\R^2} \mathcal I v(x) \varphi(x)\, dx := \int_{\partial \Omega} \varphi v. 
\end{equation}
\begin{lemma}\label{lem:fs_app}
For each $f \in \CIc(\Omega)$, the boundary value problem (\ref{eq:EBVP}) has a unique solution $u_\omega \in C^\infty(\overline \Omega)$ that depends holomorphically on $\omega \in (0, 1) \pm i(0, \infty)$. Furthermore, if $u_\omega \in C^\infty(\overline \Omega)$ is the solution to (\ref{eq:EBVP}) for some $f \in \CIc(\Omega)$, then 
\begin{align}
    P(\omega)(\indic_\Omega u_\omega) &= f - \mathcal I v_\omega \label{eq:fs_setup},\\
    \indic_\Omega u_\omega &= R_\omega f - R_\omega \mathcal I v_\omega. \label{eq:fs_conv}
\end{align}
where $v_\omega = \mathcal N_\omega u_\omega$ is the `Neumann data.'

\end{lemma}
\begin{proof}
    Uniqueness and holomorphic dependence on $\omega$ are proved in \cite[Lemma~4.4]{DWZ}. (\ref{eq:fs_setup}) is a direct computation from (\ref{eq:neumann_data}), which is done in detail in \cite[Lemma~4.5]{DWZ}. Finally, (\ref{eq:fs_conv})  is obtained by convolving (\ref{eq:fs_setup}) by the fundamental solution and using the fact that $\indic_\Omega u$ is a compactly supported distribution. 
\end{proof}

\subsubsection{Single layer potentials}
We now drop the subscript $\omega$ from $u_\omega$ and $v_\omega$ in Lemma~\ref{lem:fs_app} and leave the dependence implicit. There will not be any ambiguity since we will not be looking at the dependence on $\omega$ in this section. Lemma~\ref{lem:fs_app} motivates the definition of the single layer operator $S_\omega: \mathcal D'(\partial \Omega; T^* \partial \Omega) \to \mathcal D'(\Omega)$ by 
\begin{equation}\label{eq:single_layer_op}
    S_\omega v := (R_\omega \mathcal I v)|_\Omega, \quad v \in \mathcal D'(\partial \Omega; T^* \partial \Omega), \quad \omega \in (0, 1) \pm i(0, \infty).
\end{equation}
Similarly, we also define the limiting operators $S_{\lambda \pm i0}: \mathcal D'(\partial \Omega; T^* \partial \Omega) \to \mathcal D'(\Omega)$ by 
\begin{equation}\label{eq:single_layer_op_lim}
    S_{\lambda \pm i0} v := (R_{\lambda \pm i0} \mathcal I v)|_\Omega, \quad v \in \mathcal D'(\partial \Omega; T^* \partial \Omega), \quad \lambda \in (0, 1).
\end{equation}
Rewriting (\ref{eq:fs_conv}) using the single layer operator and restricting to $\Omega$, we have 
\begin{equation}\label{eq:boundary_to_interior}
    u = (R_\omega f)|_\Omega - S_\omega v, 
\end{equation}
where $u \in C^\infty(\Omega)$ is the unique solution to (\ref{eq:EBVP}) extends to a function in $C^\infty(\overline \Omega)$ and $v = \mathcal N_\omega u$. In this case, we see that $S_\omega v \in C^\infty(\overline \Omega)$ since $R_\omega f \in C^\infty(\R^2)$. However, we will be using (\ref{eq:boundary_to_interior}) to recover the solution on the interior $u$ from the boundary data $v$. Therefore, we will need some additional mapping properties from $S_\omega$. 

When $S_\omega$ or $S_{\lambda + i0}$ acts on smooth 1-forms, we have the explicit formula 
\begin{equation}\label{eq:S_explicit}
    S_\omega(f\, d\theta)(x) = \int_{\partial \Omega} E_{\omega}(x - y) f(y) \, d \theta(y), \qquad f \in C^\infty(\partial \Omega), \quad x \in \Omega,
\end{equation}
where $\theta$ is a positively oriented choice of coordinates on $\partial \Omega$. A similar formula for the limiting operator $S_{\lambda \pm i 0}$ holds with $\omega$ replaced with $\lambda + i0$. We then have the following additional mapping property:
\begin{lemma}\label{lem:S_mapping_prop}
    Both $S_\omega$ and $S_{\lambda \pm i0}$ are continuous from $C^\infty(\partial \Omega; T^* \partial \Omega)$ to $C^\infty(\overline{\Omega})$, where $\omega \in (0, 1) \pm i(0, \infty)$ and $\lambda \in (0, 1)$. 
\end{lemma}
For a detailed proof of this theorem see \cite[Lemmas 4.6 -- 4.8]{DWZ}. We stress here that this lemma as stated does \textit{not} give uniform control in $\omega$ of the $C^\infty(\partial \Omega; T^* \partial \Omega) \to C^\infty(\overline{\Omega})$ seminorms. However, smooth dependence of $S_\omega$ on $\omega$ is crucial for us and will be discussed in Proposition \ref{prop:S_converge}.

\subsubsection{Normal derivatives}
Fix a positively oriented choice of coordinates $\theta$ on $\partial \Omega$ and let $\mathbf x: \mathbb S^1 \to \partial \Omega$ be the corresponding coordinate chart. Let $\mathbf v(\theta)$ be an inward-pointing smooth vector field on $\partial \Omega$. 
By Lemma~\ref{lem:S_mapping_prop}, we can define operators $\mathscr N_{\omega, k}^{\pm} : C^\infty(\partial \Omega; T^* \partial \Omega) \to C^\infty(\partial \Omega)$ for $k \in \N$ by 
\begin{equation}\label{eq:normal_restriction}
\begin{aligned}
    \mathscr N_{\omega, k}^{\pm} v (\theta) :=& \frac{(-1)^{k - 1}}{c_\omega (k - 1)!} \lim_{\delta \to 0 +} ((L_\omega^\pm)^k S_\omega v)(\mathbf x(\theta) + \delta \mathbf v(\theta)) \\
    =& \lim_{\delta \to 0 +} \int_{\mathbb S^1} \frac{f(\theta)}{\ell^\pm(\mathbf x(\theta) - \mathbf x(\theta') + \delta \mathbf v(\theta), \omega)^{k}} \, d\theta'
\end{aligned}
\end{equation}
where $v = f d\theta \in C^\infty(\partial \Omega; T^* \partial \Omega)$. The second equality above follows immediately from~\eqref{eq:L_dual} and~\eqref{eq:S_explicit}. Note that up to a constant, $\mathscr N_{\omega, k}^\pm$ is simply the $k$-th derivative by $L^\pm_\omega$ of $S_\omega v$ restricted to the boundary from the interior of $\Omega$. It is clear that $\mathscr N_{\omega, k}^\pm$ is independent of the choice of the vector field $\mathbf v$. 

The goal of this small section is to show that for $v \in C^\infty(\partial \Omega; T^* \partial \Omega)$, the smooth functions $\mathscr N_{\omega, k}^\pm v$ are locally uniformly bounded in $\omega$ near the real line. This will in fact give us uniform control over \emph{all} the derivatives of $S_\omega v$ at the boundary, which in turn gives uniform control over $S_\omega v$ on the interior of $\Omega$ later in Proposition~\ref{prop:S_converge}. 

Let us first characterize the Schwartz kernel of $\mathscr N_{\omega, k}^\pm$. With the fixed choice of positively oriented coordinates $\theta$, we can view $\mathscr N_{\omega, k}^\pm$ as an operator on $C^\infty(\mathbb S^1)$. We can see from \eqref{eq:normal_restriction} that the Schwartz kernel lying in $\mathcal D'(\mathbb S^1 \times \mathbb S^1)$ is given by 
\begin{equation}\label{eq:normal_kernel}
    K_{\omega, k}^\pm(\theta, \theta') := \lim_{\delta \to 0+} \ell^\pm(\mathbf x(\theta) - \mathbf x(\theta') + \delta \mathbf v(\theta), \omega)^{-k}.
\end{equation}
The following fact about distributions that follows from the Malgrange Preparation Theorem \cite[Chapter 7.5]{H1} will help simplify the analysis of the Schwartz kernel.
\begin{lemma}\label{lem:malgrange}
    Let $U \subset \R$ be an open interval containing $0$ and let 
    \[\psi \in C^\infty(U \times [0, \epsilon_0); \C), \quad \Re \psi > 0 \quad \text{on} \quad U \times [0, \epsilon_0).\]
    Then 
    \[(x \pm i \epsilon \psi(x, \epsilon))^{-1} \to (x \pm i0)^{-1} \quad \text{in} \quad \mathcal D'_\Gamma(U), \quad \Gamma := \{(0, \xi) \mid \pm \xi > 0\}\]
    as $\epsilon \to 0 +$. 
\end{lemma} 
The proof is nearly identical to \cite[Lemma~3.7]{DWZ}. It follows the exact same computation, but in this case, we just specify the wavefront set in the convergence statement so that we can multiply such distributions later. 
\begin{proof}
    By \cite[Lemma~3.6]{DWZ} (which is proved by the Malgrange Preparation Theorem), there exist $r^\pm, z^\pm \in C^\infty([0, \epsilon_0))$ and $q^\pm \in C^\infty(U \times [0, \epsilon_0))$ such that 
    \begin{equation}
        (x \pm i \epsilon \psi(x, \epsilon))^{-1} = r^\pm(\epsilon)(x \pm i \epsilon z^\pm(\epsilon))^{-1} + q^\pm(x, \epsilon), 
    \end{equation}
    where $\Re z^\pm > 0$, $r^\pm(0) = 1$, and $q^\pm(x, 0) = 0$. Observe that 
    \[(x \pm i \epsilon z^\pm(\epsilon))^{-1} \to (x \pm i0)^{-1} \quad \text{in} \quad \mathcal D'_\Gamma(U), \quad \Gamma := \{(0, \xi): \pm \xi > 0\},\]
    since the Fourier transform of the right-hand side is $\mp 2 \pi i H(\pm \xi)$ and the Fourier transform of the left-hand side is given by $\mp2 \pi i H(\pm \xi) e^{-\epsilon z^\pm(\epsilon)|\xi|}$. In particular, note that the Fourier transforms of both sides are supported on $\{\pm \xi \ge 0\}$, so convergence in $\mathcal D_\Gamma'(U)$ follows immediately. 
\end{proof}
Define the sets 
\begin{equation}\label{eq:normal_kernel_regions}
    \begin{gathered}
        \mathrm{Diag} := \{(\theta, \theta') \mid \theta \in \mathbb S^1\}, \\
        \mathrm{Ref}_{\lambda}^\pm := \{(\theta, \gamma_\lambda^\pm(\theta))\mid \theta \in \mathbb S^1\}.
    \end{gathered}
\end{equation}
In the next four lemmas, we examine the Schwartz kernel in the following regions:
\begin{enumerate}
    \item away from $(\mathrm{Diag} \cup \mathrm{Ref}_{\lambda}^\pm)$, we will see that the Schwartz kernel $K_{\omega, k}^\pm$ is nonsingular, see Lemma~\ref{lem:normal_1},

    \item near $\mathrm{Diag}$ but away from $(\mathrm{Diag} \cap \mathrm{Ref}_{\lambda}^\pm)$, we will see that the Schwartz kernel is a smooth multiple of $$(\theta - \theta' \pm i0)^{-k},$$ see Lemma~\ref{lem:normal_2},

    \item near $\mathrm{Ref}_\lambda^\pm$ but away from $(\mathrm{Diag} \cap \mathrm{Ref}_{\lambda}^\pm)$, the Schwartz kernel approaches a smooth multiple of $$(\gamma^\pm_\lambda(\theta) - \theta' \pm i0)^{-k}$$ in the $h \to 0+$ limit where $\omega = \lambda \pm ih$, see Lemma~\ref{lem:normal_3},

    \item near $\mathrm{Diag} \cap \mathrm{Ref}_{\lambda}^\pm$, the Schwartz kernel approaches a smooth multiple of $$(\theta - \theta' \pm i0)^{-k}(\gamma^\pm_\lambda(\theta) - \theta' \pm i0)^{-k}$$ in the $h \to 0+$ limit where $\omega = \lambda \pm ih$, see Lemma~\ref{lem:normal_4}.
\end{enumerate}
The proofs of these lemmas follow very similar computations as \cite[Lemmas~4.11--4.14]{DWZ}, but applied to slightly different operators.

\begin{lemma}\label{lem:normal_1}
    Let $\lambda_0 \in (0, 1)$ be such that $\Omega$ is $\lambda_0$-simple. Let $$(\theta_0, \theta'_0) \in (\mathbb S^1 \times \mathbb S^1) \setminus (\mathrm{Diag} \cup \mathrm{Ref}_{\lambda_0}^\pm).$$ Then there exists a neighborhood $\mathcal J$ of $\lambda_0$ and $U$ of $(\theta_0, \theta_0')$ such that
    \[K_{\omega, k}^\pm |_{\overline U} \in C^\infty (\overline U)\]
    smoothly and uniformly in $\omega$ for $\omega \in \mathcal J + i(0, \infty)$.
\end{lemma}
\begin{proof}
    Let $U$ be a small neighborhood of $(\theta_0, \theta'_0)$ that does not intersect $(\mathrm{Diag} \cup \mathrm{Ref}_{\lambda_0}^\pm)$. For sufficiently small neighborhood $\mathcal J$ of $\lambda_0$, $\Omega$ is $\lambda$-simple and $U$ does not intersect $(\mathrm{Diag} \cup \mathrm{Ref}_{\lambda_0}^\pm)$ for all $\lambda \in \mathcal J$. Note that $\mathbf x(\theta) - \mathbf x(\theta') \neq 0$ and $\mathbf x(\theta) - \mathbf x(\gamma_\lambda^\pm(\theta')) \neq 0$ for $(\theta, \theta') \in U$ and $\lambda \in \mathcal J$. This means that $\ell^\pm(\mathbf x(\theta) - \mathbf x(\theta'), \omega)$ is uniformly nonvanishing for $(\theta, \theta') \in U$ and $\omega \in \mathcal J + i(0, \infty)$. Then the lemma follows immediately from~\eqref{eq:normal_kernel}. 
\end{proof}

\begin{lemma}\label{lem:normal_2}
    Let $\lambda_0 \in (0, 1)$ be such that $\Omega$ is $\lambda_0$-simple. Let $\theta_0 \in \mathbb S^1$ such that $\gamma_{\lambda_0}^\pm (\theta_0) \neq \theta_0$. Then there exist sufficiently small neighborhoods $U \ni \theta_0$ and $\mathcal J \ni \lambda_0$, and a sufficiently small $h_0 > 0$, such that for $\theta, \, \theta' \in U$ and $\omega \in \mathcal J + i(0, h_0)$, we have
    \begin{equation*}
        K_{\omega, k}^\pm(\theta, \theta') = c_{\omega, k}^\pm(\theta, \theta')(\theta - \theta' \pm i0)^{-k}
    \end{equation*}
    where $c_{\omega, k}^\pm(\theta, \theta')$ is smooth in $\theta, \theta', \omega$ and extends to a smooth function up to $\Im \omega = 0$. 
\end{lemma}
\begin{proof}
    For $U \ni \theta_0$ and $\mathcal J \ni \theta$ sufficiently small, $\Omega$ is $\lambda$-simple and $\gamma_\lambda^\pm(\theta) \neq \theta$ for all $\theta \in U$ and $\lambda \in \mathcal J$. Up to further shrinking $U$, we have that for $h_0$ sufficiently small, we can write
    \begin{equation*}
        \ell^\pm(\mathbf x(\theta) - \mathbf x(\theta')) = G_\omega^\pm(\theta, \theta')(\theta - \theta'), \quad \theta, \theta' \in U, \quad \omega \in \mathcal J + i(0, h_0)
    \end{equation*}
    where $G_\omega^\pm$ is a nonvanishing smooth function of $\theta, \theta', \omega$ up to $\Im \omega = 0$ with 
    \begin{equation*}
        G_\omega^\pm(\theta, \theta) = \partial_\theta \ell^\pm(\mathbf x(\theta), \omega).
    \end{equation*}
    Then we can write \eqref{eq:normal_kernel} locally as
    \begin{equation}
    \begin{gathered}
        K_{\omega, k}^\pm(\theta, \theta') = G_\omega^\pm(\theta, \theta')^{-k} \lim_{\delta \to 0 +} \left(\theta - \theta' + \delta \varphi_\omega^\pm(\theta, \theta') \right)^{-k}, \\
        \text{where} \quad \varphi_\omega^\pm (\theta, \theta') := \frac{\ell^\pm( \mathbf v(\theta), \omega)}{G_\omega^\pm(\theta, \theta')}.
    \end{gathered}
    \end{equation}
    To evaluate the limit, we use Lemma~\ref{lem:malgrange}, so we need to verify $\pm \Im \varphi^\pm_\omega(\theta, \theta') > 0$ for $\omega \in U + i(0, h_0)$ and $\theta, \theta' \in U$. Observe that $\Im \varphi_\lambda^\pm(\theta, \theta') = 0$ for $\lambda \in U$. Therefore, up to shrinking $U$, it suffices to check
    \begin{equation}
        \pm \partial_h|_{h = 0} \Im \varphi_{\lambda + ih}^\pm(\theta_0, \theta_0) = \pm \partial_h|_{h = 0} \Im \frac{\ell^\pm(\mathbf v(\theta), \lambda + ih)}{\partial_\theta \ell^\pm(\mathbf x(\theta), \lambda + ih)} > 0.
    \end{equation}
    This is directly verified in \cite[(4.56)]{DWZ}, and the lemma follows with $c_{\omega, k}^\pm := G_\omega^\pm(\theta, \theta')^{-k}$. 
\end{proof}

\begin{lemma}\label{lem:normal_3}
    Let $\lambda_0 \in (0, 1)$ be such that $\Omega$ is $\lambda_0$-simple. Let $\theta_0 \in \mathbb S^1$ such that $\gamma_{\lambda_0}^\pm (\theta_0) \neq \theta_0$. Then there exist sufficiently small neighborhoods $U \ni \theta_0$ and $\mathcal J \ni \lambda_0$, and a sufficiently small $h_0 > 0$, such that for $\theta \in U$, $\theta' \in \gamma_{\lambda_0}^\pm(U)$, and $\omega = \lambda + ih \in \mathcal J + i(0, h_0)$, we have
    \begin{equation*}
        K_{\omega, k}^\pm(\theta, \theta') = c_{\lambda, k}^\pm(\theta, \theta')\left(\gamma_\lambda^\pm(\theta) - \theta' \pm ih \psi_\omega^\pm(\theta, \theta')\right)^{-k}
    \end{equation*}
    where $c_{\lambda, k}^\pm(\theta, \theta')$ is smooth in $\theta, \theta', \lambda$, and $\psi_\omega^\pm(\theta, \theta')$ is smooth in $\theta, \theta', \omega$ and extends to a smooth function up to $\Im \omega = 0$. Furthermore, 
    \begin{equation}\label{eq:psi_lower_bound}
        \Re \psi_\omega^\pm(\theta, \theta') \ge \delta > 0.
    \end{equation}
\end{lemma}
\begin{proof}
Taylor expanding $\ell^\pm(x, \lambda + ih)$ in $h$ at $h = 0$, we see that 
\begin{equation}\label{eq:ell_taylor}
    \ell^\pm(x, \omega) = \ell^\pm(x, \lambda) + ih \ell_1^\pm(x, \lambda) + h^2 \ell_2^\pm(x, \omega)
\end{equation}
where, by the holomorphy of $\ell_1^\pm(x, \omega)$ in $\omega$, we have
\begin{equation}
    \ell_1^\pm(x, \lambda):= \partial_\lambda \ell^\pm(x, \lambda)
\end{equation}
and $\ell_2^\pm(x, \omega)$ is smooth in $x$ and $\omega = \lambda + ih$ up to $h = 0$.  

For a sufficiently small neighborhoods $U \ni \theta_0$ and $\mathcal J \ni \lambda_0$, $\Omega$ is $\lambda$-simple and $\gamma_\lambda^\pm(\theta) \neq \theta$ for all $\lambda \in \mathcal J$ and $\theta \in U$. Recall that $\ell^\pm(\mathbf x(\theta), \lambda) = \ell^\pm(\mathbf x(\gamma_\lambda^\pm(\theta)), \lambda)$. So up to further shrinking $U$
\begin{equation*}
    \ell^\pm(\mathbf x(\theta) - \mathbf x(\theta'), \lambda) = G_\lambda^\pm(\theta, \theta') (\gamma_\lambda^\pm(\theta) - \theta'), \quad \theta \in U, \quad \theta' \in \gamma_{\lambda_0}^\pm(U),
\end{equation*}
where $G_\lambda^\pm(\theta, \theta')$ is real and smooth in $\theta, \theta', \lambda$, and $G_\lambda^\pm(\gamma_\lambda^\pm(\theta'), \theta') = \partial_{\theta'} \ell^\pm(\mathbf x(\theta'), \lambda)$. So $G_\lambda^\pm$ is nonvanishing for $(\theta, \theta') \in U \times \gamma_\lambda^\pm(U)$ for a sufficiently small neighborhood $U$ of $\theta_0$. Put
\begin{equation}
    \psi_\omega^\pm(\theta, \theta') := \pm \frac{\ell_1^\pm(\mathbf x(\theta) - \mathbf x (\theta'), \lambda) - ih \ell_2^\pm(\mathbf x(\theta) - \mathbf x(\theta'), \omega)}{G_\lambda^\pm(\theta, \theta')}.
\end{equation}
Once we show that $\psi_\omega^\pm$ satisfies \eqref{eq:psi_lower_bound}, we will have
\begin{equation*}
    K_{\omega, k}^\pm(\theta, \theta') = G_\lambda^\pm (\theta, \theta')^{-k}  \left( \gamma^\pm_\lambda(\theta) - \theta' \pm ih \psi_\omega^\pm(\theta, \theta') \right)^{-k}.
\end{equation*}
It remains to check \eqref{eq:psi_lower_bound}, which is verified in \cite[Lemma~4.13, Step 2]{DWZ}.
\end{proof}

\begin{lemma}\label{lem:normal_4}
    Let $\lambda_0 \in (0, 1)$ be such that $\Omega$ is $\lambda_0$-simple. Let $(\theta_0, \theta_0) \in \mathrm{Diag} \cap \mathrm{Ref}_{\lambda_0}^\pm$. Then there exist sufficiently small neighborhoods $U \subset \mathbb S^1 \times \mathbb S^1$ of $(\theta_0, \theta_0)$ and $\mathcal J \ni \lambda_0$, and a sufficiently small $h_0 > 0$, such that for $(\theta, \theta') \in U$ and $\omega = \lambda + ih \in \mathcal J + i(0, h_0)$, we have
    \begin{equation*}
        K_{\omega, k}^\pm(\theta, \theta') = c_{\lambda, k}^\pm(\theta, \theta')(\theta - \theta' \pm i0)^{-k} \left(\gamma^\pm_\lambda(\theta) - \theta' \pm ih \psi_\omega^\pm(\theta, \theta')\right)^{-k}
    \end{equation*}
    where $c_{\lambda, k}^\pm(\theta, \theta')$ is smooth in $\theta, \theta', \lambda$, and $\psi_\omega^\pm(\theta, \theta')$ is smooth in $\theta, \theta', \omega$ and extends to a smooth functions up to $\Im \omega = 0$. Furthermore, 
    \begin{equation}\label{eq:psi_lower_bound_1}
        \Re \psi_\omega^\pm(\theta, \theta') \ge \delta > 0.
    \end{equation} 
\end{lemma}
\begin{proof}
We write as in \eqref{eq:ell_taylor} 
\begin{equation*}
    \ell^\pm(x, \omega) = \ell^\pm(x, \lambda) + ih \ell_1^\pm(x, \lambda) + h^2 \ell_2^\pm(x, \omega).
\end{equation*}
Choose $\mathcal J \ni \lambda_0$ so that $\Omega$ is $\lambda$-simple for every $\lambda \in \Omega$. For $(\theta, \theta')$ in a sufficiently small neighborhood $U$ of $(\theta_0, \theta_0)$, we can write
\begin{align*}
    \ell^\pm(\mathbf x(\theta) - \mathbf x(\theta'), \lambda) &= G_0(\theta, \theta', \lambda)(\gamma_\lambda^\pm(\theta) - \theta') (\theta - \theta') \\
    \ell_1^\pm(\mathbf x(\theta) - \mathbf x(\theta'), \lambda) &= G_1(\theta, \theta', \lambda)(\theta - \theta'), \\
    \ell_2^\pm(\mathbf x(\theta) - \mathbf x(\theta'), \lambda) &= G_2(\theta, \theta', \omega)(\theta - \theta'),
\end{align*}
where $G_0$, $G_1$, $G_2$ are nonvanishing and smooth in each of their arguments up to $\Im \omega = 0$, and $G_0, G_1$ are real. We will henceforth omit writing the arguments for $G_j$ for simplicity. To compute the kernel \eqref{eq:normal_kernel} near $U \times U$, first observe that 
\begin{equation}
    \ell^\pm(\mathbf x(\theta) - \mathbf x(\theta') + \delta \mathbf v(\theta), \omega) = G_0 \big(\theta - \theta' + \delta \varphi_\omega^\pm(\theta, \theta')\big) \big(\gamma_\lambda^\pm(\theta) - \theta' + ih \psi_\omega^\pm(\theta, \theta') \big)
\end{equation}
where
\begin{equation}
    \varphi_\omega^\pm(\theta, \theta') := \frac{\ell^\pm(\mathbf v(\theta), \omega)}{(\gamma_\lambda^\pm(\theta) - \theta') G_0 + ihG_1 + h^2 G_2}, \quad \psi_\omega^\pm(\theta, \theta') := \frac{G_1 - ihG_2}{G_0}. 
\end{equation}
By \cite[(4.67)]{DWZ}, $\pm \Im \varphi_\omega^\pm(\theta, \theta') < 0$ for $h > 0$ sufficiently small. Furthermore, $\psi_\omega^\pm(\theta, \theta)$ satisfies \eqref{eq:psi_lower_bound_1} by \cite[Step 3]{DWZ}. 
\end{proof}

Now that we have a complete description of the Schwartz kernel of $\mathscr N_{\omega, k}^\pm$, we obtain the following mapping property:
\begin{proposition}\label{prop:normal}
    Let $\lambda_0$ be such that $\Omega$ is $\lambda_0$-simple and let $k \in N$. Assume that $v \in C^\infty(\partial \Omega; T^* \partial \Omega)$. There exists a small neighborhood $\mathcal J \ni \lambda_0$, $h_0 > 0$, and some $s > 0$ such that
    \begin{equation*}
        \|\mathscr N_{\omega, k}^\pm v\|_{L^\infty} \le C\|v\|_{H^s}
    \end{equation*}
    for $\omega \in \mathcal J + i(0, h_0)$, where the constant $C$ is independent of $\omega$.  
\end{proposition}
\Remark The uniform bound above also holds for $\omega \in \mathcal J - i(0, h_0)$ by the same argument. Also, from differentiating the proof below, it is also clear that we have uniform bound in $C^k$ rather than $L^\infty$, but this will not be needed. 
\begin{proof}
    There exists a small interval $\mathcal J$ of $\lambda_0$, $h_0 > 0$, and a finite partition of unity over $\mathbb S^1 \times \mathbb S^1$ given by 
    \[\sum_{j = 1}^N \chi_{j}^2 = 1, \quad \chi_j \in \CIc(\mathbb S^1 \times \mathbb S^1),\]
    such that each $\chi_j$ is supported in a set over which $K_{\omega, k}^\pm$ can be written in the form of one of Lemmas~\ref{lem:normal_1}, \ref{lem:normal_2}, \ref{lem:normal_3}, or \ref{lem:normal_4} for all $\omega \in \mathcal J + i(0, h_0)$. It suffices to show that the distributional pullback (see \cite[Corollary 7.9]{GS_94})
    \begin{equation}\label{eq:kernel_pullback}
        \chi_j^2(\theta, \bullet) K_{\omega, k}^\pm(\theta, \bullet) = \iota_\theta^* \left(\chi_j^2(\theta, \bullet) K_{\omega, k}^\pm \right), \quad \iota_\theta(\theta') := (\theta, \theta')
    \end{equation} 
    is well-defined and is uniformly bounded in $\mathcal D'(\mathbb S^1)$ for every $j = 1, \dots, N$, $\theta \in \mathbb S^1$, and $\omega \in \mathcal J \pm i(0, h_0)$ for some $\mathcal J \ni \lambda_0$ and $h_0 > 0$.
    
    In the case that the support of $\chi_j$ falls into the case of Lemma~\ref{lem:normal_1} and Lemma~\ref{lem:normal_2} the pulled-back cutoff Schwartz kernel \eqref{eq:kernel_pullback} is clearly uniformly bounded in distribution. In the case that the support of $\chi_j$ falls into the case of Lemma~\ref{lem:normal_3}, then Lemma~\ref{lem:malgrange} implies that \eqref{eq:kernel_pullback} is uniformly bounded in distribution.

    It remains to consider the case that the support of $\chi_j$ falls into the case of Lemma~\ref{lem:normal_4}. We then have the formula
    \[\chi_j^2(\theta, \theta') K_{\omega, k}^\pm(\theta, \theta') = \chi_j^2(\theta, \theta') c_{\lambda, k}^\pm(\theta, \theta')(\theta - \theta' \pm i0)^{-k} \left(\theta - \theta' \pm ih \psi_\omega^\pm(\theta, \theta')\right)^{-k}.\]
    We may assume that the support of $\chi_j$ is contained in a small neighborhood $U$ of $\theta_0 \in \mathrm{Diag} \cap \mathrm{Ref}_{\lambda_0}^\pm$. It follows from Lemma~\ref{lem:malgrange} that
    \begin{equation*}
    \begin{gathered}
        \chi(\theta, \bullet) (\gamma_\lambda^\pm(\theta) - \bullet \pm ih \psi_\omega^\pm(\theta, \bullet))^{-k} \in \mathcal D'_{\Gamma^\mp} (U), \\
        \chi(\theta, \bullet)(\theta - \bullet \pm i0)^{-k} \in \mathcal D'_{\Gamma^\mp}(U), \\
        \Gamma^\pm := \{(\theta', \eta) \mid \theta' \in U, \, \pm \eta \ge 0\}
    \end{gathered}
    \end{equation*}
    uniformly for $\omega \in \mathcal J + i(0, h_0)$. These two families of distributions also satisfies the wavefront set condition~\eqref{eq:wf_mult} since $\Gamma^\pm\cap (-\Gamma^\pm)=\emptyset$. By Proposition \ref{prop:multiplication}, these distributions can be multiplied sequentially continuously. The desired uniform bound of~\eqref{eq:kernel_pullback} in distributions then follows immediately. 
\end{proof}


\subsubsection{Restricted single layer potential}\label{subsubsec:normal}
Above, we studied the mapping properties of the normal derivatives of $S_\omega v$ at the boundary. For just $S_\omega v$ restricted to the boundary, we will need a much more detailed microlocal description. For $v \in C^\infty(\partial \Omega; T^* \partial \Omega)$, Lemma~\ref{lem:S_mapping_prop} enables us to take the boundary trace of $S_\omega v \in C^\infty(\overline \Omega)$. Define the operator $\mathcal C_\omega: C^\infty(\partial \Omega; T^* \partial \Omega) \to C^\infty(\partial \Omega)$ by
\begin{equation}
    \mathcal C_\omega v := (S_\omega v)|_{\partial \Omega}, \quad \omega \in (0, 1) \pm i(0, \infty).
\end{equation}
We refer to $\mathcal C_\omega$ as the \textit{restricted single layer operator}. Similarly, we can define the limiting operators $\mathcal C_{\lambda \pm i 0}: C^\infty(\partial \Omega; T^* \partial \Omega) \to C^\infty(\partial \Omega)$ for $\lambda \in (0, 1)$ by simply replacing $\omega$ by $\lambda \pm i0$ in the definition of $\mathcal C_\omega$.

Since $u|_{\partial \Omega} = 0$, it follows from (\ref{eq:fs_conv}) that the restricted single layer operator satisfies
\begin{equation}\label{eq:restricted_problem}
    \mathcal C_\omega v = (R_\omega f)|_{\partial \Omega}
\end{equation}
where $v = \mathcal N_\omega u$ and $u$ is the unique solution to the boundary value problem (\ref{eq:EBVP}).

Composing $\mathcal C_\omega$ (or $\mathcal C_{\lambda + i0}$) with the differential $d: C^\infty(\partial \Omega) \to C^\infty(\partial \Omega; T^* \partial \Omega)$, we have
\[d\mathcal C_\omega: C^\infty(\partial \Omega; T^* \partial \Omega) \to C^\infty(\partial \Omega; T^* \partial \Omega).\]
Upon fixing a choice of positively oriented coordinates $\theta \in \mathbb S^1$, we can treat $d \mathcal C_\omega$ as an operator on $C^\infty(\partial \Omega; T^* \mathbb S^1)$ and also identify 1-forms on $\partial \Omega$ with functions on $\mathbb S^1$ by $v = f(\theta)\, d \theta$. 

The complete microlocal description of $d \mathcal C_\omega$ is already done in \cite[\S4.6]{DWZ}. What will be important for us is a full asymptotic expansion for each of the pseudodifferential pieces in the decomposition of $d\mathcal C_\omega$, which will allow us to understand $d\mathcal C_\omega$ in terms of the positive and negative semiclassical calculus developed in Section~\ref{subsubsec:positive_semiclassical}. Here, we will simply give a summary of the decomposition done in \cite[Section~4.6]{DWZ} and then compute the full symbols in detail. 

First, $d \mathcal C_\omega$ can be separated into positive and negative frequency pieces by using the vector fields $L^\pm_\omega$ and the dual linear functions $\ell^\pm(x, \omega)$. Observe that 
\begin{equation}\label{eq:dC_decomp}
    d \mathcal C_\omega = T_\omega^+ + T_\omega^- \quad \text{where} \quad T_\omega^\pm v = \mathbf j^*\big((L^\pm_\omega S_\omega v ) d \ell^\pm \big),
\end{equation}
where again $\mathbf j: \partial \Omega \to \Omega$ is the canonical embedding. We also have the mapping property $T^\pm_\omega: C^\infty(\partial \Omega; T^* \partial \Omega) \to C^\infty(\partial \Omega; T^* \partial \Omega)$.


Essentially, $T_\omega^\pm$ can be understood by considering each of the four cases in Section~\ref{subsubsec:normal}. What one finds is that $T^\pm_\omega$ can be decomposed into the sum of a pseudodifferential operator and a pulled-back pseudodifferential operator, whose Schwartz kernels have singularities along $\mathrm{Diag}$ and $\mathrm{Ref}_{\lambda}^\pm$ respectively (see \eqref{eq:normal_kernel_regions}). More precisely, we have the following lemma, compiled from results in \cite[\S4.6]{DWZ}. Fix a small subinterval $\mathcal J \Subset (0, 1)$ such that $\Omega$ is $\lambda$-simple for every $\lambda \in J$.
\begin{lemma}\label{eq:T_diag_ref}
    Let $\omega = \lambda + ih \in \mathcal J + i(0, \infty)$. Then $T^\pm_\omega$ has the following decomposition:
    \begin{equation}\label{eq:T_decomp}
        T^\pm_\omega = T^\pm_{\mathrm{Diag}} + (\gamma^\pm_{\lambda})^* \widehat T_{\mathrm{Ref}}^\pm
    \end{equation}
    where $T^\pm_{\mathrm{Diag}}$ and $\widehat T^\pm_{\mathrm{Ref}}$ are $\omega$-dependent families of pseudodifferential operators with Schwartz kernel $K_{\mathrm{Diag}}^\pm$ and $\widehat K_{\mathrm{Ref}}^\pm$ respectively. Furthermore, the Schwartz kernels have explicit formulas 
    \begin{equation}\label{eq:2_K_kernels}
        \begin{gathered}
            K_{\mathrm{Diag}}^\pm \equiv c_\omega \chi_{\Diag}(\theta, \theta')(\theta - \theta' \pm i0)^{-1}\\
            \widehat K_{\mathrm{Ref}}^\pm \equiv \chi_{\mathrm{Diag}}(\theta, \theta') \tilde c_\omega^+(\theta') (\partial_\theta \gamma_\lambda^+(\theta)) (\theta - \theta' + i \epsilon z_\omega^+ (\theta'))^{-1}
        \end{gathered}
    \end{equation}
    where `$\equiv$' denotes equivalence modulo smoothing functions uniform in $\omega$. Here, 
    \begin{enumerate}
    \item $\chi_{\mathrm{Diag}} \in C^\infty(\mathbb S^1 \times \mathbb S^1)$ is a cutoff function that is 1 near the set $\mathrm{Diag}$ and is supported in a small neighborhood of $\mathrm{Diag}$. 
    
    \item $z_\omega^\pm(\theta) \in C^\infty(\mathbb S^1)$ is smooth in $\omega = \lambda + ih$ up to $h = 0$. Furthermore $\Re z_\omega^\pm (\theta) \ge \ell > 0$ where $\ell$ is independent of $h$ and $\theta$. 
    
    \item $\tilde c_\omega^\pm(\theta) \in C^\infty(\mathbb S^1)$ is also smooth in $\omega$ up to $h = 0$, and 
    \begin{equation}
        \tilde c_\omega^\pm(\theta') = \frac{c_\omega}{\partial_{\theta'} \gamma^\pm_\lambda(\theta')} + \mathcal O(h).
    \end{equation}
\end{enumerate}
\end{lemma}
We only need the explicit formulas for $K_{\mathrm{Diag}}^\pm$ and $\widehat K_{\mathrm{Diag}}^\pm$ up to a uniform smoothing error since the asymptotic expansion for the full symbol will be up to a uniform smoothing error as well.

\subsection{Full symbols}
Now we give the full symbol expansion of the operators $T^\pm_\omega$, which will show that $d \mathcal C_\omega$ can be understood using the positive and negative semiclassical calculi developed in \S\ref{subsubsec:positive_semiclassical}. In particular, we show that $d\mathcal C_\omega$ can essentially be written as identity plus a positive semiclassical operator and a negative semiclassical operator.

\begin{proposition}\label{prop:full_symbol}
    Let $\omega = \lambda + i h$ for $\lambda \in (0, 1)$ and $h \ge 0$. Assume that $\Omega$ is $\lambda$-simple. Then 
\begin{equation}\label{eq:EdC=}
    \mathcal E_\omega d \mathcal C_\omega = I + (\gamma_\lambda^+)^* A^+_\omega + (\gamma^-_{\lambda})^* A^-_\omega
\end{equation}
where $\mathcal E_\omega \in \Psi^0(\mathbb S^1; T^* \mathbb S^1)$ uniformly bounded in $h$, and $A_\omega^\pm \in \Psi_{h, \pm}^{-\infty}$. Furthermore, we have the positive (and negative) semiclassical principal symbol 
\begin{equation}\label{eq:A_w_principal}
    \sigma_h^{\pm}(A^\pm_\omega)(\theta, \xi) = -e^{- z_\omega^\pm(\theta)|\xi|}
\end{equation}
where $z_\omega^\pm(\theta) \in C^\infty(\mathbb S^1)$ is given in Lemma~\ref{eq:T_diag_ref}.
\end{proposition}


\begin{proof}
1. Let us first derive the expression (\ref{eq:EdC=}). From (\ref{eq:dC_decomp}) and (\ref{eq:T_decomp}), we get the decomposition 
\begin{equation}\label{eq:dC_to_T_decomp}
    d\mathcal C_\omega = T^+_{\mathrm{Diag}} + T^-_{\mathrm{Diag}} + (\gamma_\lambda^+)^* \widehat T_{\mathrm{Ref}}^+ + (\gamma_\lambda^-)^* \widehat T^-_{\mathrm{Ref}}.
\end{equation}
Let
\[T_{\text{Diag}} = T^+_{\text{Diag}} + T^-_{\text{Diag}},\]
which by \eqref{eq:2_K_kernels} has Schwartz kernel
\begin{align*}
    K_{\mathrm{Diag}}(\theta, \theta') &\equiv c_\omega \chi_{\Diag}(\theta, \theta')\big[(\theta + \theta' \pm i0)^{-1} + (\theta - \theta' - i0)^{-1}\big] \\
    &= 2c_\omega \chi_{\Diag}(\theta, \theta') \mathrm{P.V.}(\theta - \theta')^{-1}.
\end{align*}
This can be written as an oscillatory integral as 
\[K_{\mathrm{Diag}}(\theta, \theta') = \frac{1}{2 \pi} \int_{\mathbb R} e^{i (\theta - \theta') \xi} a_{\mathrm{Diag}}(\theta, \theta', \xi)\, d \xi\]
with the symbol
\[a_{\mathrm{Diag}}(\theta, \theta', \xi) = -2 \pi i \chi(\theta, \theta') c_\omega \sgn(\xi).\]
See \cite[\S3.4]{DWZ} for details. Using the left reduction formula (\ref{eq:left_reduction}), we see that the full left symbol is just
\[a_{\Diag}(\theta, \xi) = -2 \pi i c_\omega \sgn(\xi).\]
This is just a multiple of the Hilbert transform on the circle (up to a smoothing error uniform in $h$), so the full symbol of the parametrix $\mathcal E_\omega$ is simply given by 
\begin{equation}\label{eq:E_omega_full}
    \sigma(\mathcal E_\omega)(\theta, \xi) = \frac{i}{2 \pi c_\omega} \sgn(\xi).
\end{equation}
Applying $\mathcal E_\omega$ to (\ref{eq:dC_to_T_decomp}), we find that 
\begin{equation*}
    \mathcal E_\omega d\mathcal C_\omega \equiv I + (\gamma_\lambda^+)^* \big((\gamma_\lambda^+)^* \mathcal E_\omega (\gamma_\lambda^+)^*\big) \widehat T_{\mathrm{Ref}}^+ + (\gamma_\lambda^-)^*\big((\gamma_\lambda^-)^* \mathcal E_\lambda (\gamma_\lambda^-)^*\big) \widehat T^-_{\mathrm{Ref}}.
\end{equation*}
Here, `$\equiv$' denotes equivalence modulo a smoothing operator uniform in $\omega$. We also used here the fact that $\gamma_\lambda^\pm$ is an involution. Thus we recover (\ref{eq:EdC=}) upon setting 
\[A^\pm_\omega \equiv \big((\gamma_\lambda^\pm)^* \mathcal E_\omega (\gamma_\lambda^\pm)^*\big) \widehat T_{\mathrm{Ref}}^\pm,\]
where a smoothing operator should be added to either the $+$ or $-$ operator in order to have true equality as in (\ref{eq:EdC=}), but for our purposes, this makes no difference. 

\noindent
2. We will only compute the full symbol of $A_\omega^+$, since the computation for $A_\omega^-$ follows in a similar fashion with minor sign changes. First we consider the operator $\widehat T^+_{\text{Ref}}$. We have that the Schwartz kernel is given by (up to smoothing uniformly in $h$)
\begin{equation*}
    \widehat K^+_{\mathrm{Ref}}(\theta, \theta') \equiv \chi_{\mathrm{Diag}}(\theta, \theta') \tilde c_\omega^+(\theta') (\partial_\theta \gamma_\lambda^+(\theta)) (\theta - \theta' + i h z_\omega^+ (\theta'))^{-1},
\end{equation*}
which is pseudodifferential. This can be written as an oscillatory integral with symbol 
\[a_{\text{Ref}}(\theta, \theta', \xi) = - 2 \pi i \chi_{\mathrm{Diag}} \tilde c^+_\omega(\theta') (\partial_\theta \gamma_\lambda^+(\theta)) e^{-z_\omega^+(\theta) h|\xi|} H(\xi).\]
Now we apply the left reduction formula (\ref{eq:left_reduction}) to obtain the full asymptotic expansion for the left symbol uniformly in $h$. 
To obtain the expansion, we compute for $\xi > 0$ 
\begin{align*}
    &\left. \left(\partial^k_\xi \partial^k_{\theta'} \tilde c^+_\omega(\theta') (\partial_\theta \gamma_\lambda^+(\theta)) e^{-h z_\omega^+(\theta') |\xi|} H(\xi) \right) \right|_{\theta = \theta'} \\
    =& \sum_{j = 0}^k \binom{k}{j} \left( \partial_{\theta}^{k - j} \tilde c_\omega^+(\theta) \right) \left(\partial_\theta \gamma_\lambda^+ (\theta)\right) \left( \partial_\xi^k \partial_{\theta}^{j} e^{-z_\omega^+(\theta) h|\xi|} \right) H(\xi) \\
    =& h^k e^{- z_\omega^+(\theta) h|\xi|} H(\xi) \sum_{j = 0}^k p^{j, k}_{\omega}(\theta)(h|\xi|)^j
\end{align*}
where $p^{k, j}_\omega(\theta)$ is smooth in $\theta$ and $h$ up to $h = 0$. Therefore we have the full left symbol
\begin{equation}\label{eq:T_full}
    \sigma(\widehat T_{\mathrm{Ref}}^+)(\theta, \xi) \sim \sum_{k = 0}^\infty h^k p_{\mathrm{Ref},\, \omega}^{k, +}(\theta, h|\xi|) e^{-h|\xi|z_\omega^+(\theta)} H(\xi)
\end{equation}
where
\begin{equation}\label{eq:p_in_S^k}
    p_{\mathrm{Ref},\, \omega}^{k, +}(\theta, \xi) = \frac{-2 \pi i^{-k + 1}}{k!} \sum_{j = 0}^k p_\omega^{j, k}(\theta)|\xi|^j \in S^k(T^* \mathbb S^1)
\end{equation}
uniformly in $h$. This means that modulo smoothing uniformly in $h$, $\widehat T^+_{\mathrm{Ref}}$ can be realized a positive semiclassical operator:
\[\widehat T_{\mathrm{Ref}}^+ \equiv \Op_h^+(\tau) \in \Psi^{-\infty}_{h, +} \quad \text{where} \quad \tau(x, \xi) \sim \sum_{k = 0}^\infty h^k p_{\mathrm{Ref},\, \omega}^{k, +}(\theta, \xi) e^{-\xi z_\omega^+(\theta)}\]
From (\ref{eq:p_in_S^k}), we see that the positive semiclassical principal symbol is given by 
\begin{equation}\label{eq:T_Ref_principal}
    \sigma_h^+(\widehat T^+_{\text{Ref}})(\theta, \xi) = -2 \pi i \tilde c_\omega^+(\theta) (\partial_\theta \gamma_\lambda^\pm(\theta)) e^{-\xi z_{\omega}^\pm(\theta)}
\end{equation}
Note that from (\ref{eq:E_omega_full}), we see that $\mathcal E_\omega \equiv \frac{i}{2 \pi c_\omega} (\Pi^+ - \Pi^-)$, so by Corollary \ref{cor:Heaviside_identities}, we have
\begin{equation*}
    (\gamma^+_\lambda)^* \mathcal E_\omega ( \gamma_\lambda^+)^* \equiv \frac{i}{2 \pi c_\omega} (\Pi^- - \Pi^+)
\end{equation*}
modulo smoothing uniformly in $h$. It then follows from the remarks at the end of \S\ref{subsubsec:positive_semiclassical} that  $A_\omega^+ = (\gamma^+_\lambda)^* \mathcal E_\omega (\gamma^+_\lambda)^* \widehat T^+_{\mathrm{Ref}} \in \Psi^{-\infty}_{h, +}$, and has principal symbol 
\begin{align*}
    \sigma_h^+ (A_\omega^+) &= -\frac{\tilde c_\omega^+(\theta)}{c_\omega} \partial_\theta \gamma_{\lambda}^+(\theta) e^{-\xi z_\omega^\pm(\theta)} \\
    &= (-1 + \mathcal O(h)) e^{-\xi z_\omega^\pm(\theta)} \\
    &\equiv -e^{-\xi z_\omega^\pm(\theta)}. 
\end{align*}
The last equivalence here is in the sense of positive semiclassical principal symbol, i.e. modulo $hS_h^{-\infty}(T^* \mathbb S^1)$.  
\end{proof}

\subsection{Convergence properties}
Now we show strong convergence properties of the restricted single layer operator and the single layer operator. We will need the derivative estimates to obtain the final spectral result. 

For the restricted single layer operator, we cite the result from \cite[Lemma~4.16]{DWZ}. 
\begin{proposition}\label{prop:C_converge}
    Let $\lambda_0 \in (0, 1)$ be such that $\Omega$ is $\lambda_0$-simple, and let $k \in \N_0$ and $s + 1 > t$. Then for all $\omega_j \to \lambda$, $\Im \omega_j > 0$, 
    \begin{equation}
        \|\partial_\omega^k \mathcal C_{\omega_j} - \partial_\lambda^k \mathcal C_{\lambda_0 + i0} \|_{H^{s + k}(\partial \Omega; T^* \partial \Omega) \to H^t(\partial \Omega)} \to 0.
    \end{equation}
    Similarly, if the sequence $\omega_j \to \lambda_0$ satisfies $\Im \omega_j < 0$, then the same convergence statement holds with $\mathcal C_{\lambda_0 - i0}$ instead of $\mathcal C_{\lambda_0 + i0}$.
\end{proposition}
The above lemma is proved in \cite{DWZ} by directly analyzing the Schwartz kernel of the operator $\partial_\omega^k \mathcal C_\omega$ to obtain uniform bounds, similar to how we obtained the normal derivative bounds in Proposition \ref{prop:normal}. Then using Lemma~\ref{lem:fs_holomorphic}, one can upgrade the uniform bounds to convergence statements. 

In addition to the convergence results from \cite{DWZ} for the restricted single layer operator, we also need nice convergence properties for the single layer operator, which maps one-forms on the boundary to functions on~$\overline \Omega$. The strategy to obtain regularity estimates on the interior of the domain is to use the fact that $L^\pm_\omega$ is a Cauchy-Riemann type operator, so we can apply the maximum modulus principle.  

\begin{proposition}\label{prop:S_converge}
    Let $\lambda_0 \in (0, 1)$ be such that $\Omega$ is $\lambda_0$-simple. Let $v \in C^\infty(\partial\Omega; T^* \partial \Omega)$. Then for all $\omega_j \to \lambda_0$, $\Im \omega_j > 0$, and $s \in \N_0$,
    \begin{equation}
        \partial_\omega^s S_{\omega_j} v \to \partial_\lambda^s S_{\lambda_0 + i0} v \quad \text{in $C^\infty(\overline \Omega)$}.
    \end{equation}
    Similarly, if the sequence $\omega_j \to \lambda_0$ satisfies $\Im \omega_j < 0$, then the same convergence statement holds with $S_{\lambda_0 - i0}$ instead of $S_{\lambda_0 + i0}$.
\end{proposition}
\begin{proof}
    1. We show the $\omega_j > 0$ the lower half plane case follows similarly. We first show that $S_\omega v \in C^1(\overline \Omega)$ uniformly for $\omega \in \mathcal J + i(0, h_0)$ where $\mathcal J$ is some small neighborhood of $\lambda_0$ and $h_0 > 0$ is small. We know from Lemma~\ref{lem:S_mapping_prop} that $S_\omega v \in C^\infty(\overline \Omega)$, and from Proposition \ref{prop:C_converge}, $(S_\omega v)|_{\partial \Omega} \in C^\infty(\partial \Omega)$ uniformly. Therefore by the fundamental theorem of calculus, it suffices to show that 
    \begin{equation}\label{eq:sufficient_for_C^1}
        L^\pm_\omega S_\omega v \in L^\infty(\Omega) \quad \text{uniformly for $\omega \in \mathcal J + i(0, h_0)$}.
    \end{equation}
    Recall that $E_\omega$ is the fundamental solution, so 
    \[L_\omega^+ L_\omega^- E_\omega(x) = \frac{1}{4} P(\omega)E_\omega(x) = 0 \quad \text{for all $x \neq 0$}.\] 
    Fix coordinates $\theta$ on $\partial \Omega$ and write $v = f \, d\theta$. Then for $x \in \Omega$ (not on the boundary), we see from (\ref{eq:S_explicit}) that
    \[L^-_\omega L^+_\omega S_\omega v = L^+_\omega L^-_\omega S_\omega v = L^+_\omega L^-_\omega \int_{\partial \Omega} E_{\omega}(x - y) f(y) \, d \theta(y) = 0.\]
    Since $L^+_\omega$ and $L^-_\omega$ are Cauchy--Riemann type operators, it follows from the maximum modulus principle that 
    \begin{equation}\label{eq:first_max_mod}
    \begin{gathered}
        \sup_{x \in \Omega} |L^\pm_\omega S_\omega v(x)| \le c_\omega \sup_{x \in \partial \Omega} |\mathscr N_{\omega, 1}^\pm v(x)| \\
    \end{gathered}
    \end{equation}
    where $\mathscr N_{\omega, 1}^\pm$ is the restriction to boundary of $L^\pm_\omega S_\omega$ by extending from the interior (see \eqref{eq:normal_restriction}). By Proposition \ref{prop:normal}, the right-hand side of \eqref{eq:first_max_mod} is uniformly bounded for $\omega \in \mathcal J + i(0, h_0)$ for some $\mathcal J \ni \theta_0$ and $h_0 > 0$, which implies (\ref{eq:sufficient_for_C^1}).

    \noindent
    2. Now we prove an easy upgrade of Step 1 that shows $S_\omega v \in C^\infty(\overline \Omega)$ uniformly in $\omega$. Since $L_\omega^+$ and $L_\omega^-$ commute, it follows that
    \[L_\omega^\mp (L^{\pm}_\omega)^k S_\omega v = 0, \quad k \in \N\]
    Again by the maximum modulus principle, 
    \begin{equation}\label{eq:kth_sup_bound}
        \sup_{x \in \Omega} \left| (L^\pm_{\omega})^k S_\omega v \right| \le C \sup_{x \in \partial \Omega} \left| \mathscr N_{\omega, k}^\pm v \right|
    \end{equation}
    for some uniform constant $C$ for $\omega \in \mathcal J + i(0, h_0)$. By Proposition \ref{prop:normal}, it follows that the right-hand side is uniformly bounded for $\omega \in \mathcal J + i(0, h_0)$. Furthermore, $(L_\omega^+)^j (L_\omega^-)^{k - j} S_\omega v = 0$ for every $1 \le j \le k-1$, so we see that $S_\omega v \in C^k(\overline \Omega)$ uniformly bounded uniformly in~$\omega$ for every $k \in \N$.

    \noindent
    3. Using the $s = 0$ as the base case, we now induct to obtain uniform bounds for higher derivatives in $\omega$. Take $s \in \N$ and assume that $\partial_\omega^r S_\omega v$ is uniformly bounded in $C^\infty(\overline \Omega)$ for $\omega \in \mathcal J + i(0, h_0)$ for all $r \le s - 1$. Observe that 
    \[[(L^\pm_\omega)^k, \partial_\omega^s] = \sum_{r \le s - 1} \sum_{|\alpha| = k} \mu^\pm_{r, \alpha}(\omega) \partial_{x}^\alpha \partial_\omega^r\]
    for some $\mu^\pm_{r, \alpha} \in C^\infty(\mathcal J + i[0, h_0))$. Thus by the induction hypothesis, 
    \begin{equation*}
        [(L^\pm_\omega)^k, \partial_\omega^s] S_\omega v \in C^\infty(\overline \Omega) \quad \text{uniformly for $\omega \in \mathcal J + i[0, h_0)$}. 
    \end{equation*}
    Writing out $S_\omega$ explicitly using \eqref{eq:S_explicit}, we have that for $x \in \Omega$, 
    \begin{align}
        \partial_\omega^s (L^\pm_\omega)^k S_\omega v &= \partial_\omega^s \int_{\partial \Omega} c_\omega \frac{v(y)}{\ell^\pm(x - y, \omega)^k} \nonumber \\
        &= \sum_{n = k}^{k + s} \int_{\partial \Omega} \frac{P_{\omega, n}(x, y) v(y)}{\ell^\pm(x - y, \omega)^n} \label{eq:partial_S_explicit}
    \end{align}
    for some polynomials $P_{\omega, n}(x, y)$ whose coefficients depend smoothly on $\omega \in \mathcal J + i[0, h_0)$. Therefore, $\partial_\omega^s (L_\omega^\pm)^k S_\omega v$ is the sum of terms of the form 
    \[\mathcal M_\omega \big((L^\pm_\omega)^n S_\omega\big) \mathcal M'_\omega v, \quad k \le n \le k + s\]
    where $\mathcal M_\omega$ and $\mathcal M'_\omega$ are multiplication operators by uniformly smooth functions in $C^\infty(\overline \Omega)$ and $C^\infty(\partial \Omega)$ respectively for $\omega \in \mathcal J + i(0, h_0)$. Note that by Proposition \ref{prop:normal}, the right-hand side of \eqref{eq:kth_sup_bound} depends only finitely many derivatives of $v$, and hence the right-hand side is uniformly bounded in $\omega$ with $v$ replaced with $\mathcal M'_\omega v$. So it follows that $\partial_\omega^s (L^\pm_\omega)^k S_\omega v$ lies in $L^\infty(\overline \Omega)$ uniformly for $\omega \in \mathcal J + i(0, h_0)$.

    Now for $j \in 1, \dots, k - 1$, 
    \[[(L^+_\omega)^j (L^-_\omega)^{k - j}, \partial_\omega^s] = \sum_{r \le s - 1} \sum_{|\alpha| = k} \mu^\pm_{r, \alpha}(\omega) \partial_{x}^\alpha \partial_\omega^r\]
    for some (possibly different) $\mu^\pm_{r, \alpha} \in C^\infty(\mathcal J + i[0, h_0))$. Thus by the induction hypothesis, 
    \begin{equation*}
    [(L^+_\omega)^j (L^-_\omega)^{k - j}, \partial_\omega^s] S_\omega v \in C^\infty(\overline \Omega) \quad \text{uniformly for $\omega \in \mathcal J + i[0, h_0)$}. 
    \end{equation*}
    Since $(L^+_\omega)^j (L^-_\omega)^{k - j} S_\omega v = 0$ in $\Omega$, we see that $\partial^s_\omega S_\omega v \in C^k(\overline \Omega)$ uniformly in~$\omega$ for every $k$.

    \noindent
    4. Finally, we upgrade from uniform boundedness of $\partial_\omega^s S_\omega v$ in $C^\infty(\overline \Omega)$ to convergence. 
    Lemma~\ref{lem:fs_holomorphic}, we have weak convergence 
    \begin{equation}\label{eq:S_weak}
        \partial_\omega^s S_{\omega_j} v \to \partial_\lambda^s S_{\lambda_0 + i0} v  \quad \text{in} \quad \mathcal D'(\Omega).
    \end{equation}
    We have compact embedding $H^k(\overline \Omega) \to H^{k'}(\overline \Omega)$ for all $k > k'$, so $\{\partial_\omega^s S_{\omega_j} v\}$ is precompact in $H^k(\overline \Omega)$ for every $k$ (see, for instance, \cite[Appendix B]{H3}). Every convergent subsequence must converge to $\partial_\lambda^s S_{\lambda_0 + i0} v$ by (\ref{eq:S_weak}), so indeed we must have $\partial_\omega^s S_{\omega_j} v \to \partial_\lambda^s S_{\lambda_0 + i0} v$ in $H^k(\overline \Omega)$ for every $k$. 
\end{proof}

\section{High frequency analysis}\label{sec:high_freq_est}
Let us return to the elliptic boundary value problem (\ref{eq:EBVP}) given by 
\[P(\omega) u_\omega = f, \qquad u_\omega|_{\partial \Omega} = 0.\]
Recall from Lemma~\ref{lem:fs_app} that when $\Re \omega \neq 0$ and $f \in C_c^\infty(\Omega)$, there exists a unique solution $u_\omega \in C^\infty(\overline \Omega)$. We wish to control the $H^s$-norm of $u_\omega$ uniformly as $\omega$ approaches the real line for large $s$. As seen in the toy example in \S \ref{subsubsec:toy_square} as well as the result in \cite{DWZ}, the spectrum of $P$ is rather complicated, so controlling the $H^s$-norm of $u_\omega$ for $\omega$ close to the real line is delicate. The hope is that when the rotation number of the underlying chess billiard map is sufficiently irrational, we can obtain uniform high frequency estimates for spectral parameter $\omega$ near such forcing frequency.

In this section, we establish high frequency estimates for the boundary reduced problem in polynomial regions (see Figure \ref{fig:polyregion}) above or below a forcing frequency $\lambda$ whose associated chess billiard map has a Diophantine rotation number. The precise boundary reduced equation will be given by (\ref{eq:v=Bb+g}), and the high frequency estimate for the boundary reduced problem will be given by Proposition \ref{prop:high_freq_est}. Throughout this section, we fix a choice of the identification $\partial \Omega \simeq \mathbb S^1$ (See Remark at the end of \ref{subsec:prelim_reduct}.)

\subsection{Preliminary reductions}\label{subsec:prelim_reduct}
First we apply the identity (\ref{eq:EdC=}) to itself so that the chess billiard map appears. Define 
\begin{equation}\label{eq:curly_A}
    \mathcal A_\omega := (\gamma_\lambda^+)^* A_\omega^+ + (\gamma_\lambda^-)^* A_\omega^-.
\end{equation}
For now, let 
$$
\omega=\lambda+ih \in \mathcal J + i(0, \infty)
$$
where $\mathcal J \Subset (0, 1)$ is some small open subinterval such that $\Omega$ is $\lambda$-simple for all $\lambda \in U$. As before, we only study the upper half plane case for $\omega$ since the lower half plane case follows similarly with minor sign modifications. 
\begin{lemma}\label{lem:A^2}
There exists $B_\omega^\pm = B_{\lambda + ih}^\pm \in \Psi_{h, \pm}^{-\infty} (\mathbb S^1)$ locally uniformly in $\lambda$ such that we have the identity
\begin{equation}\label{eq:A^2}
    \mathcal A^2_\omega = B_\omega^+ b_\lambda^* + B_\omega^- b_\lambda^{-*} + \mathscr R_\omega,
\end{equation}
where $\mathscr R_\omega \in \Psi^{-\infty}$ locally uniformly for $\omega \in \mathcal J + i [0, \infty)$. Furthermore, we have the principal symbol
\begin{equation}\label{eq:beta^pm}
\sigma_h^\pm(B_\omega^\pm) = e^{-\tilde z_\omega^\pm(x) |\xi|} \quad \text{where} \quad \Re \tilde z_\omega^\pm \ge \ell > 0 \quad \text{uniformly in $\omega$}
\end{equation}
\end{lemma}

\begin{proof}
The diffeomorphism $\gamma^\pm_\lambda$ is orientation reversing, so by Proposition \ref{prop:full_symbol}, we have $(\gamma_\lambda^\pm)^* A_\omega^\pm (\gamma_\lambda^\pm)^* A_\omega^\pm \in \Psi^{-\infty}$ uniformly in $\omega$. Then squaring $\mathcal A_\omega$, we find 
\begin{align}
    \mathcal A_\omega^2 &= (\gamma_\lambda^-)^* A_\omega^- (\gamma_\lambda^+)^* A_\omega^+ + (\gamma^+_\lambda)^* A^+_\omega (\gamma^-_\lambda)^* A^-_\omega + \Psi^{-\infty} \nonumber \\
    &= \left((\gamma_\lambda^-)^* A_\omega^-  (\gamma_\lambda^{-})^* \right) \left( b_\lambda^* A_\omega^+ b_\lambda^{-*} \right) b_\lambda^* \nonumber \\
    & \qquad + \left((\gamma_\lambda^+)^* A_\omega^+  (\gamma_\lambda^{+})^* \right) \left( b_\lambda^{-*} A_\omega^- b_\lambda^{*} \right) b_\lambda^{-*} + \Psi^{-\infty} \label{eq:A^2_presub}
\end{align}
where the remainder in (\ref{eq:A^2_presub}) lies in $\Psi^{-\infty}$ uniformly in $\omega$. By Lemma~\ref{lem:positive_comp} and Lemma~\ref{lem:positive_COV}, we see that 
\begin{equation}\label{eq:B^pm}
    B_\omega^\pm := \left((\gamma_\lambda^\mp)^* A_\omega^\mp  (\gamma_\lambda^{\mp})^* \right) \left( (b_\lambda^\pm)^* A_\omega^\pm (b_\lambda^\mp)^* \right) \in \Psi_{h, +}^{-\infty}.
\end{equation}
Furthermore, we have the principal symbol 
\begin{equation*}
    \sigma_h^{\pm}(B_\omega^\pm) = e^{-\tilde z_\omega^\pm(x) |\xi|}
\end{equation*}
where
\begin{equation}\label{eq:tilde_z}
    \tilde z_\omega^\pm(x) = \frac{z_\omega^\mp(\gamma_{\lambda}^\mp(x))}{|\partial_x \gamma^\mp_\lambda(x)|} + \frac{z_\omega^\pm(b_\lambda^{\pm     1}(x))}{\partial_x b_\lambda^{\pm 1}(x)}.
\end{equation}
Note that (\ref{eq:tilde_z}) gives
\[\Re \tilde z_\omega^\pm \ge \ell > 0,\]
so the lemma follows upon putting $\mathscr R_\omega$ as the uniform smoothing remainder in (\ref{eq:A^2_presub}). 
\end{proof}

Reducing the boundary value problem (\ref{eq:EBVP}) via the restricted single layer potential, we have (\ref{eq:restricted_problem}):
\[\mathcal C_\omega v_\omega = F_\omega, \qquad F_\omega := (R_\omega f)|_{\partial \Omega}\]
where $\mathcal C_\omega$ is the restricted single layer operator and $v_\omega := \mathcal N_\omega u_\omega$ is the `Neumann data.' Therefore, we can rewrite (\ref{eq:EdC=}) as
\begin{equation*}
    v_\omega = -\mathcal A_\omega v_\omega + \mathcal E_\omega dF_\omega.
\end{equation*}
Applying this identity to itself and using (\ref{eq:A^2}), 
we get the identity
\begin{equation}\label{eq:v=Bb+g}
    v_\omega = B_\omega^+ b_\lambda^* v_\omega + B_\omega^- b_\lambda^{-*} v_\omega + \mathscr R_\omega v_\omega +  f_\omega
\end{equation}
for $\omega = \lambda + ih \in \mathcal J + i(0, \infty)$ where
\[f_\omega := (I - \mathcal A_\omega) \mathcal E_\omega d F_\omega.\]
By Proposition \ref{prop:full_symbol}, both $A^\pm_\omega$ and $\mathcal E_\omega$ are pseudodifferential uniformly in $\omega$, 
so $f_\omega \in C^\infty(\mathbb S^1; T^* \mathbb S^1)$. This motivates the analysis of (\ref{eq:v=Bb+g}). The main goal of this section is to obtain high frequency estimates for $v_\omega$ given high frequency control of some $f_\omega$. The precise statement is given in Proposition~\ref{prop:high_freq_est}.

To simplify the analysis of (\ref{eq:v=Bb+g}), we split the identity into positive and negative frequency pieces. Define a pseudodifferential partition of unity by 
\begin{equation}
    \begin{gathered}
        I = \Pi^+ + \Pi^-, \quad  \Pi^\pm \in \Psi^0(\mathbb S^1; T^* \mathbb S^1), \\
        \WF(\Pi^\pm) \subset \{\pm \xi > 0\}, \quad \sigma(\Pi^\pm)(x, \xi) = H(\pm \xi).
    \end{gathered}
\end{equation}
Put
\begin{equation}\label{eq:v,f_projections}
    v_\omega^\pm := \Pi^\pm v_\omega, \quad f_\omega^\pm := \Pi^\pm f_\omega.
\end{equation}
Then applying $\Pi^\pm$ to (\ref{eq:v=Bb+g}), we find that
\begin{equation}\label{eq:split_freq}
    v_\omega^\pm = B^\pm_\omega (b_\lambda^{\pm 1})^* v_\omega^{\pm} + \mathscr R_\omega^{\pm} v_\omega + f_\omega ^\pm
\end{equation}
where we have the smoothing remainder
\[\mathscr R_\omega^\pm := \left( [\Pi^\pm, B_\omega^\pm] + B_\omega^\pm (\Pi^\pm - (b_\lambda^{\pm 1})^* \Pi^\pm (b_\lambda^\mp)^*) \right) (b_\lambda^{\pm 1})^* + \Pi^\pm B_\omega^\mp(b_\lambda^{\mp 1})^* + \Pi^\pm \mathscr R_\omega,\]
which lies in $\Psi^{-\infty}(\mathbb S^1; T^* \mathbb S^1)$ uniformly in $h$. 

\Remarks 1. We use the same notation for the projectors $\Pi^\pm$ as in \S\ref{subsubsec:positive_semiclassical} since they differ by at most $\Psi^{-\infty}$, which makes no difference. 

\noindent
2. For convenience, note that we now use $x$ instead of $\theta$ as the variable for $\partial \Omega \simeq \mathbb S^1$ since we are now exclusively working on the boundary reduced problem, so there is no ambiguity with the coordinates on $\Omega \subset \R^2$.

\subsection{An approximate Denjoy's theorem}
Fix $\lambda_0 \in (0, 1)$ such that $\mathbf r(\lambda)$ is Diophantine (see Definition \ref{def:diophantine}). We localize the analysis near $\lambda_0$, so we now put
$$
\omega = \lambda_0 + \epsilon + ih, \quad h > 0.
$$
The analysis for $\omega$ in the lower half plane is the same up to minor sign changes. We wish to first understand the operator $B^\pm_\omega (b^{\pm 1}_{\lambda_0 + \epsilon})^*$. First we conjugate by an $\epsilon$-dependent pullback operator so that the chess billard map becomes $\epsilon^\infty$-close to a pure rotation. More precisely, we have the following lemma:

\begin{lemma}\label{lem:approx_conj}
For all sufficiently small $\epsilon$, there exists a family of orientation-preserving diffeomorphisms $\varphi_\epsilon: \mathbb{S}^1 \to \mathbb{S}^1$ depending smoothly on $\epsilon$ so that
\[(\varphi_\epsilon^{-1} \circ b_{\lambda_0 + \epsilon} \circ \varphi_\epsilon)(x) = x + \alpha(\epsilon) + r_\epsilon(x),\]
where $\alpha$ is a real-valued smooth function on a neighborhood of $0$ with $\alpha(0) = \mathbf r(\lambda_0)$, and the remainder $r_\epsilon: \R \to \R$ is a smooth periodic function that satisfies $r_\epsilon(x) = \mathcal O(\epsilon^\infty)_{C^\infty}.$
\end{lemma}
\Remark We can lift smooth circle diffeomorphisms to smooth maps on $\R$. In both the lemma and the proof, we use the same notation for both since there will be no ambiguity, and it avoids writing $(\mathrm{mod}\,  1)$ everywhere.
\begin{proof}
We build $\varphi_\epsilon$ asymptotically. By Proposition \ref{prop:denjoy_moser}, since $\mathbf r(\lambda_0)$ is Diophantine, there exists a diffeomorphism $\psi_0: \mathbb{S}^1 \to \mathbb{S}^1$ such that 
\[\psi_0 \circ b_{\lambda_0} \circ \psi_0^{-1} = \rho_{\mathbf r(\lambda_0)}, \qquad \rho_{\mathbf r(\lambda_0)}(x) = x + \mathbf r(\lambda_0).\] 
There exists a family of periodic functions $\beta^{(1)}_\epsilon: \R \to \R$ depending smoothly on $\epsilon$ such that
\[b_{\lambda_0 + \epsilon}^{(1)}(x) := (\psi_0 \circ b_{\lambda_0 + \epsilon} \circ \psi_0^{-1})(x) = x + \mathbf r(\lambda_0) + \epsilon \beta^{(1)}_\epsilon(x).\]
From the above, we see that conjugating $b_{\lambda_0 + \epsilon}$ by $\psi_0$ yields a pure rotation plus an order $\epsilon$ error. In order to improve the error, we introduce periodic function $p_1: \R \to \R$ that solves the cohomological equation 
\begin{equation}\label{eq:p_1}
    p_1(x + \mathbf{r}(\lambda_0)) - p_1(x) + \beta^{(1)}_{\epsilon = 0}(x) = \alpha_1 \quad \text{where } \alpha_1 = \int_0^1 \beta^{(1)}_{\epsilon = 0}(x)\, dx,
\end{equation}
in the sense of Lemma~\ref{lem:cohom}, which indeed has a smooth periodic solution since $\mathbf{r}(\lambda_0)$ is Diophantine. Consider the ($\epsilon$-dependent family of) circle diffeomorphisms $\psi_1(x) = x + \epsilon p_1(x)$. Then observe that by (\ref{eq:p_1}) and Taylor's theorem, there exist periodic functions $\beta^{(2)}_\epsilon: \R \to \R$ such that 
\begin{align*}
    \big(\psi_1 \circ b_{\lambda_0 + \epsilon}^{(1)} \big)(x) &= x + \mathbf r(\lambda_0) + \epsilon \beta_\epsilon^{(1)} (x) + \epsilon p_1\big(x + \mathbf r(\lambda_0) + \epsilon \beta_\epsilon^{(1)}(x) \big) \\
    &= \psi_1(x) + \mathbf r(\lambda_0) + \epsilon \alpha_1 + \epsilon^2 \beta_\epsilon^{(2)} (\psi_1(x)).
\end{align*}
Therefore, we have the improvement
\[b_{\lambda_0 + \epsilon}^{(2)} := \big(\psi_1 \circ b_{\lambda_0 + \epsilon}^{(1)} \circ \psi_1^{-1} \big)(x) = x + \mathbf r(\lambda_0) + \epsilon \alpha_1 + \epsilon^2 \beta_\epsilon^{(2)} (x).\]
Now we can iteratively construct 1-periodic functions $p_n: \R \to \R$ and $\alpha_n \in \R$ that solve the cohomological equation 
\[p_n(x + \mathbf{r}(\lambda_0)) - p_n(x) + \beta_{\epsilon = 0}^{(n)}(x) = \alpha_n \quad \text{where } \alpha_n = \int_0^1 \beta^{(n)}_{\epsilon = 0}(x)\, dx\]
to produce diffeomorphisms 
\[\psi_{n}(x) = x + \epsilon^n p_n(x).\]
Here, $\beta_\epsilon^{(n)}: \R \to \R$ are periodic functions chosen so that 
\begin{equation}\label{eq:denjoy_induction}
    b_{\lambda_0 + \epsilon}^{(n)} := \big(\psi_{n - 1} \circ b_{\lambda_0 + \epsilon}^{(n - 1)} \circ \psi_{n - 1}^{-1}\big)(x) = x + \mathbf{r}(\lambda_0) + \sum_{k = 1}^{n - 1} \epsilon \alpha_k + \epsilon^n \beta_\epsilon^{(n)} (x),
\end{equation}
For all sufficiently small $\epsilon$, we can define an $\epsilon$-dependent family of diffeomorphisms $\psi_\epsilon: \mathbb S^1 \to \mathbb S^1$ that satisfies
\[\psi_\epsilon(x) - (\psi_n \circ \psi_{n - 1} \circ \cdots \circ \psi_{0})(x) = \mathcal O(\epsilon^{n + 1})_{C^\infty}\]
for every $n \in \N$. The existence of $\psi_\epsilon$ is guaranteed by Borel's lemma. Again summing asymptotically in $\epsilon$, there exists $\alpha(\epsilon) \sim \mathbf{r} (\lambda_0) + \sum_{k = 1}^\infty \epsilon^{k} \alpha_k$. It then follows from the inductive step (\ref{eq:denjoy_induction}) that 
\[(\psi_\epsilon \circ b_{\lambda_0 + \epsilon} \circ \psi_\epsilon^{-1})(x) = x + \alpha(\epsilon) + \mathcal O(\epsilon^\infty)_{C^\infty}.\]
The lemma then follows upon setting $\varphi_\epsilon = \psi_\epsilon^{-1}$.
\end{proof}

\subsection{Construction of the microlocal weight}
Let $\varphi_\epsilon$ be as in Lemma~\ref{lem:approx_conj}. Conjugating the operator $B_\omega^\pm$
from Lemma~\ref{lem:A^2} by $\varphi_\epsilon^*$ is a change of variables for each $\epsilon$, so 
\[\widetilde B_\omega^\pm := \varphi_\epsilon^* B_\omega^\pm \varphi_\epsilon^{-*}\]
remains pseudodifferential. By the change of variables formula (Proposition \ref{prop:classical_change}) and Lemma~\ref{lem:A^2}, $\widetilde B_\omega^\pm$ is a positive (or negative) semiclassical operator with principal symbol
\begin{equation}\label{eq:tilde_B_principal}
    \qquad \sigma_h^\pm( \widetilde B_\omega^\pm) = e^{-|\xi| \zeta_\omega^\pm(x)} \quad \text{where} \quad \zeta_\omega^\pm(x) = \frac{\tilde z^\pm_\omega(\varphi_\epsilon(x))}{\partial_x \varphi_\epsilon(x)}.
\end{equation}
Note that 
\begin{equation}
\Re \zeta_\omega^\pm(x) \ge \ell > 0
\end{equation}
uniformly in $\omega$ for a possibly different $\ell$ than in (\ref{eq:beta^pm}).

By (\ref{eq:split_freq}), it suffices to understand the operators
\begin{equation}\label{eq:tilde_b_epsilon}
    \widetilde B^\pm_\omega (\tilde b_\epsilon^{\pm 1})^* \quad \text{where } \tilde b_\epsilon(x) := (\varphi_\epsilon^{-1} \circ b_{\lambda_0 + \epsilon} \circ \varphi_\epsilon)(x) = x + \alpha(\epsilon) + r_\epsilon(x).
\end{equation}

We split $\widetilde B_\omega^\pm$ into a low frequency and a high frequency piece and treat them separately. More specifically, fix a cutoff
\begin{equation}\label{eq:cutoffs}
    \begin{gathered}
        \chi \in \CIc([0, \infty);[0, 1]), \quad  \\
        \chi(\xi) = 1 \quad \text{for $x \in [0, 1]$}, \quad \supp \chi \subset [0, 2).
    \end{gathered}
\end{equation} 
Define the low frequency and high frequency pieces of $\widetilde B_\omega^\pm$ as follows:
\begin{gather}
    B_L^\pm := \Op^\pm_h(\chi(\pm \xi)) \widetilde B_\omega^\pm  \label{eq:low_freq},\\
    \qquad B_H^\pm := \Op_h(1 - \chi(\pm \xi)) \widetilde B_\omega^\pm. \label{eq:high_freq}
\end{gather}
Note that $B_H^\pm$ is a semiclassical operator, and we see that modulo smoothing uniformly in $\omega$,
\begin{equation}
\begin{gathered}
    \widetilde B_\omega^\pm \equiv B_L^\pm + B_H^\pm, \\
    \sigma_h^\pm(B_L^\pm) = \chi(\pm \xi) e^{-|\xi| \zeta_\omega^\pm(x)}, \qquad \sigma_h(B_H^\pm) = \big(1 - \chi(\pm \xi) \big) e^{-|\xi| \zeta_\omega^\pm(x)},
\end{gathered}
\end{equation}
where the principal symbol follows from (\ref{eq:tilde_B_principal}). We stress that $B_L^\pm$ and $B_H^\pm$ depend on $\omega = \lambda_0 + \epsilon + ih$. The implicit dependence causes no ambiguity since they lie in $\Psi_{h, \pm}^{-\infty}$ and $\Psi_h^{-\infty}$, respectively, uniformly in $\epsilon$ for all sufficiently small $\epsilon$. The crucial property of the high frequency piece is that since $\Re \zeta^\pm_\omega(x) \ge \ell > 0$, $B_H^\pm$ is a contraction on $L^2$ for all sufficiently small $h$. To deal with the low frequency piece $B_L^\pm$, we conjugate away most of the dependence on $x$ using a scheme that resembles Moser averaging to obtain a Fourier multiplier plus a manageable error.

To carry out the conjugation, we first need the following lemma, which approximately solves a ``warped'' cohomological equation. We consider 
\[\omega = \lambda_0 + \epsilon + ih, \quad h > 0\]
for $\epsilon$ in a sufficiently small neighborhood of $0$ so that Lemma~\ref{lem:approx_conj} holds. 
\begin{lemma}\label{lem:moser_conjugator}
    Let $\alpha(\epsilon)$ and $r_\epsilon(x)$ be as constructed in Lemma~\ref{lem:approx_conj}. Consider an $\omega$-dependent family of symbols $a_\omega \in S^{-\infty}(T^* \mathbb{S}^1)$ uniformly in $\omega$. Then there exist $\omega$-dependent symbols $g_\omega \in S^{-\infty}(T^* \mathbb{S}^1)$ uniformly and $\omega$ such that
    \begin{equation}\label{eq:warped_cohom_eq}
        a_\omega(x, \xi) + g_\omega(x + \alpha(\epsilon) + r_\epsilon(x), \xi) - g_\omega(x, \xi) = \kappa_\omega(\xi) + \mathcal O(\epsilon^\infty)_{S^{-\infty}(T^* \mathbb S^1)},
    \end{equation} 
    where
    \begin{equation}
        \kappa_\omega(\xi) = \int_0^1 a_\omega(x, \xi)\, dx.
    \end{equation}
\end{lemma}

\Remarks 1. From the change of variables formula in Proposition \ref{prop:classical_change}, the symbol $g_\omega(x, \xi)$ should become
\[g_\omega(x + \alpha(\epsilon) + r_\epsilon(x), \xi/(1 + \partial_x r_\epsilon(x)))\]
if we were to perform a change of variables by $x \mapsto x + \alpha(\epsilon) + r_\epsilon(x)$. However $r_\epsilon(x) = \mathcal O(\epsilon^\infty)_{C^\infty}$, so 
\[g_\omega(x + \alpha(\epsilon) + r_\epsilon(x), \xi/(1 + \partial_x r_\epsilon(x))) - g_\omega(x + \alpha(\epsilon) + r_\epsilon(x), \xi)= \mathcal O(\epsilon^\infty)_{S^{-\infty}(T^* \mathbb S^1)},\]
which can get absorbed into the error in \eqref{eq:warped_cohom_eq} anyways. 

\noindent
2. Note that the error term $\mathcal O(\epsilon^\infty)_{S^{-\infty}(T^* \mathbb S^1)}$ possibly depends on $x$, in contrast to $\kappa_\omega$, which crucially does not depend on $x$.

\noindent 
3. We emphasize before the proof that $\omega = \lambda_0 + \epsilon + ih$. However, it helps in the proof to distinguish between $\epsilon$ and $\omega$. In particular, $g_\omega$ will be constructed asymptotically in powers of $\epsilon$ and will be in $S^{-\infty}(T^* \mathbb{S}^1)$ uniformly in $\omega$, so there is no harm in setting $\epsilon = \Re \omega - \lambda_0$ at the end. The strategy is simply to Taylor expand in $\epsilon$ and solve a cohomological equation for each coefficient. 

\begin{proof}
Since $\alpha (0)$ is Diophantine, there exists $g_\omega^{(0)} \in S^{-\infty}(T^* \mathbb{S})$ that solves the cohomological equation
\[a_\omega(x, \xi) + g^{(0)}_\omega(x + \alpha(0), \xi) - g^{(0)}_\omega(x, \xi) = \kappa_\omega(\xi) \]
in the sense of Lemma~\ref{lem:cohom}. Now, we can inductively define $g^{(n)}_\omega \in S^{-\infty}(T^* \mathbb{S}^1)$ as a solution to the homological equation 
\begin{equation}\label{eq:inductive_cohom}
    g^{(n)}_\omega(x + \alpha(0), \xi) - g_\omega^{(n)}(x, \xi) = f^{(n)}_\omega(x, \xi)
\end{equation}  
where 
\[f^{(n)}_\omega (x, \xi) = -\sum_{k = 0}^{n - 1} \frac{1}{(n - k)!} \frac{\partial^{n - k}}{\partial \epsilon^{n - k}} g^{(k)}_\omega (x + \alpha(\epsilon), \xi) \Big|_{\epsilon = 0}.\]
Note that since (the lift of) $g_\omega^{(k)}(\bullet, \xi)$ is periodic, and $\partial_\epsilon^{n - k} g^{(k)}_\omega (x + \alpha(\epsilon), \xi)\big|_{\epsilon = 0}$ is a linear combination of $\partial_x^{k'}  g^{(k)}_\omega (x + \alpha( 0), \xi)$ for $k' = 1,\dots, n - k$, it follows that 
\[\int_0^1 f^{(n)}_\omega(x, \xi)\, dx = 0.\]
This means that (\ref{eq:inductive_cohom}) indeed has a solution. By induction, we see that 
\begin{multline*}
    a_\omega(x, \xi) + \left(\sum_{k = 0}^n \epsilon^k g_\omega^{(k)}(x + \alpha(\epsilon) + r_\epsilon(x), \xi) \right) - \left( \sum_{k = 0}^n \epsilon^k g_\omega^{(k)}(x, \xi) \right) \\
    = \kappa_\omega(\xi) + \mathcal O(\epsilon^{n+1})_{S^{-\infty}(T^* \mathbb S^1)}.
\end{multline*}
By Borel's lemma, we can find $g_\omega \sim \sum_{k = 0}^\infty \epsilon^k g^{(k)}_\omega$ which satisfies the claims of the lemma.
\end{proof}

Recall that we put $\omega = \lambda_0 + \epsilon +ih$. Note that both Lemma~\ref{lem:approx_conj} and Lemma~\ref{lem:moser_conjugator} produces errors are small in $\epsilon$. On the other hand, the pseudodifferential parts behave like semiclassical operators with parameter $h$, which will produce errors that are small in $h$. Thus we need simultaneous smallness in both $\epsilon$ and $h$, so we cannot in general expect to obtain uniform control of the resolvent in a region arbitrarily close to the real line near $\lambda_0$. Nevertheless, we can still get quite close since the errors in both Lemma~\ref{lem:approx_conj} and Lemma~\ref{lem:moser_conjugator} are order $\epsilon^\infty$. We consider the degree $d \ge 1$ polynomial region 
\begin{equation}\label{eq:poly_region}
    \Lambda^\pm_{\lambda_0, d} := \{\lambda_0 + \epsilon \pm ih \mid h \ge |\epsilon|^d,\, |\epsilon| \le \epsilon_0\}
\end{equation}
where $\epsilon_0$ is small. See Figure \ref{fig:polyregion}. 

\begin{figure}
    \centering
    \includegraphics[scale = 0.175]{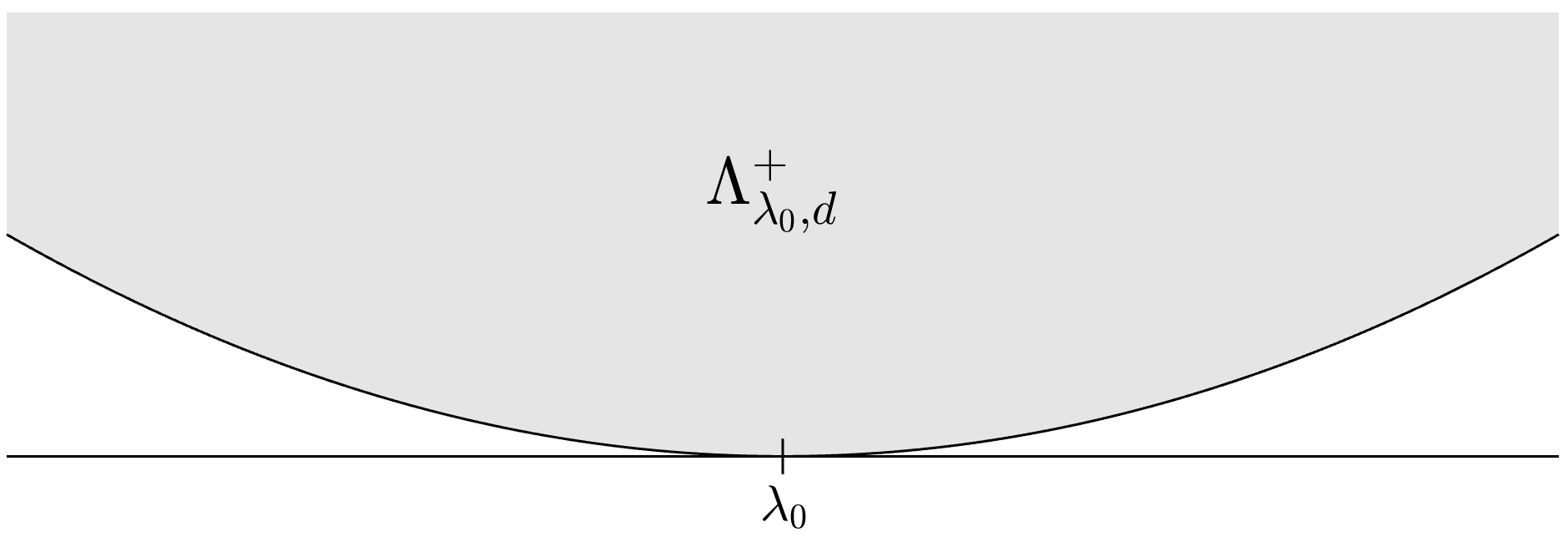}
    \caption{The polynomial region in the complex plane above $\lambda_0$ in which we will have uniform control of the resolvent. }
    \label{fig:polyregion}
\end{figure}

Again, the upper and lower half plane cases are completely symmetric, and we remain focused on the upper half plane for convenience. In the region $\Lambda = \Lambda^+_{\lambda_0, d}$, we now introduce a microlocal weight which conjugates $B_L^+$ to what is essentially a Fourier multiplier. Symbolically, this means that we conjugate away the dependence on $x$ in the full symbol.

\begin{lemma}\label{lem:low_freq_conj}
    For all $\omega = \lambda_0 + \epsilon + ih \in \Lambda$, there exists elliptic operator $Q \in \Psi^0_{h, +}( \mathbb S^1)$ with parametrix $\widetilde Q$ such that
    \begin{equation}\label{eq:low_freq_conj}
        \widetilde Q (B_L^+ \tilde b_\epsilon^*) Q = \left[M + K_L + \mathcal R_L \right] \tilde b_\epsilon^*
    \end{equation}
    where $\mathcal R_L \in \Psi^{-\infty}$ uniformly in $\omega$ and $K_L \in h \Psi^{-\infty}_h$ with 
    \[\WF_h(K_L) \subset \{\xi \ge 1\}.\]
    Furthermore, $M \in \Psi^{-\infty}_{h, +} (\mathbb S^1)$ is a Fourier multiplier given by
    \begin{equation}\label{eq:M_symbol}
    \begin{gathered}
        Mu(x) = \sum_{k \in \Z_{> 0}} e^{2 \pi i x k} m(2 \pi h k) \hat u(k), \qquad \hat u(k):= \int_0^1 e^{-2 \pi i x k} u(x) \, dx, \\
        m(\eta) := \chi(\eta) \big(e^{-\eta \kappa_0} + \mathcal O(h)_{C^\infty([0, 2])} \big).
    \end{gathered}
    \end{equation}
    where $\kappa_0 \in \C$ is $\omega$ dependent with $\Re \kappa_0 > 0$ uniformly for $\omega$
\end{lemma}

\Remarks 1. It is important to recall the differences in the semiclassical quantization of smoothing symbols (see~\eqref{eq:psi-inf}) and the nonsemiclassical quantization of symbols that are uniformly smoothing in $h$. See \eqref{eq:zero_infinity} for details. We stress this distinction since we will be working with elements in $\Psi_h^{-\infty}$ and $\Psi^{-\infty}$ simultaneously in the proof below, and elements of $\Psi_h^{-\infty}$ are not in $\Psi^{-\infty}$ uniformly in $h$. 

\noindent
2. $Q$ is elliptic as an operator in the positive semiclassical calculus as defined in (\ref{eq:elliptic_def}). The parametrix $\widetilde Q$ as constructed in Lemma~\ref{lem:parametrix} satisfies $\widetilde Q Q = \Pi^+ + \Psi^{-\infty}$. 

\noindent
3. Since $K_L \in h \Psi_h^{-\infty}$ and the principal symbol of $h^{-1} K_L$ is supported away from the zero section, we see from Lemma~\ref{lem:contraction} that $K_L$ must be a contraction for small $h$. This means that $I - K_L \tilde b^*_\epsilon$ is invertible, so it is really a negligible error term for the sake of proving a high frequency estimate.

\begin{proof}
    1. Fix a cutoff function $\tilde \chi \in \CIc([0, \infty); \R)$ so that $\tilde \chi(\xi) = 1$ for all $\xi \in \supp \chi$, where $\chi$ is the cutoff fixed in \eqref{eq:cutoffs}. 
 By Lemma~\ref{lem:moser_conjugator}, there exists $g \in S_h^{-\infty} (T^*_{+} \mathbb{S}^1)$ and $\kappa_0 \in \C$ that solves the approximate cohomological equation
    \begin{equation}\label{eq:first_it}
        -\xi \tilde \chi(\xi) \zeta_\omega^+ (x) + g(\tilde b_\epsilon(x), \xi) - g(x, \xi) = - \xi \tilde \chi(\xi) \kappa_0 + \mathcal O(h^\infty)_{S^{-\infty}(T^*_{+} \mathbb S^1)}.
    \end{equation}
    Note that in the region $\Lambda$, the $\mathcal O(\epsilon^\infty)$ remainder in Lemma~\ref{lem:moser_conjugator} can be considered as an $\mathcal O(h^\infty)$ remainder, and we suppressed the $\omega$-dependence of $g$ and $\kappa_0$ in the notation. To be precise, $g \in S_h^{-\infty} (T_+^* \mathbb{S}^1)$ uniformly in $\omega$, and 
    \begin{equation}
        \kappa_0 = \int_0^1 \zeta_\omega^+(x) \, dx.
    \end{equation}
    Since $\Re \zeta_\omega^+ \ge \ell > 0$, uniformly in $\omega$, we see that $\Re \kappa_0 \ge \ell > 0$ uniformly in $\omega$ as well.

    Define
    \begin{equation}\label{eq:Q_0}
        q_0(x, \epsilon) := e^{g(x, \xi)}, \qquad Q_0 := \Op^+_h(q_0).
    \end{equation}
    Note that $q_0 \in S_h^0(T_+^* \mathbb S^1)$ is an elliptic symbol. Also define 
    \[m_0(x, \xi) := \chi(\xi)e^{-\xi \kappa_0}, \qquad M_0 := \Op_h^+(m_0).\]
    Note that $M_0$ is a Fourier multiplier. Now observe that from (\ref{eq:first_it}) and the change of variables formula, we have the principal symbol
    \begin{align*}
        \sigma_h^+ \big(B_L^+ (\tilde b^*_\epsilon Q_0 \tilde b^{-*}_\epsilon) - Q_0 M_0\big) (x, \xi) &= \chi(\xi) e^{-\xi \zeta^+_\omega(x)} e^{g(\tilde b_\epsilon(x), \xi)} - \chi(\xi) e^{g(x, \xi)} e^{-\xi \kappa_0} \\
        &= \chi(\xi) e^{g(x, \xi) - \xi \kappa_0} \\
        & \qquad \cdot \big[\exp \big(-\xi \tilde \chi(\xi) \zeta_\omega^+(x) + g(\tilde b_\epsilon(x), \xi) \\
        &\qquad \hspace{2cm} - g(x, \xi) + \xi \tilde \chi(\xi) \kappa_0 \big) - 1 \big] \\
        &= 0
    \end{align*}
    where the equalities are in the sense of principal symbols. Therefore
    \begin{equation*}
        B_L^+ (\tilde b^*_\epsilon Q_0 \tilde b^{-*}_\epsilon) - Q_0 M_0 \in h \Psi_{h, +}^{-\infty}(T^* \mathbb S^1) + \Psi^{-\infty}(T^* \mathbb S^1).
    \end{equation*} 
    Then there exists $a_1 \in S_h^{-\infty}(T^*_+ \mathbb S^1)$ such that 
    \begin{equation}\label{first_it_remains}
        B_L^+ (\tilde b^*_\epsilon Q_0 \tilde b^{-*}_\epsilon) - Q_0 M_0 + h \Op^+_h(a_1) \in h^2 \Psi_{h, +}^{-\infty}(T^* \mathbb S^1) + \Psi^{-\infty}(T^* \mathbb S^1).
    \end{equation}
    
    \noindent
    2. Now we add corrections to the operator $Q_0$ and the multiplier $M_0$. Again by Lemma~\ref{lem:moser_conjugator}, there exists $g_1 \in S_h^{-\infty}(T^*_+ \mathbb S^1)$ that satisfies
    \begin{equation}\label{eq:second_it}
        e^{-g(x, \xi)} a_1(x, \xi) + g_1(\tilde b_\epsilon(x), \xi) - g_1(x, \xi) = m_1(\xi) + \mathcal O(h^\infty)_{S^{-\infty}(T^*_+ \mathbb S^1)}
    \end{equation}
    where 
    \[m_1(\xi) := \int_0^1 e^{-g(x, \xi)} a_1(x, \xi)\, dx\]
    is $x$-independent and belongs in $S_h^{-\infty}(T^*_+ \mathbb S^1)$. Put
    \begin{equation*}
        M_1 := \Op_h^+(m_1)
    \end{equation*}
    and
    \begin{equation*}
        q_1(x, \xi):= g_1(x, \xi) e^{\xi \tilde \chi(\xi) \kappa_0 + g(x, \xi)}, \qquad Q_1 := \Op_h^+(q_1).
    \end{equation*}
    Observe that by (\ref{eq:first_it}), 
    \begin{align}
        \sigma_h^+ (B_L^+ \tilde b_\epsilon Q_1 \tilde b_\epsilon^{-*}) &= \chi(\xi) e^{-\xi \zeta_\omega(x)} q_1(\tilde b_\epsilon(x), \xi) \nonumber \\
        &= \chi(\xi) e^{-\xi \zeta_\omega(x) + \xi \kappa_0 + g(\tilde b_\epsilon(x), \xi)} g_1(\tilde b_\epsilon(x), \xi) \nonumber \\
        &= e^{g(x, \xi)} g_1(\tilde b_\epsilon(x), \xi) + (1 - \chi(\xi)) e^{-\xi \zeta_\omega(x)} q_1(\tilde b_\epsilon(x), \xi). \label{eq:BbQb_symbol}
    \end{align}
    Here, the equalities are in the sense of principal symbols. Furthermore, we have the principal symbols 
    \begin{equation}\label{eq:Q0M1_symbol}
        \sigma_h^+(Q_0 M_1) = e^{g(x, \xi)} m_1(\xi),
    \end{equation}
    and
    \begin{equation}\label{eq:Q1M0_symbol}
        \sigma_h^+(Q_1 M_0) = \chi(\xi) e^{-\xi \kappa_0 + \xi \tilde \chi(\xi) \kappa_0 + g(x, \xi)} g_1(x, \xi) = \chi(\xi) e^{g(x, \xi)} g_1(x, \xi).  
    \end{equation}
    Combining (\ref{eq:BbQb_symbol}), (\ref{eq:Q0M1_symbol}), and (\ref{eq:Q1M0_symbol}) using the approximate cohomological equation (\ref{eq:second_it}), we see that 
    \begin{multline}\label{eq:remainder_semiclassical}
        a_1(x, \xi) + \sigma_h^+\left(B_L^+ \tilde b_\epsilon^* Q_1 \tilde b_\epsilon^{-*} - Q_0 M_1 - Q_1 M_0\right) \\
        = (1 - \chi(\xi)) \left(e^{-\xi \zeta_\omega(x)} q_1(\tilde b_\epsilon(x), \xi) - e^{g(x, \xi)} g(x, \xi) \right)
    \end{multline}
    where the equality is modulo $h S_h^{-\infty}(T^*_+ \mathbb S^1)$, which does not affect the principal symbol. In particular, the crucial property of (\ref{eq:remainder_semiclassical}) is that the right hand side is clearly is supported away from $0$, and it is a standard semiclassical symbol. Hence it follows that
    \begin{multline}
        B_L^+\tilde b_\epsilon^*(Q_0 + h Q_1) \tilde b_\epsilon^{-*} - (Q_0 + h Q_1)(M_0 + h M_1) \\
        \in h^2 \Psi^{-\infty}_{h, +} (T_+^* \mathbb S^1) + h\Psi^{-\infty}_h(T^* \mathbb S^1) + \Psi^{-\infty}(T^* \mathbb S^1).
    \end{multline}

    \noindent
    3. Now we can inductively construct $Q_k \in \Psi^{-\infty}_{h, +}$ and Fourier multipliers $M_k \in \Psi^{-\infty}_{h, +}$ so that
    \begin{multline}\label{eq:moser_induct}
        B_L^+\tilde b_\epsilon^* \left( \sum_{k = 0}^N h^k Q_k \right) \tilde b_\epsilon^{-*} - \left( \sum_{k = 0}^N h^k Q_k \right) \left( \sum_{k = 0}^N h^k M_k \right) \\
        \in h^{N + 1} \Psi^{-\infty}_{h, +} (T_+^* \mathbb S^1) + h\Psi^{-\infty}_h(T^* \mathbb S^1) + \Psi^{-\infty}(T^* \mathbb S^1)
    \end{multline}
    for every $N \in \N$. Indeed, if (\ref{eq:moser_induct}) holds for some $N \in \N$, then there exists $a_{N + 1} \in S^{-\infty} (T^*_+ \mathbb S^1)$ such that 
    \begin{multline}
        B_L^+\tilde b_\epsilon^* \left( \sum_{k = 0}^N h^k Q_k \right) \tilde b_\epsilon^{-*} - \left( \sum_{k = 0}^N h^k Q_k \right) \left( \sum_{k = 0}^N h^k M_k \right) + h^{N + 1} \Op^+_h(a_{N + 1}) \\
        \in h^{N + 2} \Psi^{-\infty}_{h, +} (T_+^* \mathbb S^1) + h\Psi^{-\infty}_h(T^* \mathbb S^1) + \Psi^{-\infty}(T^* \mathbb S^1).
    \end{multline}
    By Lemma~\ref{lem:moser_conjugator}, there exists $g_{N + 1} \in S_h^{-\infty}(T^*_+ \mathbb S^1)$ that satisfies
    \begin{multline}\label{eq:cohom_mess}
        e^{-g(x, \xi)} a_{N + 1}(x, \xi) + g_{N + 1}(x + \alpha(\epsilon) + r_\epsilon(x), \xi) \\
        - g_{N + 1}(x, \xi) = m_{N + 1}(\xi) + \mathcal O(h^\infty)_{S^{-\infty}(T^*_+ \mathbb S^1)},
    \end{multline}
    where 
    \[m_{N + 1} (\xi):= \int_0^1 e^{-g(x, \xi)} a_{N + 1}(x, \xi) \, dx.\]
    Put 
    \begin{gather*}
        M_{N + 1} = \Op_h^+(m_{N + 1}), \\
        q_{N + 1}(x, \xi):= g_{N + 1}(x, \xi) e^{\xi \tilde \chi(\xi) \kappa_0 + g(x, \xi)}, \qquad Q_{N + 1} = \Op_h^+(q_{N + 1}).
    \end{gather*}
    Then with the same computation as (\ref{eq:remainder_semiclassical}) using the cohomological equation (\ref{eq:cohom_mess}) instead of (\ref{eq:second_it}), we complete the induction and recover (\ref{eq:moser_induct}). Put
    \[Q \sim \sum_{k = 0}^\infty h^k Q_k \in \Psi^0_{h, +}, \qquad \widetilde M \sim \sum_{k = 0}^\infty h^k M_k \in \Psi^{-\infty}_{h, +}.\]
    The sum here is again in the asympotitic sense, and existence of $Q$ and $\widetilde M$ are guaranteed by Borel's lemma. Futhermore, we are free to choose $\widetilde M$ so that the full symbol depends only on $\xi$ since $M_k$ depends only on $\xi$ for every $k \in \N_0$, and we see that the principal symbol of $\widetilde M$ is given by
    \begin{equation}
        \sigma^+_h(\widetilde M) = \chi(\xi) e^{-\xi \kappa_0}.
    \end{equation}
    Define
    \begin{equation}
        M = \Op_h^+(\chi(\xi) e^{-\xi \kappa_0}) + \Op_h^+(\chi(\xi)) \left(\widetilde M - \Op_h^+(\chi(\xi) e^{-\xi \kappa_0}) \right). 
    \end{equation}
    $M$ simply truncates the subprincipal parts of $\widetilde M$, and $M - \widetilde M \in h \Psi^{-\infty}_h$ is a uniform contraction for all sufficiently small $h$. Furthermore, $M$ is still a Fourier multiplier. 
    
    Note that $Q$ is elliptic in the positive semiclassical calculus, so there exists $\widetilde K_L \in h\Psi^{-\infty}_h(T^* \mathbb S^1)$ and $\tilde {\mathcal R}_L \in \Psi^{-\infty}(T^* \mathbb S^1)$ such that 
    \[B_L^+\tilde b_\epsilon^* Q \tilde b_\epsilon^{-*} - Q M = \widetilde K_L + \widetilde {\mathcal R}_L.\]
    $Q$ is elliptic in the positive semiclassical calculus, so we have a parametrix $\widetilde Q \in \Psi^0_{h, +}$. Then setting 
    \[K_L := \widetilde Q \widetilde K_L \in h\Psi^{-\infty}_h(T^* \mathbb S^1), \quad \mathcal R_L := \widetilde Q \widetilde {\mathcal R}_L \in \Psi^{-\infty}(T^* \mathbb S^1),\]
    we have (\ref{eq:low_freq_conj}) as desired. 
\end{proof}

The microlocal weight $Q$ from above has the positive semiclassical principal symbol $e^{g(x, \xi)}$ where $g$ is compactly supported. For large frequencies $h \xi \gtrsim 1$, this weight looks like the identity to the leading order. Therefore, morally speaking, conjugating the high frequency piece $B_H^+$ by $Q$ has negligible effect. In the following lemma, we show that the high frequency $B_H^+$ remains uniformly contracting upon conjugating by $Q$. 
\begin{lemma}\label{lem:high_freq_conj}
    Let $Q$ and $\widetilde Q$ be the operators constructed in Lemma~\ref{lem:low_freq_conj}. Then
    \[\widetilde Q B_H^+ \tilde b_\epsilon^* Q = K \tilde b_\epsilon^*\]
    where $K \in \Psi_h^{-\infty}(\mathbb S^1)$ and
    \begin{equation*}
        \WF_h(K) \subset \{\xi \ge 1\}.
    \end{equation*}
    Furthermore, there exists $0 < c < 1$ such that the semiclassical principal symbol satisfies
    \begin{equation}\label{eq:K_bound}
        0 \le \sigma_h(K)(x, \xi) \le c.
    \end{equation}
    In particular, 
    \[\|K\|_{H^s \to H^s} \le c < 1\]
    for all sufficiently small $h$. 
\end{lemma}

\begin{proof}
    Note that $K$ is simply given by 
    \[K = \widetilde Q B_H^+ (\tilde b_\epsilon^* Q \tilde b_\epsilon^{-*}).\]
    Since $\WF_h(B_H^+) \subset \{\xi \ge 1\}$, it is then clear that $\WF_h(K) \subset \{\xi \ge 1\}$. 

    Fix a microlocal cutoff $A \in \Psi_h^1(T^* \mathbb S^1)$ such that 
    \[\WF_h(A) \subset \{\xi \ge 1/2\}, \qquad \WF_h(I - A) \subset \{ \xi < 1\}.\] 
    Then we can rewrite
    \[K = (\widetilde Q A) B_H^+ (A \tilde b_\epsilon^* Q \tilde b_\epsilon^{-*}).\]
    The advantage here is that $\widetilde QA$ and $(A \tilde b_\epsilon^* Q \tilde b_\epsilon^{-*})$ are regular semiclassical operators, and it follows tha $K \in \Psi_h^{-\infty}(T^* \mathbb S^1)$ from the semiclassical calculus.
    
    Finally, we can compute the principal symbol using the symbol formulas~\eqref{eq:high_freq} and~\eqref{eq:Q_0} as well as the cohomological equation (\ref{eq:first_it}) to find that 
    \begin{align*}
        \sigma_h(K)(x, \xi) &= (1 - \chi(\xi)) e^{-g(x, \xi)} e^{-\xi \zeta^+_\omega(x)} e^{g(x + \alpha(\xi) + r_\epsilon(x), \xi)} \\
        &= (1 - \chi(\xi)) \exp\left(-(1 - \tilde \chi(\xi)) \xi \zeta_\omega^+(x) - \xi \tilde \chi(\xi) \kappa_0 \right).
    \end{align*}
    Again, we have the slight abuse of notation where the equalities above are really equivalence in the quotient $S^{-\infty}_h/hS^{-\infty}_h$. Recall that $\chi,\, \tilde \chi \in \CIc([0, \infty); [0, 1])$. Since $\Re \zeta_\omega^+ \ge \ell > 0$ and $\Re \kappa_0 \ge \ell > 0$, we conclude that $\sigma_0(K) < C < 1$. 
\end{proof}

What we have accomplished in Lemma~\ref{lem:low_freq_conj} and Lemma~\ref{lem:high_freq_conj} is that we have conjugated the operator $\widetilde B_\omega^{\pm} (\tilde b_\epsilon^{\pm1})^*$ so that the low frequency part behaves like a Fourier multiplier and the high frequency part is a contraction. For a high frequency estimate, we wish to invert the operator $I - \widetilde B_\omega^{\pm}$. The high frequency part is no problem since $I - K_{L}$ is invertible by Neumann series. For the low frequency piece, we use the fact that the rotation $(\tilde b_\epsilon^{\pm1})^*$ is Diophantine. 

\subsection{The high frequency estimate}
Now we finally prove the high frequency estimate for the boundary problem~\eqref{eq:v=Bb+g}. We preface the actual estimate with a simple model case that captures the key features of~\eqref{eq:v=Bb+g} that lead to the high frequency estimates. 

\subsubsection{The model case} \label{sec:model_case} 
We have split~\eqref{eq:v=Bb+g} into positive and negative frequencies by~\eqref{eq:split_freq}. We will see that~\eqref{eq:split_freq} is basically a cohomological equation with damping. In particular, let us suppose for the moment that $\widetilde B_\omega^+$ is simply a Fourier multiplier $\widetilde B^+ = \Op_h^+(e^{-\kappa_0\xi})$ and assume that $\kappa_0 > 0$ is real and $\omega$-independent. Also suppose that the conjugated chess billiard map $\tilde b_\epsilon$ is a pure rotation (without the error $r_{\epsilon}(x)$). Consider $\alpha \in (0, 1)$ Diophantine with constants $c, \beta > 0$, i.e. 
\[\left| \frac{p}{q} - \alpha \right| \ge \frac{c}{q^{2 + \beta}}\]
for all $p \in \Z$ and $q \in \N$. Let $\rho_\alpha: \mathbb S^1 \to \mathbb S^1$ by $\rho_\alpha(x) = x + \alpha$ is a rotation by $\alpha$. We wish to obtain high frequency estimates for $v$ in terms of $f$ satisfying a simplified version of~\eqref{eq:split_freq}:
\begin{equation}\label{eq:model_case}
    f = (I - \widetilde B^+ \rho_\alpha^*) v, \quad \hat v(k) = 0 \quad \text{for all} \quad k \le 0.
\end{equation}
Taking the Fourier series of both sides,
\[\hat f(k) = (1 - e^{2 \pi i k \alpha - 2 \pi hk \kappa_0}) \hat v(k), \quad k > 0.\]
Observe that
\[\left|\frac{1}{1 - e^{2 \pi i k \alpha - 2 \pi h k \kappa_0}} \right| \le 2 c^{-1} k^{1 + \beta}\]
for all $h \ge 0$. The key takeaway here is that the extra damping only improves the Diophantine estimate. Therefore
\[\|v\|_{H^s} \lesssim \|f\|_{H^{s + 1 + \beta}}\]
uniformly as $h \to 0$. 

Now we wish to pull high frequency estimates out of~\eqref{eq:v=Bb+g}, which is split into positive and negative frequencies by~\eqref{eq:split_freq}. The strategy is to conjugate (\ref{eq:v=Bb+g}) by the elliptic operator $Q$ constructed in Lemma~\ref{lem:moser_conjugator} so that it almost looks like the model case (\ref{eq:model_case}). However, there are two burdensome technical difficulties in the actual problem that differ from the model case:

\begin{enumerate}
    \item Only a low frequency part is close to a Fourier multiplier, but this will be enough since the high frequency part will be a contraction, which can be inverted using Neumann series. 

    \item The rotation is nearly Diophantine rather than precisely Diophantine, but this is tempered by the damping from the pseudodifferential piece $B_L^+$.
\end{enumerate}
Finally, we mention an additional minor inconvenience: $\kappa_0$ has an imaginary part. For high frequencies, damping dominates so this does not matter. For sufficiently small frequencies, we will show that the rotation dominates the imaginary part of $\kappa_0$ as $h \to 0$.

\subsubsection{High-low frequency splitting} 
Again we focus on the positive frequency case since the negative frequencies are treated similarly. We have already split into high and low frequency pieces once with $B_L^+$ and $B_H^+$ in (\ref{eq:low_freq}) and (\ref{eq:high_freq}). The purpose was so that we can conjugate the low frequency piece to a Fourier multiplier plus a contraction. The contraction term is entirely technical as it appeared as subprinciple contributions supported away from the zero section. Nevertheless, we really need to isolate the Fourier multiplier, and to do so, we need some more cutoffs.

Define a pseudodifferential partition of unity as follows. Fix cutoffs  
\begin{gather*}
    \chi_1, \chi_2 \in C^\infty([0, \infty); [0, 1]),\\
    \chi_1 + \chi_2 = 1, \quad \supp \chi_1 \subset [0, 2/3), \quad \supp \chi_2 \subset (1/2, \infty).
\end{gather*}
Then we can directly define the Fourier multipliers
\begin{equation}
    \Pi_1 = \Op_h^+(\chi_1(\xi)), \quad \Pi_2 = \Op_h(\chi_2(\xi)), \quad \Pi_1 + \Pi_2 = \Pi^+.
\end{equation}
Recall that we fixed $\lambda_0$ so that $\Omega$ is $\lambda_0$-simple and $\mathbf r(\lambda_0)$ is Diophantine. Again, we put $\omega = \lambda_0 + \epsilon + ih$ as the spectral parameter. Let $Q$ be the microlocal weight constructed in Lemma~\ref{lem:low_freq_conj} and $\widetilde Q$ its elliptic parametrix, and let $\varphi_\epsilon$ be the change of variables constructed in Lemma~\ref{lem:approx_conj}. Assume that $v_\omega$ and $f_\omega$ satisfy (\ref{eq:v=Bb+g}). Consider the positive frequency projections $v_\omega^+ = \Pi^+ v_\omega$ and $f_\omega^+ = \Pi^+ f_\omega$ as in (\ref{eq:v,f_projections}). In particular, these projections satisfy (\ref{eq:split_freq}) with $\lambda = \lambda_0 + \epsilon$. Put     
\begin{equation}\label{eq:tilde_v_def}
    \tilde v_\omega^+ = \widetilde Q \varphi_\epsilon^* v_\omega^+, \qquad \tilde f_\omega^+ = \widetilde Q \varphi_\epsilon^* f_\omega^+.
\end{equation}
\begin{lemma}\label{lem:more_high_low_split}
    Let $\omega = \lambda_0 + \epsilon + ih \in \Lambda^+_{\lambda_0, d} \setminus \{\lambda_0\}$. Projecting onto low and high frequencies, $v_\omega^+$ and $f_\omega^+$ must satisfy
    \begin{gather}
        \Pi_1 \tilde v_\omega^+ = M \rho^*_{\alpha(\epsilon)} \Pi_1 \tilde v_\omega^+ + \Pi_1 \tilde f_\omega^+ + \mathcal R_1 v_\omega, \label{eq:low_freq_proj} \\
        \Pi_2 \tilde v_\omega^+ = \widetilde K \rho_{\alpha(\epsilon)}^* \Pi_2 \tilde v_\omega^+ + \Pi_2 \tilde f_\omega^+ + \mathcal R_2 v_\omega. \label{eq:high_freq_proj}
    \end{gather}
    Here $M$ is the Fourier multiplier constructed in Lemma~\ref{lem:low_freq_conj}, and $\widetilde K \in \Psi^{-\infty}_h(\mathbb S^1)$ is uniformly a contraction for all sufficiently small $h$, i.e.
    \[\|K\|_{H^s \to H^s} \le c < 1\]
    where $c$ is independent of $\omega \in \Lambda$ for all sufficiently small $h$. Furthermore, $\mathcal R_1, \, \mathcal R_2 \in \Psi^{-\infty}$ uniformly for $\omega \in \Lambda^+_{\lambda_0, d} \setminus \{\lambda_0\}$. 
\end{lemma}
\begin{proof}
    1. We begin by applying $\widetilde Q \varphi_\epsilon^*$ to both sides of (\ref{eq:split_freq}), we have
    \begin{align}
        \tilde v^+_\omega &= \widetilde Q \varphi_\epsilon^* \left(B_\omega^+ b_{\lambda_0 + \epsilon}^* v_\omega + \mathscr R^+_\omega v_\omega^+ + f_\omega^+\right) \nonumber \\
        &= (M + K_L + K) \tilde b^*_\epsilon \tilde v_\omega^+ + \mathcal R^+_\omega v_\omega + \tilde f^+_\omega. \label{eq:M+K+R}
    \end{align}
    Here, $M$, $K$, and $K_L$ are the operators constructed in Lemma~\ref{lem:low_freq_conj} and Lemma~\ref{lem:high_freq_conj}, and we have the smoothing operator 
    \[\mathcal R_\omega^+ :=  \widetilde Q \widetilde B_\omega^+ \tilde b_\epsilon^* (I - \widetilde Q Q) \varphi_\epsilon^* \Pi^+ + \widetilde Q \varphi_\epsilon^* \mathscr R^+_\omega + \mathcal R_L \tilde b_\epsilon^* \widetilde Q \varphi_\epsilon^*,\]
    which lies in $\Psi^{-\infty}$ uniformly in $\omega$ since $I - \widetilde Q Q$, $\mathscr R_\omega^+$, and $\mathcal R_L$ all lie in $\Psi^{-\infty}$ uniformly in $\omega$.

    \noindent
    2. Recall that $\tilde b_\epsilon(x) = x + \alpha(\epsilon) + r_\epsilon(x)$ where $r_\epsilon(x) = \mathcal O(\epsilon^\infty)_{C^\infty(\mathbb S^1)}$, and since we are working in the region $\Lambda^+_{\lambda_0, d}$, we have $r_\epsilon(x) = \mathcal O(h^\infty)$. Put
    \[J_\epsilon = \tilde b_\epsilon^* - \rho_{\alpha(\epsilon)}^*.\]
    For any 1-form $v(x)dx \in C^\infty(\mathbb S^1; T^* \mathbb S^1)$, we have
    \begin{align*}
        (J_\epsilon v)(x) &= v(x + \alpha(\epsilon) + r_\epsilon(x)) (1 + \partial_x r_\epsilon(x)) - v(x + \alpha(\epsilon)) \\
        &= v(x + \alpha(\epsilon) + r_\epsilon(x)) \partial_x r_\epsilon(x) + \int_0^1 r_\epsilon(x) v'(x + \alpha(\epsilon) + t r_\epsilon(x)) \, dt. 
    \end{align*}
    Then we have the mapping property
    \[\|J_\epsilon\|_{H^s \to H^{s - 1}}  = \mathcal O(h^\infty).\]
    Since $M, K \in \Psi_h^{-\infty}$ uniformly in $h$, we also have the mapping property
    \begin{equation}\label{eq:jiggle_smoothing}
    \|MJ_\epsilon\|_{H^s \to H^k} = \mathcal O(h^\infty)
    \end{equation}
    for all $s, k \in \R$. 

    \noindent
    3. From Lemmas \ref{lem:low_freq_conj} and \ref{lem:high_freq_conj} we see that $\WF_h(K + K_L) \subset \{\xi \ge 1\}$. Therefore, $\Pi_1$ is microlocally $0$ near $\WF_h(K + K_L)$, and $\Pi_2$ is microlocally the identity near $\WF_h(K)$, so
    \begin{equation}\label{eq:Pi_on_K}
        \Pi_1 (K + K_L) \in \Psi^{-\infty}, \qquad [\Pi_2, K] \in \Psi^{-\infty}.
    \end{equation}
    uniformly in $\omega$. Since $\Pi_1$, $\Pi_2$, and $M$ are all Fourier multipliers,
    \begin{equation}\label{eq:Pi_on_M}
        [\Pi_1, M] = [\Pi_2, M] = 0. 
    \end{equation}
    Finally, define 
    \begin{equation}\label{eq:tilde_K}
        \widetilde K = K + K_L + M\Op_h(\chi_3(\xi))
    \end{equation}
    where $\chi_3 \in C^\infty((0, \infty), [0, 1])$ and equals to $1$ near $\supp \chi_2$. Note that $\widetilde K \in \Psi^{-\infty}_h$, and it follows from (\ref{eq:M_symbol}), (\ref{eq:K_bound}), and the fact that $K_L \in h \Psi^{-\infty}_h$ that there exists a constant $0 < C < 1$ such that 
    \begin{equation*}
        0 \le \sigma_h(\widetilde K) \le C,
    \end{equation*}
    which implies that $\widetilde K$ is a contraction for all sufficiently small $h$. Applying $\Pi_1$ to (\ref{eq:M+K+R}) and using properties (\ref{eq:Pi_on_K}) and (\ref{eq:Pi_on_M}), we find
    \begin{align}
        \Pi_1 \tilde v^+_\omega &= \Pi_1 M (\rho_{\alpha(\epsilon)}^* + J_\epsilon) \tilde v_{\omega}^+ + \Pi_1 (K + K_L) \tilde b_\epsilon^* \tilde v_\omega^+ + \Pi_1 \mathcal R^+_\omega v_\omega + \Pi_1 \tilde f^+_\omega \nonumber\\
        &= M \rho^*_{\alpha(\epsilon)} \Pi_1 \tilde v_\omega^+ + \Pi_1 \tilde f^+_\omega + \big[\Pi_1 (K + K_L) \tilde b_\epsilon^* \tilde v_\omega^+ + \Pi_1 M J_\epsilon \tilde v_\omega^+ + \Pi_1 \mathcal R_\omega^+ v_\omega \big] \nonumber \\
        &= M \rho^*_{\alpha(\epsilon)} \Pi_1 \tilde v_\omega^+ + \Pi_1 \tilde f_\omega^+ + \mathcal R_1 v_\omega, \nonumber
    \end{align}
    where 
    \[\mathcal R_1 = \Pi_1 (K + K_L) \tilde b_\epsilon^* \widetilde Q \varphi_\epsilon^* \Pi^+ + \Pi_1 M J_\epsilon \widetilde Q \varphi_\epsilon^* \Pi^+ + \Pi_1 \mathcal R_\omega^+. \]
    The exact formula for the remainder is unimportant, but using (\ref{eq:jiggle_smoothing}) and (\ref{eq:Pi_on_K}), the key property we see from the exact formula is that $\mathcal R_1 \in \Psi^{-\infty}$ uniformly in $h$. 

    Similarly, to obtain the high frequency piece, we apply $\Pi_2$ to (\ref{eq:M+K+R}) and again use properties (\ref{eq:Pi_on_K}) and (\ref{eq:Pi_on_M}) to find
    \begin{align}
        \Pi_2 \tilde v_\omega^+ &= \Pi_2 (M + K + K_L)(\rho_{\alpha(\epsilon)}^* + J_\epsilon) \tilde v_\omega^+ + \Pi_2 \mathcal R_\omega^+ v_\omega + \Pi_2 \tilde f_\omega^+ \nonumber \\
        &= \widetilde K \rho_{\alpha(\epsilon)}^* \Pi_2 \tilde v_\omega^+ + \Pi_2 \tilde f_\omega^+ + \mathcal R_2 v_\omega \nonumber
    \end{align}
    where
    \[\mathcal R_2 = [\Pi_2, K + K_L] \tilde b_\epsilon^* \widetilde Q \varphi_\epsilon^* \Pi^+ + \Pi_2 M J_\epsilon \widetilde Q \varphi_\epsilon^* \Pi^+ + \Pi_2 \mathcal R_\omega^+.\]
    Again, the exact formula is unimportant, but analyzing each term using (\ref{eq:jiggle_smoothing}) and (\ref{eq:Pi_on_K}), we see that $\mathcal R_2 \in \Psi^{-\infty}$ uniformly in $h$. 
\end{proof}

\subsubsection{The main estimate}
Now we are ready to assemble all the pieces together for the high frequency estimate. This differs from the model case in \S\ref{sec:model_case} primarily in that the rotation is $\epsilon$-close to a Diophantine rotation, so the damping becomes crucial. 
\begin{proposition}\label{prop:high_freq_est}
Let $\lambda_0$ be such that $\mathbf r(\lambda_0)$ is Diophantine with the constant $\beta$ given in Definition \ref{def:diophantine}, and let $\Lambda = \Lambda^+_{\lambda_0, d}$ be the polynomial region above $\lambda_0$ as defined in (\ref{eq:poly_region}). Suppose $f_\omega \in C^\infty(\mathbb S^1; T^* \mathbb S^1)$ and $v_\omega \in \mathcal C^\infty(\mathbb S^1; T^* \mathbb S^1)$ satisfy (\ref{eq:v=Bb+g}), then for all $s, N \in \R$ and $\omega\in\Lambda \setminus \{\lambda_0\}$,
\begin{equation}\label{eq:high_freq_est_1}
    \|v_\omega\|_{H^s} \le C \left(\|f_\omega\|_{H^{s + d(2 + \beta)}} + \|v_\omega\|_{H^{-N}} \right)
\end{equation}
for some constant $C = C_{\lambda_0, s, N, d}$. 

In particular, if $v_\omega = \mathcal N_\omega u_\omega$ where $u_\omega$ is the solution to (\ref{eq:EBVP}), then for every $s \in \R$, there exists $s' \in \R$ such that the estimate
\begin{equation}\label{eq:high_freq_est_2}
    \|v_\omega\|_{H^s} \lesssim \|\mathcal C_\omega v_\omega\|_{H^{s'}} + \|v_\omega\|_{H^{-N}}
\end{equation}
holds for $\omega \in \Lambda$, where the hidden constant is independent of $\omega \in \Lambda$.
\end{proposition}
\begin{proof}
    1. The estimate \eqref{eq:high_freq_est_2} follows immediately from \eqref{eq:high_freq_est_1} by unwinding the definitions. Indeed, note that if $u_\omega$ is the solution to \ref{eq:EBVP}, then $v_\omega = \mathcal N_\omega u_\omega$ satisfies (\ref{eq:v=Bb+g}) for 
    \[f_\omega = (I - \mathcal A_\omega) \mathcal E_\omega d\mathcal C_\omega v_\omega.\]
    Recall that $A_\omega$ and $\mathcal E_\omega$ are in $\Psi^{0}(\mathbb S^1; T^* \mathbb S^1)$ uniformly in $\omega$, so \eqref{eq:high_freq_est_2} follows by applying \eqref{eq:high_freq_est_1} to such $v_\omega$ and $f_\omega$.

    Now to prove \eqref{eq:high_freq_est_1} itself, we see from (\ref{eq:split_freq}) that positive and negative frequencies can be treated separately, and by symmetry of the positive and negative frequencies, it suffices to prove
    \begin{equation}\label{split_est}
        \|v_\omega^+\|_{H^s} \le  C \left(\|f_\omega\|_{H^{s + d(2 + \beta)}} + \|v_\omega\|_{H^{-N}} \right).
    \end{equation}
    For this, it suffices to estimate $\tilde v_\omega^+ = \tilde Q \varphi_\epsilon^* v_\omega^+$, where $\varphi_\epsilon$ is the the diffeomorphism constructed in Lemma~\ref{lem:approx_conj} and $Q$ is the elliptic microlocal weight constructed in Lemma~\ref{lem:low_freq_conj}   Following Lemma~\ref{lem:more_high_low_split}, we proceed by estimating $\Pi_1 \tilde v_\omega^+$ and $\Pi_2 \tilde v_\omega^+$ separately.

    \noindent
    2. In Steps 2 -- 5, we estimate the low frequency piece, $\Pi_1 \tilde v_\omega^+$, using the equation (\ref{eq:low_freq_proj}). To reduce a lot of the bulk of notation, we put $v :=\Pi_1\tilde v_\omega^+ \in C^\infty(\mathbb S^1)$. Note that it suffices to obtain estimates for $v$ in terms of $f$ given by
    \begin{equation}\label{eq:simple_high_freq}
        (I - M \rho^*_{\alpha(\epsilon)}) v = f
    \end{equation}
    where the Fourier coefficients of $v$ satisfies
    \[\hat v(k) = 0 \quad \text{for all} \quad k \le 0.\]
    By Lemma~\ref{lem:low_freq_conj}, $M \rho_{\alpha(\epsilon)}^*$ is simply a Fourier multiplier, so taking the Fourier series on both sides of (\ref{eq:simple_high_freq}) and using the principal symbol formula in Lemma~\ref{lem:low_freq_conj}, we see that 
    \begin{equation}\label{eq:multiplier_setup}
         \left( 1 - m(2 \pi hk) e^{2 \pi i k \alpha(\epsilon)} \right) \hat v(k) = \hat f(k),
    \end{equation}
    where we recall that
    \[m(\eta) = \chi(\eta) \left[\exp(-\eta \kappa_0) + \mathcal O(h)_{C^\infty([0, 2]_{\eta})} \right]\]
    where $\Re \kappa_0 \ge \ell > 0$ uniformly in $\omega$. We remark that $m$ and $\kappa_0$ both depend implicitly on $\omega = \lambda_0 + \epsilon + ih$, and that the remainder $\mathcal O(h)_{C^\infty([0, 2]_{\eta})}$ is uniform in $\epsilon$ and need only be defined on $[0, 2]$ since $\supp \chi \subset [0, 2]$. Then we can rewrite (\ref{eq:multiplier_setup}) as
    \begin{equation}
    \begin{gathered}
        \big[1 - \chi(2 \pi hk) \exp(\tau_\omega(k)) \big] \hat v(k) = \hat f(k), \\
        \text{where} \quad \tau_\omega(k) := 2 \pi i k \alpha(\epsilon) - 2 \pi hk \kappa_0 + \mathcal O(h)_{C^\infty([0, 2]_{h \xi})}.
    \end{gathered}
    \end{equation}
    The imaginary part of $\tau_\omega$ creates rotation and the real part of $\tau_\omega$ creates damping, so we need to estimate 
    \begin{equation}\label{eq:re_im}
        \begin{gathered}
        \Im \tau_{\omega}(k) = 2 \pi k [\alpha(\epsilon) - h \Im \kappa_0] + \mathcal O(h)_{C^\infty([0, 2]_{hk})}, \\
        \Re \tau_{\omega}(k) = -2 \pi hk \Re \kappa_0 + \mathcal O(h)_{C^\infty([0, 2]_{h\xi})}.
        \end{gathered}
    \end{equation}

    \noindent
    3. We first consider the imaginary part of $\tau_\omega$. Recall from Lemma~\ref{lem:approx_conj} that $\alpha$ is smooth near $0$ and $\omega \in \Lambda \setminus \{\lambda_0\}$, so 
    \begin{equation}\label{eq:im_tau}
        \Im \tau_\omega (k) = 2 \pi k [\alpha(0) - h \Im \kappa_0] + k \cdot \mathcal O(h^{1/d}) + \mathcal O (h)_{C^\infty([0, 2]_{h\xi})}.
    \end{equation}
    Furthermore, $\alpha(0) = \mathbf r(\lambda_0)$ is Diophantine, i.e. there exists constants $c, \beta > 0$ such that 
    \[\left|\alpha(0) - \frac{p}{q} \right| \ge \frac{c}{q^{2 + \beta}}\]
    for all $p \in \Z$ and $q \in \N$. Then it follows from (\ref{eq:im_tau}) and the Diophantine condition that there exists a small constant $c_1 > 0$ such that
    \begin{equation}\label{eq:im_estimate}
        \left|1 - e^{i \Im \tau_\omega(k)} \right| \ge \frac{c_1}{2 k^{1 + \beta}} \quad \text{for all $0 < k^{d(2 + \beta)} \le c_1 h^{-1}$ and $k > 0$.}
    \end{equation}
    
    \noindent
    4. Now we consider the real part of $\tau_\omega$. We consider the frequencies not covered by Step~3 in (\ref{eq:im_estimate}). Let $k^{d(2 + \beta)} \ge c_1 h^{-1}$, $k \in \Z_{> 0}$. Then taking  sufficiently small $h$, and thus $k$ will be sufficiently big, $2 \pi h k \Re \kappa_0$ will dominate in (\ref{eq:re_im}) for $\Re \tau_\omega(k)$:
    \begin{align}
        -\Re \tau_\omega(k) &= 2 \pi h k \Re \kappa_0 + \mathcal O(h)_{C^\infty([0, 2]_{h k})} \nonumber \\
        &\ge \pi h k \Re \kappa_0 \nonumber \\
        &\ge \frac{\pi c_1 \Re \kappa_0 }{k^{d(2 + \beta) - 1}}. \label{eq:re_estimate}
    \end{align}

    \noindent
    5. Now we combine the estimates in Steps 3 and~4 to obtain estimates for the low frequency piece, $\Pi_1 \tilde v_\omega^+$, which satisfies (\ref{eq:low_freq_proj}). In particular, it follows from (\ref{eq:im_estimate}) and (\ref{eq:re_estimate}) that for all sufficiently small $h$, $\omega \in \Lambda \setminus \{\lambda_0\}$, and $k > 0$
    \begin{align*}
        \left|1 - \chi(2 \pi h k) e^{\tau_\omega(k)} \right| &\ge \max \left\{ \left|1 - e^{\Re \tau_\omega(k)} \right|, \tfrac{1}{2} \left| 1 - e^{i\Im \tau_\omega (k)} \right| \right\} \\
        &\ge \frac{c_2}{k^{d(2 + \beta)}}
    \end{align*}
    for some small constant $c_2 > 0$. For $v$ and $f$ satisfying (\ref{eq:simple_high_freq}), we have
    \[\|v\|_{H^s} \lesssim \|f\|_{H^{s + d(2 + \beta)}}.\]
    Applying this estimate to (\ref{eq:low_freq_proj}) immediately yields
    \begin{equation}\label{eq:low_piece_est}
        \|\Pi_1 \tilde v_\omega^+\|_{H^s} \lesssim \|\tilde f_\omega^+\|_{H^{s + d(2 + \beta)}} + \|v_\omega\|_{H^{-N}}.
    \end{equation}

    \noindent
    6. Finally, it only remains to look at the high frequency piece. From (\ref{eq:tilde_K}) and the discussion right below it, $\widetilde K$ is uniformly a contraction for all sufficiently small $h$, and $\rho^*_{\alpha(\epsilon)}$ is an isometry since it is just a rotation. Therefore $(I - \widetilde K \rho^*_{\alpha(\epsilon)})^{-1}$ is uniformly bounded by Neumann series, so (\ref{eq:high_freq_proj}) then implies that 
    \begin{equation}\label{eq:high_piece_est}
        \|\Pi_2 \tilde v_\omega^+\|_{H^s} \lesssim \|\tilde f_\omega^+\|_{H^s} + \|v_\omega\|_{H^{-N}}.
    \end{equation}
    Combining (\ref{eq:low_piece_est}) and (\ref{eq:high_piece_est}) and unwinding the definitions from (\ref{eq:tilde_v_def}), we see that 
    \begin{equation}\label{eq:tilde_estimate}
        \|\widetilde Q \varphi_\epsilon^* v^+_\omega\|_{H^s} \lesssim \| \widetilde Q \varphi_\epsilon^* f_\omega^+\|_{H^{s + d(2 + \beta)}} + \|v_\omega\|_{H^{-N}}
    \end{equation}
    Recall that $\widetilde Q$ is an order $0$ elliptic operator in the positive semiclassical calculus and $\varphi^*_\epsilon$ is $H^{s} \to H^{s}$ bounded with a bounded inverse, so (\ref{eq:tilde_estimate}) recovers (\ref{split_est}) as desired. 
\end{proof}

\section{Limiting absorption principle}\label{sec:LAP}
The main purpose of this section is to study the limiting behavior solutions to the equation 
\[P(\omega) u_\omega = f, \qquad u_\omega|_{\partial \Omega} = 0\]
as $\omega$ approaches a $\lambda_0$ for which $\Omega$ is $\lambda_0$-simple (Definition \ref{lambda_simple_def}) and $\mathbf r(\lambda_0)$ is Diophantine (Definition \ref{def:diophantine}). Here, $P(\omega)$ is the operator defined in (\ref{eq:P(omega)}) and $f \in \CIc(\Omega)$. In particular, we consider $\omega$ in a polynomial region above or below $\lambda_0$ (see Figure \ref{fig:polyregion}). In these regions, we establish in Proposition \ref{prop:u_LAP} below that $u_\omega$ depend smoothly on $\omega$ with respect to all $C^\infty(\overline \Omega)$ seminorms.

To establish a limiting absorption principle, we first need a uniqueness result for the limiting problem. This eliminates the possibility of having an embedded spectrum for $P$ in the case that $\mathbf r(\lambda_0)$ is Diophantine. 

\subsection{Uniqueness for the limiting problem}
We first prove an interior uniqueness statement, appearing as early as in the work of John~\cite[Section~3]{John_41}.
\begin{lemma}\label{lem:interior_uniqueness}
    Fix $\lambda_0 \in (0, 1)$ such that $\Omega$ is $\lambda_0$-simple and $\mathbf r(\lambda_0)$ is Diophantine. If $u \in C^\infty(\overline \Omega)$ is a solution to
    \[P(\lambda_0) u = 0, \quad u|_{\partial \Omega} = 0,\]
    then $u = 0$. 
\end{lemma}
\begin{proof}
    Observe that (\ref{eq:L_factor}) and (\ref{eq:L_dual}) imply that
    \begin{equation}\label{eq:u_decomp}
        u = u_+(\ell^+(x)) - u_-(\ell^-(x))
    \end{equation}
    for some $u_+ \in C^\infty([\ell^+_{\min}, \ell^+_{\max}])$ and $u_- \in C^\infty([\ell^-_{\min}, \ell^-_{\max}])$, where 
    \[\ell^\pm_{\min} = \min\{\ell^\pm(x) \mid x \in \partial \Omega\} \quad \text{and} \quad \ell^\pm_{\max} = \max\{\ell^\pm(x) \mid x \in \partial \Omega\}.\]
    Define the restriction to boundary 
    \[w_\pm := u_\pm(\ell^\pm(x))|_{\partial \Omega} \in C^\infty (\partial \Omega).\]
    By construction, $w_\pm$ are invariant under the involution $(\gamma^\pm)^* w_\pm = w_\pm$. By assumption $u|_{\partial \Omega} = 0$, so we can define $w := w_+ = w_-$, which is then invariant under the chess billiard map:
    \[b^* w = w.\]
    Since $\mathbf r(\lambda_0)$ is Diophantine, $b$ is ergodic by Corollary \ref{cor:ergodic}, so $w$ must be a constant, and thus $u = 0$.
\end{proof}

Now for the boundary uniqueness, we need an additional support condition. This support condition appears since we do not know if the limiting single layer potential operator $S_{\lambda_0 \pm i0}$ is injective acting on smooth $1$-forms. Nevertheless, the following uniqueness lemma suffices. 
\begin{lemma}\label{lem:bd_uniqueness}
    Fix $\lambda_0 \in (0, 1)$ such that $\Omega$ is $\lambda_0$-simple and $\mathbf r(\lambda_0)$ is Diophantine. If $v \in C^\infty(\partial \Omega; T^* \partial \Omega)$ satisfies
    \begin{equation}\label{eq:uniqueness_assumptions}
        \mathcal C_{\lambda_0 + i0} v = 0 \quad \text{and} \quad \supp(R_{\lambda_0 + i0} \mathcal I v) \subset \overline \Omega,
    \end{equation}
    then $v = 0$.
\end{lemma}
\begin{proof}
    1. Put $U := R_{\lambda_0 + i0} \mathcal Iv$. Recall that $E_{\lambda_0 + i0} \in \mathcal D'(\R^2)$ is a fundamental solution, which satisfies $P(\lambda_0) E_{\lambda_0 + i0} =\delta_0$. Since $\mathcal I v$ is compactly supported, we have
    \begin{equation}\label{eq:PU=Iv}
        P(\lambda_0) U = \mathcal Iv.
    \end{equation}
    Let $u = U|_{\Omega} = S_{\lambda_0 + i0} v$. Then
    \begin{equation*}
        P(\lambda_0)u = 0.
    \end{equation*}
    By Lemma~\ref{lem:S_mapping_prop}, $u$ extends to the boundary to $u \in C^\infty(\overline \Omega)$. Furthermore, $u|_{\partial \Omega} = \mathcal C_{\lambda_0 + i0} v = 0$, so by Lemma~\ref{lem:interior_uniqueness}, we have $u = 0$. Then by the support assumption (\ref{eq:uniqueness_assumptions}), we see that 
    \begin{equation}\label{eq:U_support}
        \supp U \subset \partial \Omega.
    \end{equation}

    \noindent
    2. Now we deduce $v = 0$ from (\ref{eq:U_support}). We use the same method here as in the second step of \cite[Lemma 7.2]{DWZ}. Denote
    \[\mathscr C := \{x^+_{\mathrm{max}}, x^-_{\mathrm{max}}, x^+_{\mathrm{min}}, x^-_{\mathrm{min}}\}\]
    where $x^\pm_{\mathrm{min}}, x^\pm_{\mathrm{max}}$ are defined in \eqref{eq:characteristic_set}. Since $\Omega$ is $\lambda_0$-simple, we see that the characteristic set of the differential operator $P(\lambda_0)$ (i.e. the zero set of principal symbol of $P(\lambda_0)$) is disjoint from the conormal bundle of $\partial \Omega \setminus \mathscr C$. Therefore, for every $x_0 \in \partial \Omega \setminus \mathscr C$, there exists a neighborhood $V \subset \R^2$ of $x_0$ such that there exists coordinates $(y_1, y_2)$ on $V$ satisfying 
    \begin{equation}
        \partial \Omega \cap V = \{y_1 = 0, y_2 \in \mathcal J\},  \quad P(\lambda_0)|_V = \sum_{|\alpha| \le 2} a_{\alpha} \partial_y^\alpha, \quad a_{2, 0} \neq 0
    \end{equation}
    for some interval $\mathcal J \subset \R$. By \cite[Theorem 2.3.5]{H1}, \eqref{eq:U_support} implies that
    \[U|_V = \sum_{k \le K} u_k(y_2) \otimes \delta^{(k)} (y_1)\]
    for some $u_k \in \mathcal D'(\mathcal J)$. Therefore,
    \[\mathcal I v|_{V} = P(\lambda_0) U|_V = a_{2, 0}(y) u_K(y_2) \otimes \delta^{(K + 2)}(y_1) + \sum_{k \le K + 1} \tilde u_k(y_2) \otimes \delta^{(k)}(y_1)\]
    for some $\tilde u_k \in \mathcal D'(\mathcal J)$. On the other hand, we know that $\mathcal I v |_V = f(y_2) \otimes \delta(y_1)$ for some $f \in C^\infty(\mathcal J)$. Since $a_{2, 0} \neq 0$, we deduce that $u_{K}(y_2) = 0$. Iterating, we see that $U|_V = 0$. This means that $v|_{\partial \Omega \setminus \mathscr C} = 0$, and since $v \in C^\infty(\partial \Omega; T^* \partial \Omega)$, we conclude that $v = 0$.
\end{proof}

\subsection{Boundary limiting absorption}
Equipped with both a high frequency estimate and uniqueness for the limiting problem, we can now prove a convergence statement for the boundary data.

Fix $f \in C_c^\infty(\Omega)$ and $\lambda_0 \in (0, 1)$ so that $\Omega$ is $\lambda_0$-simple and the rotation number $\mathbf{r}(\lambda_0)$ is Diophantine. Let us focus on the polynomial region $\Lambda = \Lambda^+_{\lambda_0, d}$ above $\lambda_0$ as defined in (\ref{eq:poly_region}) for some $d \ge 2$. Recall from Lemma~\ref{lem:fs_app} that the elliptic boundary value problem (\ref{eq:EBVP}) given by
\begin{equation}\label{eq:BVP_reprise}
    P(\omega)u_\omega = f, \quad u_\omega|_{\partial \Omega} = 0, \quad \omega \in \Lambda \setminus \{\lambda_0\}
\end{equation}
has unique solution $u_\omega \in C^\infty(\overline \Omega)$ that depends holomorphically on $\omega$. We wish to establish a limiting absorption principle so that we can understand the evolution problem. We first look at the behavior as $\omega \to \lambda_0$ of the Neumann data.

\begin{proposition}\label{prop:v_LAP}
    Consider a sequence $\omega_j \to \lambda_0$, $\omega_j \in \Lambda\setminus \{\lambda_0\}$ for all $j$. Let $\lambda_0$ be such that $\Omega$ is $\lambda_0$-simple and $\mathbf r(\lambda_0)$ is Diophantine. Let $u_{\omega_j} \in C^\infty(\overline \Omega)$ be the solutions to (\ref{eq:BVP_reprise}) and 
    \[v_{\omega_j} := \mathcal N_{\omega_j} u_{\omega_j} \in C^\infty (\partial \Omega; T^* \partial \Omega)\]
    be the corresponding Neumann data as defined in (\ref{eq:neumann_data}). Then 
    \[v_{\omega_j} \to v_{\lambda_0 + i0} \quad \text{in} \quad C^\infty(\partial \Omega, T^* \partial \Omega)\]
    where $v_{\lambda_0 + i0}$ is the unique smooth solution to 
\[\mathcal C_{\lambda_0 + i0} v_{\lambda_0 + i0} = (R_{\lambda_0 + i0} f)|_{\partial \Omega}, \quad \supp(R_{\lambda_0 + i0} (f - \mathcal I v_{\lambda_0 + i0})) \subset \overline \Omega.\]
\end{proposition}
\begin{proof}
    1. Recall that by restricting the identities in Lemma~\ref{lem:fs_app} to the boundary, we have the identities
    \begin{equation}\label{eq:Cv=Rf}
    \begin{gathered}
        \mathcal C_\omega v_\omega = (R_\omega f)|_{\partial \Omega}, \\
        \indic_\Omega u_\omega = R_\omega(f - \mathcal I v_\omega)
    \end{gathered}
    \end{equation}
    where $R_\omega f= E_\omega * f \in C^\infty(\R^2)$ where $E_\omega$ is the fundamental solution. By Lemma~\ref{lem:fs}, $E_\omega \to E_{\lambda_0 + i0}$ in distributions, so
    \begin{equation}\label{eq:Cv_to_Rf}
        \mathcal C_{\omega_j} v_{\omega_j} \to (R_{\lambda_0 + i0} f)|_{\partial \Omega} \quad \text{in } C^\infty (\partial \Omega).
    \end{equation}

    \noindent
    2. We first prove that $v_{\omega_j}$ must be bounded in $H^s(\partial \Omega; T^* \partial \Omega)$ for any $s > 0$. Assume for the sake of contradiction that $v_{\omega_j}$ is unbounded in $H^s(\partial \Omega; T^* \partial \Omega)$, i.e. upon passing to a subsequence, we may assume that $\|v_{\omega_j} \|_{H^s} \to \infty$. Define the rescaling 
    \begin{equation*}
        v_j = \frac{v_{\omega_j}}{\|v_{\omega_j}\|_{H^s}} \quad \text{and} \quad u_j = \frac{u_{\omega_j}}{\|v_{\omega_j}\|_{H^s}}.
    \end{equation*}
    We have compact embedding $H^s(\partial \Omega ; T^* \partial \Omega) \hookrightarrow H^{-s}(\partial \Omega; T^* \partial \Omega)$, so upon passing to a subsequence, we may assume that there exists $v_0 \in H^{-s}(\partial \Omega; T^* \partial \Omega)$ so that 
    \begin{equation*}
        v_j \to v_0 \quad \text{in} \quad H^{-s}(\partial \Omega; T^* \partial \Omega).
    \end{equation*}
    By (\ref{eq:Cv_to_Rf}), $\mathcal C_{\omega_j} v_j$ is bounded in $H^{s'}$ for every $s'$, so by the high frequency estimate Corollary \ref{prop:high_freq_est}, $v_j$ is bounded in $H^{s + 1}$. Therefore we have the upgraded convergence
    \begin{equation*}
        v_j \to v_0 \quad \text{in} \quad H^{s}(\partial \Omega; T^* \partial \Omega),
    \end{equation*}
    which means $\|v_0\|_{H^s} = 1$. On the other hand, we see from (\ref{eq:Cv_to_Rf}) that 
    \begin{equation*}
        \mathcal C_{\omega_j} v_j \to 0 \quad \text{in} \quad C^\infty(\partial \Omega).
    \end{equation*}
    It follows from Proposition \ref{prop:C_converge} that $\mathcal C_{\omega_j} v_j \to \mathcal C_{\lambda_0 + i0} v_0$ in (at the very least) $\mathcal D'(\partial \Omega)$. Therefore
    \[\mathcal C_{\lambda_0 + i0} v_0 = 0.\]
    Furthermore, from (\ref{eq:Cv=Rf}), we see that 
    \[R_{\lambda_0 + i0} \mathcal I v_0 = -\lim_{j \to \infty} \indic_\Omega u_j \implies \supp(R_{\lambda_0 + i0} \mathcal I v_0) \subset \overline \Omega\]
    where the limit is taken in $\mathcal D'(\R^2)$. Therefore by Lemma~\ref{lem:bd_uniqueness}, we must have $v_0 = 0$, contradicting $\|v_0\|_{H^s} = 1$. Therefore, we have $v_{\omega_j}$ is uniformly bounded in $H^s(\partial \Omega; T^* \partial \Omega)$ for every $s > 0$. 

    \noindent
    3. Finally, we prove convergence. By step 2 and the compactness of the embedding $H^{s'}(\partial \Omega; T^* \partial \Omega) \hookrightarrow H^{s} (\partial \Omega; T^* \partial \Omega)$ for any $s' > s$, we see that $v_{\omega_j}$ is precompact in $H^s(\partial \Omega; T^* \partial \Omega)$ for any $s \in \R$. Therefore, if $v$ is an $H^s$ limit point of $v_{\omega_j}$, we must have
    \[v \in C^\infty(\partial \Omega; T^* \partial \Omega).\] 
    Then by (\ref{eq:Cv=Rf}) and (\ref{eq:Cv_to_Rf}), $v$ must satisfy 
    \[\mathcal C_{\lambda_0 + i0} v = (R_{\lambda_0 + i0} f)|_{\partial \Omega}, \qquad \supp(R_{\lambda_0 + i0}(f - \mathcal I v)) \subset \overline \Omega.\]
    Therefore by the uniqueness property in Lemma~\ref{lem:bd_uniqueness}, the $H^s$ limit of every subsequence of $v_{\omega_j}$ must be the same, so $v_{\omega_j} \to v_{\lambda_0 + i0}$ in $H^s$ for every $s > 0$.
\end{proof}

What we ultimately need to obtain the spectral estimates is to be able to Taylor expand near $\lambda_0$. For this, we first need smooth dependence of the Neumann data on $\omega$ in the region $\Lambda$.  
\begin{proposition}\label{prop:bd_deriv_limit}
    Let $v_{\omega} := \mathcal N_\omega u_\omega$ where $u_\omega$ are the solutions to~\eqref{eq:BVP_reprise}. Let $\omega_j \in \Lambda \setminus \{\lambda_0\}$ with $\omega_j \to \lambda_0$. Then there exists $v_{\lambda_0 + i0}^{k} \in C^\infty(\partial \Omega, T^* \partial \Omega)$ such that
    \[\partial_{\omega}^k v_{\omega_j} \to v_{\lambda_0 + i0}^{(k)} \quad \text{in} \quad C^\infty(\partial\Omega; T^* \partial \Omega).\]
\end{proposition}
In fact, it will be clear from the proof below that $v_{\lambda_0 + i0}^{(k)}$ is the unique solution to
\begin{equation*}
\begin{gathered}
\mathcal C_{\lambda_0+i0} v_{\lambda_0+i0}^{(k)}=(\partial_\lambda^k R_{\lambda_0+i0}f)|_{\partial\Omega}-\sum_{\ell=0}^{k-1}
\binom k\ell (\partial_\lambda^{k-\ell}\mathcal C_{\lambda_0+i0})v_{\lambda_0+i0}^{(\ell)},\\
\supp\bigg(R_{\lambda_0+i0}\mathcal Iv^{(k)}_{\lambda_0+i0}-\partial^k_{\lambda}R_{\lambda_0+i0}f
+\sum_{\ell=0}^{k-1}\binom{k}{\ell}(\partial^{k-\ell}_{\lambda}R_{\lambda_0+i0})\mathcal I v^{(\ell)}_{\lambda_0+i0}\bigg)\subset\overline\Omega.
\end{gathered}
\end{equation*}

\begin{proof}
    1. The idea of the proof is similar to Proposition \ref{prop:v_LAP}. We first show that $\partial^k_\omega v_{\omega_j}$ is uniformly bounded in $C^\infty(\partial \Omega; T^* \partial \Omega)$ for $\omega \in \Lambda \setminus \{\lambda_0\}$. Differentiating the identity (\ref{eq:Cv=Rf}) $k$ times in $\omega$, we find 
    \begin{gather}
        \mathcal C_\omega \partial^k_\omega v_\omega = \partial_\omega^k F_\omega - \sum_{\ell = 0}^{k - 1} \binom{k}{\ell} (\partial_\omega^{k - \ell} \mathcal C_\omega)(\partial_\omega^\ell v_\omega), \label{eq:Cdv}\\
        R_\omega \mathcal I \partial_\omega^k v_\omega = \partial_\omega^k R_\omega f - \indic_\Omega \partial_\omega^k u_\omega - \sum_{\ell = 0}^{k - 1} \binom{k}{\ell} (\partial_\omega^{k - \ell} R_\omega) \mathcal I (\partial_\omega^\ell v_\omega), \label{eq:RId^kv}
    \end{gather}
    where we put $F_\omega:= (R_\omega f)|_{\partial \Omega}$. By Lemma~\ref{lem:fs_holomorphic}, we see that
    \begin{equation}\label{eq:dF_bound}
        \partial_\omega^k F_\omega \quad \text{is uniformly bounded in $C^\infty(\partial \Omega)$ for $\omega \in \Lambda$}.
    \end{equation}
    We proceed by induction on $k$. Assume that for all $\ell < k$, $\partial^\ell_\omega v_\omega$ is uniformly bounded for $\omega \in \Lambda\setminus \{\lambda_0\}$. Now assume by way of contradiction that $\|\partial_\omega^k v_{\omega_j}\|_{H^s} \to \infty$ for some $s > 0$. Define the rescalings
    \begin{equation}\label{eq:rescalings2}
        v_j = \frac{\partial_\omega^k v_{\omega_j}}{\|\partial_\omega^k v_{\omega_j}\|_{H^s}} \quad \text{and} \quad u_j = \frac{\partial_\omega^k u_{\omega_j}}{\|\partial_\omega^k v_{\omega_j}\|_{H^s}}.
    \end{equation}
    Since $H^s(\partial \Omega; T^* \partial \Omega) \hookrightarrow H^{-s}(\partial \Omega; T^* \partial \Omega)$ compactly, we may assume that $v_j \to v_0$ in $H^{-s}$ upon passing to a subsequence. The right hand side of (\ref{eq:Cdv}) is uniformly bounded in $C^\infty(\partial \Omega)$ by (\ref{eq:dF_bound}) and the induction hypothesis together with Proposition~\ref{prop:C_converge}. This means that 
    \begin{equation}
        \mathcal C_{\omega_j} v_j \to 0 \quad \text{in} \quad C^\infty(\partial \Omega).
    \end{equation}
    Then it follows from the high frequency estimate Corollary \ref{prop:high_freq_est} that 
    \begin{equation}\label{eq:vj_to_v0}
        v_j \to v_0 \quad \text{in} \quad H^s(\partial \Omega; T^* \partial \Omega).
    \end{equation}
    By Proposition \ref{prop:C_converge}, we see that $\mathcal C_{\omega_j} v_j \to \mathcal C_{\lambda_0 + i0} v_0$ in $\mathcal D'(\partial \Omega)$. Therefore
    \[\mathcal C_{\lambda_0 + i0} v_0 = 0.\]
    Furthermore, it follows from the induction hypothesis that 
    \[R_{\lambda_0 + i0} \mathcal I v_0 = -\lim_{j \to \infty} \indic_\omega u_j \implies \supp (R_{\lambda_0 + i0} \mathcal I v_0) \subset \overline \Omega.\]
    By uniqueness in Lemma~\ref{lem:bd_uniqueness}, $v_0 = 0$, contradicting (\ref{eq:rescalings2}) and (\ref{eq:vj_to_v0}). 

    \noindent
    2. Now we prove convergence. This is nearly identical to Step 3 of Proposition \ref{prop:v_LAP} except with the addition of an inductive step. Assume that $\partial_\omega^\ell v_{\omega_j} \to v_{\lambda_0 + i0}^{(\ell)}$ for all $\ell < k$. By Step 1 and Sobolev embeddings, we see that $\partial^k_\omega v_{\omega_j}$ is precompact in $H^s(\partial \Omega; T^* \partial \Omega)$ for any $s \in \R$. Every convergent subsequence in $H^s$ must converge to an element of $C^\infty$ by the same argument as in Proposition \ref{prop:v_LAP}, i.e. we have $\partial^k_\omega v_{\omega_{j_n}} \xrightarrow{H^s} v \in C^\infty(\partial \Omega; T^* \partial \Omega)$. Suppose another subsequence of $v_{\omega_j}$ converges in $H^s$ to $\tilde v \in C^\infty(\partial \Omega; T^* \partial \Omega)$. Then by the induction hypothesis and (\ref{eq:RId^kv}), 
    \[R_{\lambda_0 + i0} \mathcal I (v - \tilde v) = \indic_\Omega (v - \tilde v) \implies \supp(R_{\lambda_0 + i0} \mathcal I (v - \tilde v)) \subset \overline \Omega. \]
    Therefore by Lemma~\ref{lem:bd_uniqueness}, $\partial_\omega^k v_{\omega_j}$ converges in $H^s$ for every $s$. Since the limit is identical for any choice of $\omega_j \to \lambda_0$, $\omega_j \in \Lambda \setminus \{\lambda_0\}$, so we can denote the limit by put $\partial_\omega^k v_{\omega_j} \to v_{\lambda_0 + i0}^{(k)}$ in $C^\infty(\partial \Omega; T^* \partial \Omega)$.
\end{proof}

\Remark It is clear that Propositions \ref{prop:v_LAP} and \ref{prop:bd_deriv_limit} hold with all the $+$'s replaced with $-$'s. This corresponds to the limits from polynomial regions of the lower half plane as opposed to the upper half plane. We will see below that this makes no difference since we have uniqueness of the limiting problem.

\subsection{Interior limiting absorption} 
Using (\ref{eq:boundary_to_interior}), we can now deduce the the limiting behavior of the $P(\omega)^{-1}$ using what we have proved about the limiting behavior of $v_\omega$. 

\begin{proposition}\label{prop:u_LAP}
    Assume that $\Omega$ is $\lambda_0$-simple and $\mathbf r(\lambda_0)$ is Diophantine. Let $\Lambda^{\pm}_{\lambda_0, d}$ be the polynomial region defined in (\ref{eq:poly_region}) and let $f \in \CIc(\Omega)$. Then there exists a unique solution $u_{\lambda_0} \in C^\infty(\overline\Omega)$ to
    \begin{equation}\label{eq:BVP_real}
        P(\lambda_0) u_{\lambda_0} = f, \qquad u_{\lambda_0}|_{\partial\Omega} = 0. 
    \end{equation}
    Furthermore, there exists $u_{\lambda_0}^{(k)} \in C^\infty(\overline \Omega)$ such that
    for any sequence $\omega_j \to \lambda_0$ with $\omega_j \in (\Lambda^+_{\lambda_0, d} \cup \Lambda^-_{\lambda_0, d}) \setminus\{\lambda_0\}$ for all $j$, 
    \begin{equation}
        \partial_\omega^k u_{\omega_j} \to u_{\lambda_0}^{(k)} \quad \text{in} \quad \text{$C^\infty(\overline \Omega)$ for all $k \in \N_0$,}
    \end{equation}
    where $u_\omega$ are the unique solutions to the boundary value problem (\ref{eq:BVP_reprise}). 
\end{proposition}
\Remark An important feature of Proposition~\ref{prop:u_LAP} is that the limits
$u^{(k)}_{\lambda_0}$ are the same when approaching $\lambda_0$ from above and from below in the complex plane, that is $u^{(k)}_{\lambda_0+i0}=u^{(k)}_{\lambda_0-i0}$. This is in contrast with the
usual limiting absorption principle in scattering theory. See \cite{DZ_resonances} for examples, particularly the remarks after \cite[Definition 2.3]{DZ_resonances} for the simplest case of one-dimensional potential scattering). It also contrasts the limiting absorption principles in the non-ergodic internal waves situations studied in \cite[\S3.3]{Dyatlov_Zworski_19} and \cite[\S7]{DWZ}.
\begin{proof}
    Differentiating (\ref{eq:boundary_to_interior}) $k$ times in $\omega$, we find
    \begin{equation}\label{eq:u=Rf-Sv_derivative}
        \partial_\omega^k u_\omega = \partial_\omega^k(R_\omega f)|_\Omega - \sum_{\ell = 0}^k \binom{k}{\ell}(\partial_\omega^{k - \ell} S_\omega)(\partial_\omega^\ell v_\omega)
    \end{equation}
    where
    \[v_\omega := \mathcal N_\omega u_\omega.\]
    Assume that $\omega_j \in \Lambda^\pm_{\lambda_0, d} \setminus\{\lambda_0\}$ for all $j$. By Proposition~\ref{prop:bd_deriv_limit}, Proposition~\ref{prop:S_converge}, and Lemma~\ref{lem:fs_holomorphic}, 
    the right-hand side of (\ref{eq:u=Rf-Sv_derivative}) converges in $C^\infty(\overline \Omega)$. When $k = 0$, we see that 
    \[u_\omega \to (R_{\lambda_0 + i0} f)|_{\Omega} - S_{\lambda_0 \pm i0} v_{\lambda_0 \pm i0} =: u_{\lambda_0 \pm i0 } \quad \text{in} \quad C^\infty(\overline \Omega),\]
    so $u_{\lambda_0 \pm i0}$ solves the boundary value problem (\ref{eq:BVP_real}). By Lemma~\ref{lem:interior_uniqueness}, we can then define
    \[u_{\lambda_0} := u_{\lambda_0 + i0} = u_{\lambda_0 - i0}.\]
    Now for $k \ge 1$, it follows from (\ref{eq:u=Rf-Sv_derivative}) that there exists $u_{\lambda_0 \pm i0}^{(k)} \in C^\infty(\overline \Omega)$ such that
    \begin{equation*}
        \partial_\omega^k u_{\omega_j} \to u_{\lambda_0 \pm i0}^{(k)} \quad \text{in $C^\infty (\overline \Omega)$ if $\omega_j \in \Lambda^\pm_{\lambda_0, d} \setminus\{\lambda_0\}$ for all $j$}.
    \end{equation*}
    Thus it remains to check that $ u_{\lambda_0 + i0}^{(k)} =  u_{\lambda_0 - i0}^{(k)}$. We proceed by induction on $k$ and assume that $ u_{\lambda_0 \pm i0}^{(\ell)} =  u_{\lambda_0 \pm i0}^{(\ell)}$ for all $\ell < k$. Differentiating (\ref{eq:BVP_reprise}) $k$ times in $\omega$, $k \ge 1$, we see that
    \begin{equation}\label{eq:P_omega_derivatives}
    \begin{gathered}
        P(\omega_j) (\partial_\omega^k u_{\omega_j}) = \sum_{\ell = 1}^{k} \binom{k}{\ell} \left[\partial_\omega^\ell a_1(\omega_j) \partial_{x_1}^2 - \partial_\omega^{\ell} a_2(\omega_j) \partial_{x_2}^2 \right] (\partial_\omega^{k - \ell} u_{\omega_j}), \\
        a_1(\omega) := \omega^2, \quad a_2(\omega) := 1 - \omega^2.
    \end{gathered}
    \end{equation}
    The differential operators on the right-hand side are constant-coefficient with coefficients $\partial_\omega^\ell a_1(\omega), \partial_\omega^\ell a_2(\omega)$ that converge as $\omega_j \to \lambda_0$. Also, $\partial_\omega^{k - \ell} u_{\omega_j}$ converges in $C^\infty(\overline \Omega)$. Then by the induction hypothesis, the right hand side of (\ref{eq:P_omega_derivatives}) converges to a unique limit in $C^\infty(\Omega)$ independent of if the sequence $\omega_j$ is in $\Lambda^+_{\lambda_0, d} \setminus\{\lambda_0\}$ or $\Lambda^-_{\lambda_0, d} \setminus\{\lambda_0\}$. Therefore, Lemma~\ref{lem:interior_uniqueness} gives that the limit $u_{\lambda_0 \pm i0}^{(k)}$ is independent of the sign. Set
    \[u_{\lambda_0}^{(k)} := u_{\lambda_0 + i0}^{(k)} = u_{\lambda_0 - i0}^{(k)},\]
    and the proposition follows. 
\end{proof}

\section{Evolution problem}\label{sec:evolution}
Finally, we relate all the analysis of the limiting properties of $P(\omega)$ back to the evolution problem. Throughout this section, fix a $\lambda_0 \in (0, 1)$ so that $\Omega$ is $\lambda_0$-simple (Definition \ref{lambda_simple_def}), and the rotation number of the associated chess billiard map $\mathbf r(\lambda_0)$ is Diophantine (Definition \ref{def:diophantine}). Recall that analyzing the internal waves equation (\ref{eq:internal_waves}) is equivalent to the evolution problem (\ref{eq:evolution_problem}):
\begin{equation}\label{eq:evolution_problem_reprise}
    (\partial_t^2 + P)w = f \cos \lambda_0 t, \quad w|_{t = 0} = \partial_t w_{t = 0} = 0, \quad f \in \CIc(\Omega; \R),
\end{equation}
where $P = \partial_{x_2}^2 \Delta_\Omega^{-1} :H^{-1}(\Omega) \to H^{-1}(\Omega)$.

\subsection{Spectral measure near \texorpdfstring{$\lambda_0^2$}{}}
The idea is that we will be recovering the spectral measure estimates near $\lambda_0$ from the Helffer--Sj\"ostrand formula. We begin by reviewing some necessary preliminaries following \cite[\S2.2]{davies_95}.

Consider $g \in \CIc(\R)$. We can construct an \textit{almost analytic extension} as follows: fix a cutoff function 
\begin{equation*}
    \tau \in \CIc(\R; [0, 1]), \quad \supp \tau \subset (-2, 2), \quad \tau \equiv 1 \quad \text{on} \quad [-1, 1].
\end{equation*}
Put
\begin{equation}
    \psi(x, y) := \tau(y/\langle x \rangle).
\end{equation}
Then we can extend $g$ to a function $\tilde g \in \CIc(\C)$ by 
\begin{equation}\label{eq:aa_def}
    \tilde g(z) := \psi(x, y) \sum_{r = 0}^1 \partial_{x}^r g(x) (iy)^r.
\end{equation}
The function $\tilde g$ is almost analytic in the sense that if we differentiate in $\bar z$, we see that 
\begin{equation}\label{eq:aa_bound}
    \begin{aligned}
        \partial_{\bar z} \tilde g(z) :=& \frac{1}{2} \big( \partial_x g(z) + i \partial_y g(z) \big) \\
        =& \frac{1}{2} (\partial_x \psi(x, y) + i \partial_y \psi(x, y)) \sum_{r = 0}^1 \partial_x^r g(x) (i y)^r \\
        &\qquad + \frac{1}{2} \partial_x^2 g(x) (iy) \psi(x, y). 
    \end{aligned}
\end{equation}
for any $N \in \N$, and it follows that 
\[|\partial_{\bar z} \tilde g(z)| = \mathcal O(|y|).\]
For any self-adjoint operator $H$ over a Hilbert space $\mathscr H$, the Helffer--Sj\"ostrand formula gives 
\begin{equation}\label{eq:HS_formula}
    g(H) = -\frac{1}{\pi} \int_{\C} \frac{\partial \tilde g(z)}{\partial \bar z} (z - H)^{-1} \, dm(z).
\end{equation}
where $dm(z) := dx dy$ denotes the Lebesgue measure on $\C$. The formula~\eqref{eq:HS_formula} is independent of the choice of cutoff $\psi$. We remark that since $\|(z - H)^{-1}\|_{\mathscr H \to \mathscr H} \le |\Im z|^{-1}$, the convergence of the integral is then guaranteed by~(\ref{eq:aa_bound}). The integral is also independent of the choice of $\varphi$, and in fact the functional calculus of $H$ can be recovered using the formula (see \cite{davies_95} for details). Now we are ready to give a proof of Theorem \ref{thm:spectral}.

\begin{proof}[Proof of Theorem \ref{thm:spectral}] 1. Fix a cutoff 
\begin{equation*}
    \chi \in \CIc(\R; [0, 1]), \quad \supp \chi \subset (-1, 1), \quad \chi \equiv 1 \quad \text{near $0$}.
\end{equation*}
Define the rescaling
\begin{equation*}
    \chi_\epsilon(x) = \chi \left(\frac{x - \lambda_0^2}{\epsilon} \right).
\end{equation*}
Using~\eqref{eq:aa_def}, we can construct the almost analytic extension $\tilde \chi_\epsilon \in \CIc(\C)$ of $\chi_\epsilon$. It follows from (\ref{eq:aa_bound}) that for all sufficiently small $\epsilon$, $\tilde \chi_\epsilon$ satisfies the bound
\begin{equation}\label{eq:chi_aa_bound}
    \left|\frac{\partial \tilde \chi_\epsilon}{\partial \bar z} \right| \le C\left( \frac{|\Im z|}{\epsilon^{2}} \right).
\end{equation}\label{eq:thm1_est}
Since $\mu_{f, f}$ is a positive measure, it suffices to show that 
\begin{equation}
    \int_\R \chi_\epsilon \, d\mu_{f, f} \le C \epsilon^N
\end{equation}
for all $N \in \N$.

\begin{figure}
    \centering
    \includegraphics[scale = .175]{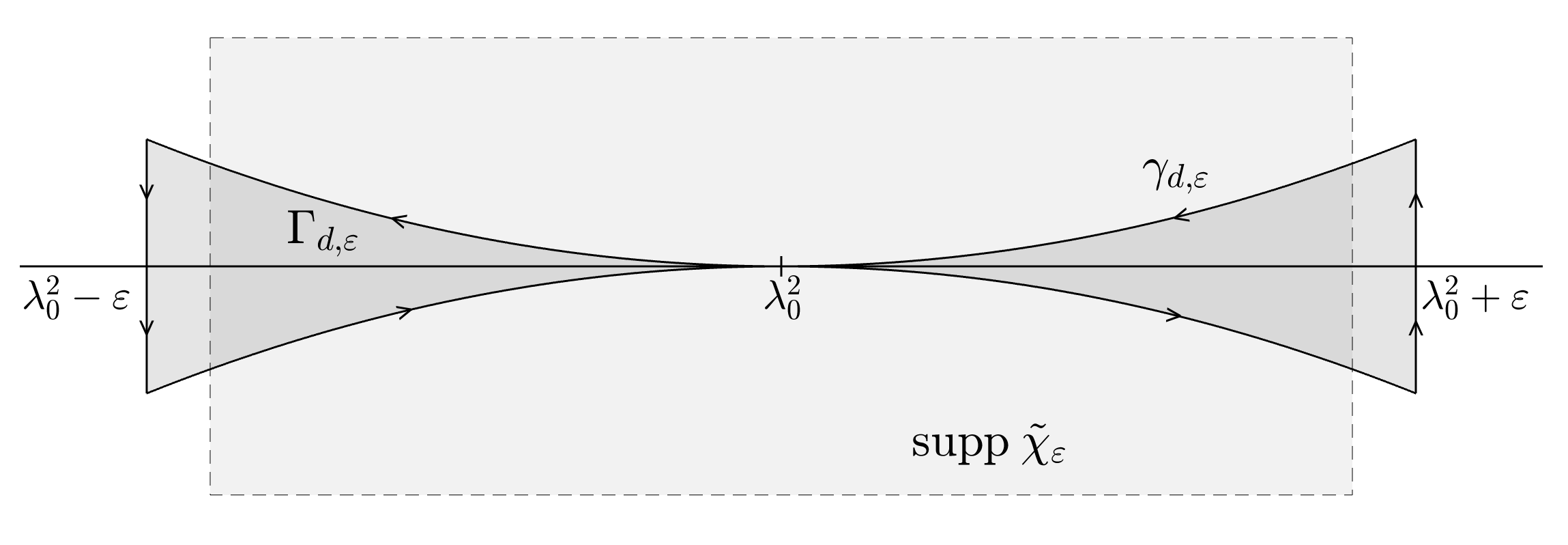}
    \caption{The contour taken in the proof of Theorem \ref{thm:spectral} relative to the support of $\tilde \chi_\epsilon$.}
    \label{fig:contour}
\end{figure}

\noindent
2. For an even $d > 0$, define the region 
\begin{equation}
    \Gamma_{d, \epsilon} := \{\lambda_0^2 + s + ih \mid |h|< s^d, \, |s| < \epsilon\},
\end{equation}
and let $\gamma_{d, \epsilon} := \partial \Gamma_{d, \epsilon}$ denote the positively oriented smooth contour that encloses $\Gamma_{d, \epsilon}$. Recall from (\ref{eq:resolvent_relation}) that
\[(P - z)\Delta_\Omega = P(\omega), \quad z := \omega^2.\]
Let $u_\omega$ denote the unique solution to~\eqref{eq:EBVP}. Therefore, with $f \in \CIc(\Omega)$, it follows from Proposition \ref{prop:u_LAP} that
\[z \mapsto \langle (z - P)^{-1}f, f\rangle_{H^{-1}(\Omega)} = \langle \Delta u_\omega, f \rangle_{H^{-1}(\Omega)}, \quad z \in \supp \tilde \chi_\epsilon \setminus \overline \Gamma_{d, \epsilon}\]
extends to a smooth function up to the boundary $\{\lambda_0^2 + s \pm i s^d \mid |s| < \epsilon\}$. By the Helffer--Sj\"ostrand formula (\ref{eq:HS_formula}), we then have
\begin{align}
    \langle \chi_\epsilon(P)f, f \rangle &= - \frac{1}{\pi} \left(\int_{\Gamma_{d, \epsilon}} + \int_{\C \setminus \Gamma_{d, \epsilon}} \right) \frac{\partial \tilde \chi_\epsilon(z)}{\partial \bar z} \langle (z - P)^{-1}f, f \rangle \, dm(z) \nonumber \\
    &= - \frac{1}{\pi} \int_{\Gamma_{d, \epsilon}} \frac{\partial \tilde \chi_\epsilon(z)}{\partial \bar z} \langle (z - P)^{-1}f, f \rangle \, dm(z) \nonumber \\
    & \qquad \qquad + \frac{1}{2 \pi i} \oint_{\gamma_{d, \epsilon}} \tilde \chi_\epsilon(z) \langle (z - P)^{-1}f, f \rangle \, dz. \label{eq:HS_modified}
\end{align}
The second equality above is due to the the fact that $\tilde \chi_\epsilon$ is compactly supported, so we can just apply Green's theorem in the region $\C \setminus \Gamma_{d, \epsilon}$. 

\noindent
3. We first bound the the first term on the right hand side of \eqref{eq:HS_modified}. It follows from~\eqref{eq:chi_aa_bound} and the spectral bound $\|(z - P)^{-1}\|_{H^{-1}(\Omega)\to H^{-1}(\Omega)} \le |\Im z|^{-1}$ that 
\begin{align}
    \left|\int_{\Gamma_{d, \epsilon}} \frac{\partial \tilde \chi_\epsilon(z)}{\partial \bar z} \langle (z - P)^{-1}f, f \rangle \, dm(z) \right| &\le C_d \int_{-\epsilon}^\epsilon \int_{-|s|^d}^{|s|^d} \left(\frac{|y|}{\epsilon^{2}} \right) \frac{1}{|y|} \, dy ds \nonumber \\
    &\le C_d \epsilon^{d - 1} \label{eq:outside}
\end{align}
where $C_d$ may change from line to line, but remains independent of $\epsilon$.

\noindent
4. It just remains to bound the second term of (\ref{eq:HS_modified}). Define the functions 
\begin{equation}
    I^\pm(s) := (1 \pm id \cdot x^{d - 1}) \big\langle (\lambda_0^2 + s \pm is^d - P)^{-1} f, f \big\rangle_{H^{-1}(\Omega)}
\end{equation}
for $s \in \R \setminus \{0\}$. Then it follows from Proposition \ref{prop:u_LAP} that $I^\pm$ extends to a smooth function on all of $\R$ with
\begin{equation}\label{eq:I_cancellation}
    \partial_s^k I^+(0) = \partial_s^k I^-(0) \quad \text{for} \quad k \le d - 2.
\end{equation}
Note that $\tilde \chi_\epsilon(z)$ vanishes on the two vertical segments of $\gamma_{d, \epsilon}$, therefore
\begin{multline}\label{eq:contour_integral}
    \oint_{\gamma_{d, \epsilon}} \tilde \chi_\epsilon(z) \big \langle (z - P)^{-1}f, f \big \rangle_{H^{-1}(\Omega)} \, dz \\
    = \int_{-\epsilon}^\epsilon \tilde \chi_\epsilon(\lambda_0^2 + s - is^d) I^-_\epsilon(s) - \tilde \chi_\epsilon(\lambda_0^2 + s + is^d) I^{+}_\epsilon(s)\, ds.
\end{multline}
Expanding the integrand using~\eqref{eq:aa_def}, we see that for sufficiently small $\epsilon$, 
\begin{multline}\label{eq:contour_integrand}
    \tilde \chi_\epsilon(\lambda_0^2 + s - is^d) I^-_\epsilon(s) - \tilde \chi_\epsilon(\lambda_0^2 + s + is^d) I^{+}_\epsilon(s) \\
    = \chi_\epsilon(\lambda_0^2 + s) \big(I^-_\epsilon(s) - I^+_\epsilon(s) \big) - is^d \chi'_\epsilon(\lambda^2 + s) \big(I^-_\epsilon(s) + I^+_\epsilon(s) \big).
\end{multline}
Combining~\eqref{eq:contour_integral} and~\eqref{eq:contour_integrand}, we find
\begin{equation}\label{eq:inside}
    \left|\oint_{\gamma_{d, \epsilon}} \tilde \chi_\epsilon(z) \big \langle (z - P)^{-1}f, f \big \rangle_{H^{-1}(\Omega)} \, dz \right| \le C_d \int_{-\epsilon}^\epsilon |s|^{d - 2} + \epsilon^{-1} |s|^d\, ds \le C_d\epsilon^{d - 1}.
\end{equation}
Since $d$ can be taken to be arbitrarily large, the estimates~\eqref{eq:outside} and~\eqref{eq:inside} immediately implies~\eqref{eq:thm1_est}.
\end{proof}

\subsection{Boundedness in energy}
Recall that we can solve the evolution problem (\ref{eq:evolution_problem_reprise}) using the functional calculus of $P$ to find the solution
\begin{equation}
    w(t) = \Re\big(e^{i \lambda_0 t} \mathbf W_{t, \lambda_0}(P) f \big )
\end{equation}
where
\begin{equation}
    \mathbf W_{t, \lambda_0}(z) = \sum_{\pm} \frac{1 - e^{-it(\lambda_0 \pm \sqrt{z})}}{2 \sqrt{z}(\sqrt{z} \pm \lambda_0)}.
\end{equation}
We wish to show that $w(t)$ remains uniformly bounded in $H^{-1}(\Omega)$. First, let us divide $\mathbf W_{t, \lambda_0}$ into more manageable pieces. Let $\delta > 0$ be a small parameter to be chosen later. Define
\begin{equation}
    w_k(t) = \indic_{\{\delta^{k + 1} \le |\lambda_0^2 - z| < \delta^k\}}(P) \mathbf W_{t, \lambda_0}(P) f, \qquad k \in \N_0. 
\end{equation}
The spectrum of $P$ is supported in $[0, 1]$, and also by Theorem \ref{thm:spectral}, $\mu_{f, f}(\{\lambda_0^2\}) = 0$. Hence 
\begin{equation*}
    w(t) = \Re \left( e^{i \lambda_0 t} \sum_{k = 0}^\infty w_k(t) \right)
\end{equation*}
where the sum converges in $H^{-1}(\Omega)$.

\begin{proof}[Proof of Theorem \ref{thm:evolution}]
    It suffices to prove that $\|w(t)\|_{H^{-1}} < C$ for all $t \ge 0$ since the solution $u(t)$ to (\ref{eq:internal_waves}) is given by $u(t) = \Delta_\Omega^{-1} w(t)$. We first bound $\Re(e^{i \lambda_0 t} w_0(t))$. We compute 
    \[\left|\Re\big(e^{i \lambda_0 t} \mathbf W_{t, \lambda_0}(z) \indic_{\{|\lambda_0^2 - z| \ge \delta\}} \big) \right| = \left|\frac{(\cos (t \sqrt{z}) - \cos(t \lambda_0))}{\lambda_0^2 - z} \indic_{\{|\lambda_0^2 - z| \ge \delta\}}\right| \le \frac{2}{\delta}. \]
    Then it follows from the functional calculus of $P$ on $H^{-1}(\Omega)$ that 
    \begin{equation}
        \Re\big(e^{i \lambda_0 t} w_0(t)\big) \le \frac{2}{\delta} \|f\|_{H^{-1}(\Omega)}.
    \end{equation}

    Now we bound the rest of $w(t)$. Note that we can fix a sufficiently small $\delta$ so that
    \begin{equation*}
        \sup_{\delta^{k + 1} \le |\lambda_0^2 - z| \le \delta^k} |\mathbf W_{t, \lambda_0}(z)| < \frac{2}{\delta^{k + 1}}.
    \end{equation*}
    Then by Theorem \ref{thm:spectral}, we see that with $f \in \CIc(\Omega)$, 
    \begin{align*}
        \|w_k(t)\|_{H^{-1}(\Omega)}^2 &= \big\langle \indic_{\{\delta^{k + 1} \le |\lambda_0^2 - z| < \delta^k\}}(P) \mathbf W_{t, \lambda_0}(P)^* \mathbf W_{t, \lambda_0}(P) f, f \big\rangle \\
        &= \int_{\delta^{k + 1} \le |\lambda_0^2 - z| \le \delta^k} |\mathbf W_{t, \lambda_0}(z)|^2 \, d\mu_{f, f}(z) \\
        &\le C_N \delta^{Nk - (k + 1)},
    \end{align*}
    for any $N \in \N$ for some constant $C$ independent of $t$ and $k$. It really suffices to take $N = 3$ so that the sum over $k$ converges, and we see that $\|w(t)\|_{H^{-1}(\Omega)} \leq C$ as desired for a possibly different~$C$. 
\end{proof}

\medskip\noindent\textbf{Acknowledgements.}
The author would like to thank Semyon Dyatlov for the countless insightful discussions about the project and for all the detailed comments on the early draft. The author would also like to thank Maciej Zworski for his generosity in lending the author code for numerical simulations of the internal waves PDE, Leo Maas for insightful discussions about the physics, and Richard Melrose, whose suggestions greatly simplified the positive semiclassical calculus in \S\ref{sec:microlocal}. Partial support from Semyon Dyatlov's NSF CAREER grant DMS-1749858 is acknowledged.

\bibliographystyle{alpha}
\bibliography{internal_waves.bib}

\end{document}